\documentclass[11pt]{article} % use larger type; default would be 10pt

\usepackage[utf8]{inputenc} % set input encoding (not needed with XeLaTeX)

\usepackage{geometry} % to change the page dimensions
\usepackage[hidelinks]{hyperref}
\usepackage{graphicx} % support the \includegraphics command and options
\usepackage{amsfonts,amsthm,amsmath,amssymb,esint}
\usepackage{slashed}
\usepackage{authblk}
\usepackage{tikz}
\usetikzlibrary{arrows}
\usepackage{stmaryrd}
\newcommand{\R}{\mathbb{R}}
\newcommand{\C}{\mathbb{C}}
\newcommand{\Z}{\mathbb{Z}}

\newcommand{\N}{\mathbb{N}}
\newcommand{\G}{\mathbb{G}}
\newcommand{\F}{\mathbb{F}}
\newcommand{\Pb}{\mathbb{P}}
\newcommand{\B}{\mathbb{B}}
\newcommand{\M}{\mathbb{M}}
\newcommand{\Sb}{\mathbb{S}}

\newcommand{\Hb}{\mathbb{H}}
\newcommand{\K}{\mathbb{K}}
\newcommand{\ket}{\rangle}
\newcommand{\bra}{\langle}
\newcommand{\T}{\mathbb{T}}
\newcommand{\Li}{\mathcal{L}}

\newcommand{\Ai}{\mathcal{A}}
\newcommand{\Bi}{\mathcal{B}}

\newcommand{\Hi}{\mathcal{H}}

\newcommand{\Mi}{\mathcal{M}}
\newcommand{\Ki}{\mathcal{K}}
\newcommand{\Ri}{\mathcal{R}}
\newcommand{\Di}{\mathcal{D}}
\newcommand{\Ei}{\mathcal{E}}
\newcommand{\Fi}{\mathcal{F}}
\newcommand{\Oi}{\mathcal{O}}

\newcommand{\Ti}{\mathcal{T}}
\newcommand{\Ii}{\mathcal{I}}

\newcommand{\GH}{\mathfrak{H}}

\newcommand{\GK}{\mathfrak{K}}

\newcommand{\id}{\operatorname{id}}

\newcommand{\ep}{\varepsilon}
\newcommand{\pd}{\partial}
\newcommand{\Tr}{\operatorname{Tr}}
\newcommand{\End}{\operatorname{End}}
\newcommand{\Hom}{\operatorname{Hom}}

\newcommand{\longto}{\longrightarrow}
\newcommand{\SU}{\operatorname{SU}}
\newcommand{\Un}{\operatorname{U}}
\newcommand{\On}{\operatorname{O}}
\newcommand{\Irrep}{\operatorname{Irrep}}
\newcommand{\Mn}{\operatorname{M}}
\newcommand{\GeL}{\operatorname{GL}}

\newcommand{\gr}{\operatorname{gr}}
\newcommand{\qgr}{\operatorname{qgr}}
\newcommand{\coh}{\operatorname{coh}}

\newtheorem{thm}{Theorem}[section]
\newtheorem{Lemma}[thm]{Lemma}
\newtheorem{prop}[thm]{Proposition}
\newtheorem{cor}[thm]{Corollary}

\numberwithin{equation}{section}

\theoremstyle{definition}
\newtheorem{dfn}[thm]{Definition}
\newtheorem{Notation}[thm]{Notation}
\newtheorem{Ass}[thm]{Assumption}
\newtheorem{Remark}[thm]{Remark}
\newtheorem{Example}[thm]{Example}

\newtheorem{Conjecture}[thm]{Conjecture}
\newtheorem{Problem}[thm]{Problem}

\newtheorem*{sop}{Sketch of proof}
\newtheorem*{Ack}{Acknowledgment}

\newcommand{\Ad}{\operatorname{Ad}}
\newcommand{\Gr}{\operatorname{Gr}}
\newcommand{\bone}{\mathbf{1}}

\newcommand{\lef}{\operatorname{left}}
\newcommand{\rig}{\operatorname{right}}

\newcommand{\FS}{\operatorname{FS}}
\newcommand{\vol}{\operatorname{vol}}

\usepackage{booktabs} % for much better looking tables
\usepackage{array} % for better arrays (eg matrices) in maths
\usepackage{paralist} % very flexible & customisable lists (eg. enumerate/itemize, etc.)
\usepackage{verbatim} % adds environment for commenting out blocks of text & for better verbatim
\usepackage{subfig} % make it possible to include more than one captioned figure/table in a single float
\usepackage{fancyhdr} % This should be set AFTER setting up the page geometry
\evensidemargin=\oddsidemargin
\usepackage{sectsty}
\allsectionsfont{\sffamily\mdseries\upshape} % (See the fntguide.pdf for font help)
\usepackage[nottoc,notlof,notlot]{tocbibind} % Put the bibliography in the ToC
\usepackage[titles,subfigure]{tocloft} % Alter the style of the Table of Contents

\title{Berezin quantization of noncommutative projective varieties}
\author{Andreas Andersson}
\affil{\small Email: fornjotnr@hotmail.com}
\affil{\footnotesize Max Planck Institute for Mathematics in the Sciences. Inselstrasse 22, D-04103 Leipzig, Germany
\\Wollongong University, School of Mathematics and Applied Statistics, 2522 Wollongong, Australia}
\affil{Mathematics Subject Classification 2010 Primary: 58B32; Secondary: 58B34, 46L89}
\affil{Keywords: Berezin quantization, Toeplitz operators, Cuntz--Pimsner algebras, compact quantum groups, generalized inductive limits, noncommutative geometry, subproduct systems, noncommutative random walks, Arveson's conjecture, balanced metrics.}

\begin{document}
\maketitle
\abstract
We use operator algebras and operator theory to obtain new result concerning Berezin quantization of compact Kähler manifolds. Our main tool is the notion of subproduct systems of finite-dimensional Hilbert spaces, which enables all involved objects, such as the Toeplitz operators, to be very conveniently expressed in terms of shift operators compressed to a subspace of full Fock space. This subspace is not required to be contained in the symmetric Fock space, so from finite-dimensional matrix algebras we can construct noncommutative manifolds with extra structure generalizing that of a projective variety endowed with a positive Hermitian line bundle and a Kähler metric in the class of the line bundle. Even in the commutative setting these constructions are very fruitful. Firstly, we show that the algebra of smooth functions on any smooth projective variety can be quantized in a strong sense of inductive limits, as was previously only accomplished for homogeneous manifolds. In this way the Kähler manifold is recovered exactly from quantization and not just approximately. Secondly, we obtain a strict quantization also for singular varieties. Thirdly, we show that the Arveson conjecture is true in full generality for shift operators compressed to the subspace of symmetric Fock space associated with any homogeneous ideal. For noncommutative examples we consider homogeneous spaces for compact matrix quantum groups which generalize $q$-deformed projective spaces, and we show that these can be obtained as the cores of Cuntz--Pimsner algebras constructed solely from the representation theory of the quantum group. We also discuss interesting connections with noncommutative random walks.

\tableofcontents 

\section{Introduction}

With motivations from physics, Berezin introduced a way of approximating projective compact Kähler manifolds $M$ by finite-dimensional matrix algebras $\Bi(\GH_m)$ parameterized by $m\in\N_0$ \cite{Bere1, Bere2, CGR1, Lan2, Schl1}. When $M=G/K$ is a homogeneous space of some Lie group $G$, each Hilbert space $\GH_m$ carries an irreducible representation of $G$, and it is obtained by the Borel--Weil construction: $\GH_m$ is the space of holomorphic sections of a suitable line bundle $L^{\otimes m}$ over $M$.  

In order to apply Berezin quantization to quantum physics, the parameters (time or temperature or energy etc.) should be chosen such that the limit $m\to\infty$ simplifies the description of the system at hand. This is a powerful method; for instance it gives the Hartree--Fock approximation as a special case \cite{Raj1}. The limit behavior is captured by a classical (compact Kähler) manifold $M$. 

Nowadays it is however becoming more and more important to have a versatile theory of open quantum systems. Recently we observed how to obtain a simplifying infinite-$m$ limit in the standard framework of quantum channels as driving the evolution of open quantum systems \cite{An4, An5}. It so happens though, that the infinite-$m$ system is (in general) not given by a classical manifold but by a \emph{noncommutative} manifold, i.e. there is a noncommutative $C^*$-algebra $C(\M)$ which is supposed to encode the properties of the dynamics and which is a surprisingly good analogue of the commutative $C^*$-algebra $C(M)$ of continuous functions on a projective variety $M$. Here we use the symbol $\M$ for a nonexisting object appearing only in the notation for the algebra $C(\M)$, while $M$ denotes an honest manifold. Such a $C^*$-algebra $C(\M)$ of continuous functions on a ``noncommutative projective variety" will be obtained as an inductive limit of matrix algebras $\Bi(\GH_m)$. These in turn arise from a sequence $\GH_\bullet=(\GH_m)_{m\in\N_0}$ of finite-dimensional Hilbert spaces such that $\GH_{m+l}\subseteq\GH_m\otimes\GH_l$ for all $m,l\in\M_0$. Such a sequence has been referred to as a ``subproduct system" \cite{ShSo1} and it generalizes the structure needed to perform ordinary Berezin quantization. In \cite{An4} we discussed the physical meaning of the special case when $\M=\G/\K$ is a ``projective" quantum homogeneous space (a generalization of compact coadjoint orbits).

%The ``noncommutative Kähler manifolds" $\M$ appearing in this way generalize only a special kind of compact Kähler manifolds, namely (complex) projective manifolds $M\subset\C\Pb^{n-1}$ (also singular varieties can occur though). 
The purpose of the present paper is to introduce and study the $C^*$-algebras $C(\M)$ defining these noncommutative manifolds. They will be constructed in such a way that they possess a lot of extra structure generalizing that complex-analytic structure, a positive line bundle and a Kähler metric. The construction is new also in the commutative setting and we will solve some very interesting problems in operator theory which provide new insight in the geometry of compact Kähler manifolds. In this way we set the stage for a new interplay between operator theory and complex differential geometry. We begin by outlining the results from several different viewpoints.

%We work with frame matrix representations of operators on Fock space. ...

%Very inspiring for the geometric interpretation of the present work was the paper \cite{LMS1}. 

\subsection{Kähler geometry}

Let $(M,L)$ be a polarized manifold, i.e. $M$ is a compact connected complex-analytic manifold admitting a Kähler metric and $L$ is a holomorphic line bundle over $M$ with the property that any choice of basis for the finite-dimensional vector space $H^0(M;L)$ of holomorphic sections of $L$ gives a holomorphic embedding of $M$ into the projective space $\C\Pb^{n-1}$, where $n:=\dim H^0(M;L)$. For $m\in\N$, let $L^m$ be the $m$th tensor power of the line bundle $L$ and let $L^0=\Oi_M$ be the trivial holomorphic line bundle. The idea of Berezin quantization is that if one chooses an inner product on $H^0(M;L^m)$ for each $m$, so that we obtain a sequence $\GH_\bullet=(\GH_m)_{m\in\N_0}$ of Hilbert spaces, then the finite-dimensional matrix algebras $\Bi(\GH_m)$ should give an increasingly good approximation of the $C^*$-algebra $C(M)$ of continuous functions on $M$ as $m$ goes to infinity. This works for any polarized manifold $(M,L)$ \cite{BMS} if the inner product on each $\GH_m$ is given by
$$
\bra\phi|\psi\ket_{h,\omega}:=m^d\int_Mh^m(\phi(x),\psi(x))\frac{\omega(x)^d}{d!},\qquad\forall \psi,\phi\in H^0(M,L^m)
$$
for some fixed Kähler metric $\omega$ on $M$ and a Hermitian metric $h$ on $L$ related to the Kähler metric via the ``prequantum condition''
\begin{equation}\label{pqcondintro}
\sqrt{-1}\bar{\pd}\pd\log h=\omega.
\end{equation}
The approximation of $C(M)$ by the matrix algebras $\Bi(\GH_m)$ is, more precisely, a ``strict quantization" of $C(M)$ realized by explicit surjective positive maps (``Toeplitz maps") (see \S\ref{quantcommsec} for more details)
$$
\breve{\varsigma}^{(m)}:C(M)\to\Bi(\GH_m)
$$
and explicit injective positive maps (covariant symbol maps) 
$$
\varsigma^{(m)}:C(M)\to\Bi(\GH_m).
$$
such that $\varsigma^{(m)}$ is the adjoint of $\breve{\varsigma}^{(m)}$ with respect to the inner product on $\Bi(\GH_m)$ and $C(M)$ defined by the trace and by $\omega$ respectively, and such that the ``Berezin transforms'' $\varsigma^{(m)}\circ\breve{\varsigma}^{(m)}$ converge to the identity map on $C(M)$ as the ``quantum number'' $m$ goes to infinity. 

One could imagine however, that it should be possible to approximate $C(M)$ in some stronger sense. It turns out that is not true for any polarized manifold $(M,L)$, if we want the prequantum condition \eqref{pqcondintro}. We are now interested in the problem of relating properties of $(M,L)$ to the well-behavedness of the quantization $\GH_\bullet$. As has been well known every since the beginnings of quantization theory, and made very clear in \cite{CGR1, CGR2}, if $(M,L)$ is a homogeneous manifold then it can be quantized in a stronger sense than that provided by a strict quantization. Hawkins made this very explicit in a $C^*$-algebraic setting \cite{Hawk1} by realizing $C(M)$ as an inductive limit of the matrix algebras $\Bi(\GH_m)$. We shall show that these inductive limits are arising very intrinsically from a choice of embedding $M\hookrightarrow\C\Pb^{n-1}$ and that the construction works for any projective variety $M$ if we drop the prequantum condition \eqref{pqcondintro} and instead choose $h$ and $\omega$ to be related in a different way. In fact we need not refer to Hermitian metrics or Kähler metrics for the construction of the inductive limit, and the covariant symbol $\varsigma^{(m)}$ and the Toeplitz maps $\breve{\varsigma}^{(m)}$ arise naturally from the very maps $\iota_{m,l}:\Bi(\GH_m)\to\Bi(\GH_l)$ which defines the inductive system. Nevertheless it is easy to recover the metrics $h$ and $\omega$ such that $\varsigma^{(m)}$ and $\breve{\varsigma}^{(m)}$ attains the same geometric meaning as before. 

Thus we drop the prequantum condition \eqref{pqcondintro} and instead look at the Hilbert spaces $\GH_m$ obtained from a sequence $(\omega,h_m)$ where the metric on $L^m$ varies with $m\in\N_0$, or equivalently a sequence $(\omega_m,h)$ where the Kähler metric varies with $m$; the important datum is the sequence $\GH_\bullet=(\GH_m)_{m\in\N_0}$ of Hilbert spaces. We call $\GH_\bullet$ a ``projectively induced quantization" of $(M,L)$ (again see \S\ref{quantcommsec} for details) if $\GH_\bullet$ is a \textbf{subproduct system} in the sense of \cite{ShSo1}, i.e.
$$
\GH_l\subset\GH_m\otimes\GH_{l-m},\qquad\forall m\leq l\in\N.
$$
This is the condition which allows the construction of an inductive system of unital completely positive maps
$$
\iota_{m,l}:\Bi(\GH_m)\to\Bi(\GH_l),\qquad\forall m\leq l\in\N_0
$$
such that $C(M)$ is recovered as a $C^*$-algebra as the limit of the matrix algebras $\Bi(\GH_m)$ as $m$ goes to infinity. 

An inductive system which recovers $C(M)$ could possibly be constructed also from a prequantum quantization, but the inductive system will have less nice properties, depending on the geometry of $M$. It would be interesting to study further if one could characterize (``stability") properties of $M$ as existence of an inductive system of matrix algebras with certain properties (see Conjecture \ref{DTYQKconj} below). 

What is special to the case when $M$ is a coadjoint orbit is that then the maps $\iota_{m,l}$ intertwine the tracial states. In that case Hawkins showed that one obtains $C(M)$ as a projective limit is well \cite{Hawk2}. He also showed that a similar construction works for general polarized manifolds if one uses a prequantum quantization \cite{Hawk2}. Our contribution is the description of these constructions on a single Hilbert space, namely the Hilbert-space direct sum
$$
\GH_\N:=\bigoplus_{m\in\N_0}\GH_m,
$$
which (as will be discussed in much more detail in future work) can be realized as a Hilbert space of analytic functions on a subvariety of the unit ball in $\C^n$ (cf. \cite{DRS1, DRS2}). The space $\GH_\N$ sits as a subspace of the symmetric Fock space $\GH^{\vee\N}$ over $\GH_1$, and the point is that $\GH^{\vee\N}$ can be identified with the so-called Drury--Arveson space which has a very central role in multivariable operator theory. Our ongoing work uses the results of the present paper to study the geometry of $(M,L)$, in particular the stability of vector bundles over $M$, using operator theory in Drury--Arveson space. 

The subproduct property of the sequence $\GH_\bullet$ says precisely that $\GH_\N$ is invariant under the backward shifts on $\GH^{\vee\N}$. The inductive and projective limits mentioned above encode the data of a ``strict quantization", as needed to generalize the classical setting, but they are much more convenient than just knowing that there is a strict quantization because they are constructed using the shift operators on $\GH_\N$. The notion of subproduct systems and the associated operator algebras of shift operators provide a machinery for explicit calculations that has previously been available mainly in the case of $\C\Pb^{n-1}$, where Berezin quantization and fuzzy geometry has been successfully described in terms of creation and annihilation operators (the unnormalized shift operators) (see e.g. \cite{BDLMC}).

\subsection{Algebraic geometry}

There is an analogue of Berezin quantization in projective algebraic geometry saying that the vector spaces $H^0(M;L^m)$ of algebraic sections of higher tensor powers of a very ample line bundle $L$ over a projective complex variety $M\subset\C\Pb^{n-1}$ determine all coherent sheaves on $M$. Namely, take the Abelian category $\gr(\Ai)$ of finitely generated graded modules over the graded algebra
$$
\Ai=\bigoplus_{m\in\N_0}H^0(M;L^m).
$$
Elements of $\Ai$ are not functions on the variety $M$. In order to pass to objects defined on $M$ we have to work modulo the modules which are finite-dimensional as vector spaces over $\C$, by replacing the homomorphism spaces $\Hom_{\gr(A)}(E,F)$ by
$$
\Hom_{\qgr(A)}(\Ei,\Fi):=\lim_{m\to\infty}\Hom_{\gr(A)}(E_{\geq m},F),
$$
where $E_{\geq m}:=\bigoplus_{l\geq m}E_l$ is a tail of a graded $A$-module $E=\bigoplus_{k\in\Z}E_k$. It was shown in \cite{Serre2} that in this way we obtain an Abelian category $\qgr(\Ai)$ which is equivalent to the Abelian category $\coh(M)$ of coherent sheaves on $M$. This result is also the starting point for noncommutative projective algebraic geometry \cite{ArZh1}, namely one starts with a finitely generated graded algebra $\Ai=\bigoplus_{m\in\N_0}\Ai_m$ satisfying less stringent assumptions than commutativity and one studies the category $\qgr(\Ai)$ guided by the geometric intuition from analogy with commutative algebraic geometry. 

%Smooth version, ``smooth homogeneous coordinate ring", ``smooth projective scheme"...
%...
Our observation in \S\ref{indlimsec} is that if we endow the subspaces $\Ai_m\subset\Ai$ by inner products such that the resulting Hilbert spaces $\GH_m$ form a subproduct system then the above inductive system can be explicitly described by unital completely positive maps and is compatible with the inner products in a sense (namely we obtain a generalized inductive limit of $C^*$-algebras as described below). Instead of $\Ai$ one considers a $\N_0$-graded ring $\Ri=\bigoplus_{m\in\N_0}\Ri_m$ which can be concretely described as a (non-$*$-closed) algebra of operators on the Fock space $\GH_\N$. The limit
$$
^{0}\Bi_\infty:=\lim_{m\to\infty}\End_{\gr(\Ri)}(\Ri_{\geq m})
$$
exists for the same reason as in the case of $\Ai$. We show that $^{0}\Bi_\infty$ coincides with the normally ordered part of the algebraic ``Cuntz--Pimsner core'' of the subproduct system (for the definitions see \S\ref{supprodsec}). In the commutative case it is the algebra of real-algebraic $\C$-valued functions on the variety $M\subset\C\Pb^{n-1}$ associated with $\GH_\bullet$. Moreover, there is a natural norm on $^{0}\Bi_\infty$ such that its completion $\Bi_\infty$ is (in the commutative case) the $C^*$-algebra $C(M)$ of continous functions on $M$. 

The quotient category $\qgr(\Ri)$ is (as will be discussed in another publication) in the commutative case equivalent to the category of torsionfree modules over the sheaf of real-algebraic functions on the variety $M$. From there it is easy to go to modules over the sheaf of continuous functions. Thus we are discussing a continuous version of Serre's theorem and thereby a continuous analogue of Artin--Zhang approach to noncommutative projective geometry.

The ring $\Ai$ is contained as a graded subring of $\Ri$,
$$
\Ai_m\subset\Ri_m,\qquad \forall m\in\N_0,
$$
and one can define a ``holomorphic structure'' on (the continuous extensions of) objects in $\qgr(\Ri)$ as an operator on modules over $\Bi_\infty$ possessing certain properties. Then we are lead to the following research problem:
\begin{Problem}[Noncommutative GAGA]
Find examples of noncommutative $\N_0$-graded algebras $\Ai$ such that $\qgr(\Ai)$ is equivalent to the holomorphic subcategory of $\qgr(\Ri)$.
\end{Problem}

The works \cite{DRS1, DRS2} can be regarded as an operator-theoretic approach to noncommutative projective geometry using shift operators on Fock space. What we are doing here suggest that this gives the analytic version of Artin--Zhang's algebraic approach.

%In this way, when we choose the inner products so that we obtain a subproduct system, the passage from algebraic to differential geometry goes as smoothly as possible and the resulting limit $^{0}\Bi_\infty$ has, due to the inner products on finite-dimensional approximands, more structure than merely that of an algebra over $\C$. 

%The inner product on the full Fock space $\GH^{\otimes\N}$ is the unique inner product ``compatible'' with the ring-theoretic inductive system $\End_{\Gr(\Ri)}(\Ri_{\geq m})$ for every connected $\N_0$-graded algebra $\Ai\subset\C\bra z_1,\dots,z_n\ket$, in the sense that the map $\End_{\Gr(\Ri)}(\Ri_{\geq m})\to\End_{\Gr(\Ri)}(\Ri_{\geq l})$ for $m\leq l$ can be realized explicitly as a \emph{unital} completely positive map $\iota_{m,l}:\Bi(\GH_m)\to\Bi(\GH_l)$ on the algebra of operators on the Hilbert spaces $\GH_m$ obtained by endowing $\Ai_m$ with the inner product of $\GH^{\otimes m}$.

\subsection{Operator algebra}
The $C^*$-algebra $\Bi_\infty$ is a special case of a ``generalized inductive limit" of $C^*$-algebras in the sense of \cite{BlKi1}. What is special with the inductive system coming from a subproduct system is not only that it is ``NF", i.e. defined by completely positive maps, but also that the underlying algebraic inductive limit of sets is already an algebra. 

A tracial state $\omega$ on a $C^*$-algebra such as $\Bi_\infty$ can have finite-dimensional approximation properties of various strengths, e.g. ``quasidiagonality'', ``uniform quasidiagonality", ``amenability'' and ``uniform amenability"; see \cite[\S3]{Brow1}. These notions involve tracial states $\phi_m:\Bi(\GH_m)\to\C$ on finite-dimensional matrix algebras such that $\omega$ is in some sense the limit of the $\phi_m$'s as $m\to\infty$, in that there are completely positive maps $\breve{\varsigma}^{(m)}:\Bi_\infty\to\Bi(\GH_m)$ intertwining $\omega$ and $\phi_m$ up to some error. In this case we call $(\breve{\varsigma}^{(m)})_{m\in\N_0}$ an \textbf{explicit realization} of the approximation property of $\omega$. 

Consider the case of a commutative smooth projective variety $M$, where as mentioned we shall show that the inductive limit $\Bi_\infty$ coincides with $C(M)$. Berezin quantization can be regarded as an explicit realization of the quasidiagonality of states on $C(M)$, namely the states associated with volume forms of Kähler metric on $M$. The point is not to show that some traces are quasidiagonal, but to obtain an explicit realization where the approximating finite-dimensional tracial states $\phi_m:\Bi(\GH_m)\to\C$ are defined on matrix algebras whose dimension $n_m=\chi(\Oi_\M(m))$ depend on the complex-analytic structure of $\M$. 

With a \emph{prequantum} quantization $\GH_\bullet$ of $(M,L)$ one will not obtain a subproduct system unless the quantization is ``regular" in the sense of \cite{CGR2}. However, Blackadar--Kirchberg's notion of generalized inductive limit is very general and one can probably still obtain $C(M)$ as such an inductive limit, albeit in a weaker sense than as the Cuntz--Pimsner algebra of a subproduct system. If one could characterize precisely the properties of such an inductive limit needed for the existence of a constant scalar curvature Kähler metric in the class $c_1(L)$, one could hope to find the correct stability condition on the manifold $(M,L)$ which characterizes the existence of such a metric.

By analogue with the Donaldson--Tian--Yau conjecture about the equivalence of the exstence of a constant scalar curvature Kähler metric and some stability property (``$K$-stability'') \cite{Don9, Don12, Tian1, Tian2} we formulate the following speculative but suggesting conjecture:
\begin{Conjecture}\label{DTYQKconj}
There exists an approximation property (say ``$QK$-stability'') of traces on $C^*$-algebras (perhaps quasidiagonality or amenability) such that the following holds: If $\omega$ is a Kähler form on a smooth projective variety $M\subset\C\Pb^{n-1}$ then $\omega$ has constant scalar curvature if and only if there is a prequantum quantization $\GH_\bullet$ of $\omega$ which explictly realizes $\omega$ as a $QK$-stable state on $C(M)$.
\end{Conjecture}
The inductive and projective limits constructed in the present papers may serve as guidance for less well-behaved structures associated with quantizations with prequantum condition. Our ongoing work is more focused on showing that such a characterization is possible for the analogous problem of stability of vector bundles (where the analogue of the Donaldson--Tian--Yau conjecture is known to be true \cite{LuTe1}).

\subsection{Operator theory}

%Finally, the identification of these shift-operator algebras with inductive limits allows us to solve the ten years-old Arveson conjecture. 
Starting from any subproduct system $\GH_\bullet=(\GH_m)_{m\in\N_0}$ one can form its Fock space $\GH_\N:=\bigoplus_{m\in\N_0}\GH_m$, which is a subspace of the full Fock space $\GH^{\otimes\N}:=\bigoplus_{m\in\N_0}\GH^{\otimes m}$ which is invariant under the backward shift operators. The shift operators $S_1,\dots,S_n$ on $\GH_\N$ commute pairwisely if and only if $\GH_\N$ is contained in the symmetric Fock space $\GH^{\vee\N}=\bigoplus_{m\in\N_0}\GH^{\vee m}$, where $\GH^{\vee m}$ denotes the $m$th symmetric tensor power of the Hilbert space $\GH=\GH_1=\C^n$. The symmetric Fock space $\GH^{\vee\N}$ can be identified with a Hilbert space of analytic functions on the unit ball $\B^n\subset\C^n$, called the \textbf{Drury--Arveson space} and usually denoted by $H^2_n$, in such a way that the shifts on $\GH^{\vee\N}$ identify with the operators $M_1,\dots,M_n$ on $H^2_n$ acting by multiplication by the coordinate functions $z_1,\dots,z_n$ on $\B^n$. 

Subspaces of $H^2_n$ which are invariant under the adjoint shifts $M_1^*,\dots,M_n^*$ (briefly, ``quotient modules" of $H^2_n$) have a simple description generalizing that of the Beurling representation of quotient modules of the Hardy space $H_1^2=H^2(\Sb^1)$ of the unit circle \cite{McTr1}. The most basic open question about these quotient modules is whether the following is true. 
\begin{Conjecture}[Arveson's conjecture]\label{Arvcon}
Let $\GH_\N$ be a graded quotient module of the Drury--Arveson space $\GH^{\vee\N}$. Then the shift operators $S_1,\dots,S_n$ on $\GH_\N$ form an essentially normal $n$-tuple, i.e.
$$
[S_j^*,S_k]\in\Ki,\qquad\forall j,k=1,\dots,n.
$$
\end{Conjecture}
In the present paper we shall prove Conjecture \ref{Arvcon} by operator-algebraic methods. Moreover, the geometric interpretations discussed here lead to an alternative more geometric proof which we will present in a separate paper.

The conjecture has been proven previously in several special cases by different means (see e.g. \cite{Arv8, Arv9, Doug4, DoWa2, DoWa4, EnEs1, Esch1, Kenn1, KeSh1, KeSh2, Sha3} and references therein), and the proofs have provided great new insights in the structure of submodules and quotient modules of Drury--Arveson space, Hardy space and Bergman space. It is probably not possible to have essential normality for any quotient module, as counterexamples have been found e.g. in the case of the Hardy space of the polydisk \cite{Doug4}, and one has to restrict attention to some special subclass of quotient modules. In relation to projective geometry one only needs \emph{graded} quotient modules, which is why Conjecture \ref{Arvcon} is very important there. 
%it seems like one has to restrict attention to ``graded completions" of the polynomial ring $\C[z_1,\dots,z_n]$ \cite{Arv9}. But for every projective variety $M\subset\C\Pb^{n-1}$ one can consider such graded completions of the homogeneous coordinate ring of $M$ defined by the restriction to $M$ of the hyperplane bundle $\Oi(1)$ over $\C\Pb^{n-1}$. We will show that, when the graded inner product on the coordinate ring is the Fock-space one, essential normality holds for any $M\subset\C\Pb^{n-1}$, just as for $\C\Pb^{n-1}$.

%\begin{thm}
%Arveson's conjecture is true for every graded quotient module $\GH_\N$ of the Drury--Arveson space $\GH^{\vee\N}$. %Moreover, the Hilbert space $\GH_\N$ identifies with the reproducing kernel Hilbert space on a homogeneous subset $\B\subset\B^n$ with reproducing kernel
%$$
%k_\B(z,w)=\frac{1}{1-h^L(z,w)}
%$$
%where $h^L(z,w)$ is the ``globalization" (cf. \cite{CaDA1}) of a positively curved Hermitian metric on a polarized manifold $(\M,\Li)$.
%\end{thm}
%In this way we obtain a direct proof of the famous Agler--McCarthy theorem \cite{AgMc1}, \cite[\S7]{AgMc2}:
%\begin{cor}

%\end{cor}
Let $\Ti_\GH$ be the $C^*$-algebra generated by $S_1,\dots,S_n$ and let $\Oi_\GH$ be the quotient $\Ti_\GH/\Ki$ by the ideal of compact operators. Then Conjecture \ref{Arvcon} can be reformulated by saying that the $C^*$-algebra $\Oi_\GH$ is commutative. The algebras $\Ti_\GH$ and $\Oi_\GH$ both have a natural $\Z$-grading due to the grading of $\GH_\N$. Our proof of the essential normality of the tuple $(S_1,\dots,S_n)$ is based on the observation that the degree-0 part $\Oi_\GH^{(0)}$ of $\Oi_\GH$ can be constructed as an inductive system of sets. The inductive system is given by explicit unital completely positive maps $\iota_{m,l}:\Bi(\GH_m)\to\Bi(\GH_l)$ for $m\leq l\in\N_0$. To show that the limit is a unital $C^*$-algebra it is therefore enough to show that the $\iota_{m,l}$'s are ``asymptotically multiplicative''. To see this we forget about the Hilbert-space structure and regard the $\iota_{m,l}$'s as maps between matrix rings. Then we observe that these matrix rings all sit inside a bigger $\N_0$-graded ring $\Ri$ and that the inductive system is of a certain type which is known to have a ring structure on the set-theoretic limit. 

%As forseen by Arveson, the establishment of his conjecture opens up a new field of research where problems in operator theory can be tackled by geometric methods and vice versa. In upcoming works we will attempt to characterize various stability conditions on vector bundles over polarized manifolds using operator theory, and in this way obtain very interesting new results in both complex differential geometry and operator theory. 

%In slightly more detail, t
The importance of the Arveson conjecture comes from the fact that, if the $C^*$-algebra $\Ti_\GH$ generated by $S_1,\dots,S_n$ on $\GH_\N$ is commutative modulo compacts, we have a short exact sequence of $C^*$-algebras
$$
0\to\Ki\to\Ti_\GH\to C(\Sb)\to 0,
$$
where $\Sb$ is a subset of the unit sphere $\Sb^{2n-1}$ in $\C^n$. Taking the $\Un(1)$-invariant part of this sequence one obtains
$$
0\to\Gamma_0\to\Ti_\GH^{(0)}\to C(M)\to 0,
$$
and $M$ is a compact Kähler manifold endowed with a positive line bundle $L$ with associated circle bundle equal to $\Sb$. The algebra $\Ti_\GH^{(0)}$ acts on $\GH_\N$ which, just as the ambient space $H^2_n$, has the property that every invariant subspace under $S_1,\dots,S_n$ has a ``Beurling decomposition" \cite{McTr1}. Such submodules give rise to coherent sheaves on $M$ and sometimes to vector bundles over $M$. We expect that vector bundles arising like this have certain stability properties not possessed by an arbitrary vector bundle over $M$.
%While Toeplitz quantization of a projective Kähler manifold does not require a subproduct structure, the latter gives a much more well-behaved quantization. We show in the is paper that such a strong kind of quantization is possible for any polarized manifold $(M,L)$. This result will be useful in studying questions such as the existence of constant scalar curvature Kähler metric in the class $c_1(L)$, since if one takes instead a prequantum quantization, the obstruction to a subproduct structure is coming from the nonconstancy of the scalar curvature. For studying similar questions for vector bundles however, one is free to choose the Kähler metric as one wishes and then it is very useful (not to say crucial) to have a subproduct quantization. Indeed, one then has all the tools and results from operator theory on spaces such as $\GH_\N$ available. 

%Somewhat surprising is perhaps the strong connection with the notion of balanced metrics \cite{Don1, Don2} 

\subsection{Quantum homogeneous manifolds}

Smooth projective varieties include in particular all coadjoint orbits $G/K$ and, going noncommutative, we observe that every compact matrix quantum group $\G$ defines a ``noncommutative manifold" $\G/\K$ with properties resembling very much those of a coadjoint orbit. Our main aim is then to show that the $C^*$-algebra $C(\G/\K)$ can be recovered from a suitably chosen subproduct system $\GH_\bullet$ via the noncommutative version of strict quantization sketched above. 

In fact the definition of $\G/\K$ is very simple. Let $\G$ be a compact matrix quantum group with defining unitary representation $u\in\Mn_n(\C)\otimes C(\G)$. We let $z_j:=u_{1,j}$ for $j=1,\dots,n$ denote the elements of the first row of $u$. Then $C(\G/\K)$ is defined as the $C^*$-algebra generated by elements of the form $z_{j_1}\cdots z_{j_m}z_{k_m}^*\cdots z_{k_1}^*$ for all multi-indices $(j_1,\dots,j_m)$ and $(k_1,\dots,k_m)$ of equal length. We also denote by $C(\Sb_\G)$ the $C^*$-algebra generated by $z_1,\dots,z_n$ (and refer to it as the ``first-row algebra").

%We can summarize one of the main results of this paper in the following way. 
Suppose that the compact matrix quantum group $\G$ is such that every element of $C(\G/\K)$ can be ``normally ordered", in the sense that all $z_j^*$'s are to the right of the $z_k$'s (this is a crucial assumption). Assume also that the Haar state is faithful on $C(\G/\K)$. As recalled in \S\ref{CMQGprels}, every compact quantum group $\G$ has a dual discrete quantum group $\hat{\G}$, and both $\G$ and $\hat{\G}$ ``act" on the $C^*$-algebra $C(\G)$. Hence they also act on the subalgebra $C(\Sb_\G)$.
\begin{thm} For each $m\in\N_0$, let $\GH_m$ be the Hilbert space spanned by products $z_{j_1}\cdots z_{j_m}$ with inner product coming from the Haar measure on $\G$ (the sequence of $\GH_m$'s is the ``$\G$-subproduct system" $\GH_\bullet$). Then there are explicit $\G$-$\hat{\G}$-biequivariant injections (``Berezin covariant symbol maps")
$$
\varsigma^{(m)}_\G:\Bi(\GH_m)\to C(\G/\K)
$$ 
and explicit surjections (``Toeplitz quantization maps")
$$
\breve{\varsigma}^{(m)}_\G:C(\G/\K)\to\Bi(\GH_m)
$$
such that $\varsigma^{(m)}_\G\circ\breve{\varsigma}^{(m)}_\G$ converges point-norm to the identity on $C(\G/\K)$ as $m\to\infty$. This effects an isomorphism
$$
C(\Sb_\G)\cong\Oi_\GH
$$
of the first-row algebra $C(\Sb_\G)$ with the Cuntz--Pimsner algebra $\Oi_\GH$ of the subproduct system $\GH_\bullet$, and this isomorphism is equivariant for the natural ergodic actions of $\G$ and its discrete dual group $\hat{\G}$.
\end{thm}
Denoting by $\Oi_\GH^{(0)}$ the $\Un(1)$-invariant part of $\Oi_\GH$ under the gauge action, we also have
\begin{equation}\label{desidisom}
C(\G/\K)\cong\Oi_\GH^{(0)},
\end{equation}
and it is the $C^*$-algebra $\Oi_\GH^{(0)}$ which will occupy most of the paper. We will first realize $\Oi_\GH^{(0)}$ as a ``generalized inductive limit" in the sense of \cite{BlKi1}, \cite{Hawk1} (as in the case of a general subproduct system as described before), and then as a generalized projective limit in the spirit of \cite{Hawk1}. Then we do the same thing for $C(\G/\K)$ to obtain the desired isomorphism \eqref{desidisom}. 

The quantization of coadjoint orbits is studied in detail in \cite{Hawk1} and \cite{Rie2}. Berezin quantization for quantum homogeneous spaces of compact quantum groups with tracial Haar state was discussed in \cite{Sain}.  

The idea of looking at the $\G$-subproduct system was partially motivated by Woronowicz' reconstruction of a compact matrix quantum group $\G$ from its irreducible representations \cite{Wor3}. Since $\GH_\bullet$ only contains a subset of all irreducible representations (in general), we recover not $C(\G)$ but $C(\Sb_\G)$. 

For a classical manifold $M$, with quantization defined by a line bundle $L$ over $M$, elements of $\Oi_\GH^{(k)}$ are continuous sections of the line bundle $L^{\otimes k}$. The subspace $\GH_m\subset\Oi_\GH^{(m)}$ consists of the holomorphic sections of the line bundle. For $M=G/K$ a homogeneous space, the $\GH_m$'s are irreducible representations of $G$. In general, the subspaces $\Oi_\GH^{(k)}$ for $k\in\Z\setminus\{0\}$ are Hilbert modules over $\Oi_\GH^{(0)}$, and for a quantum homogeneous space $\M=\G/\K$ each $\GH_m$ is an irreducible representation of $\G$. In this sense, $\Oi_\GH$ is a kind of ``Borel--Weil algebra" for $\G$ (cf. \cite[Thm. 14.1]{Seg1}).

We shall also compare our results with ``noncommutative random walks" on duals of compact quantum groups \cite{Iz1, INT1, Iz4}. The Toeplitz core $\Ti_\GH^{(0)}$ plays the role of Martin compactification of a walk restricted to the ``dual" of $\G/\K$ while $\Oi_\GH^{(0)}$ is the boundary of the walk. Recalling that the simplest Cuntz--Pimsner algebras $C(\Sb^{2n-1})$ of functions on spheres behave like boundaries of the corresponding Toeplitz algebras, these observations are not too surprising.

%\subsection{Physical motivation}

%In physical terms, the fact that the limit state behaves like the volume form associated with a Kähler form should be interpreted as an open quantum system analogue of Dirac's quantization condition ``commutators go over to Poisson brackets in the classical limit''. 

%...
%superselection, subproduct systems needed instead of product systems.
%...

This shows that in the presence of time-reversal symmetry, one obtains a noncommutative random walk, not on the dual of a compact quantum group, but on the dual of a compact quantum homogeneous space. In this way one can use results from the theory of noncommutative random walks to study physically relevant quantum walks, even though these two notions of ``walks" are a priori completely unrelated. Most significant is the possibility to describe the infinite-time limit in a mathematically rigorous way, as provided by the Martin boundary of a noncommutative random walk (or rather, since we really end up with random walks on homogeneous spaces, one has the dequantization manifold replacing the Martin boundary as the infinite-time limit).

\begin{Ack}
The author thanks Adam Rennie for great comments, support and inspiration. Many thanks also to Orr Shalit and Guy Salomon for finding an important error in an earlier version of this paper. We thank Orr Shalit also for other remarks of great value.
\end{Ack}

\section{Subproduct systems}\label{supprodsec}
\subsection{Basic properties}
In this paper we write $\N_0:=\N\cup\{0\}=\{0,1,2,\dots\}$. 
\begin{dfn}{[\cite[Def. 6.2]{ShSo1}]} A \textbf{subproduct system} is a sequence $\GH_\bullet=(\GH_m)_{m\in\N_0}$ of finite-dimensional Hilbert space $\GH_m$ such that $\GH_0=\C$ and
\begin{equation}\label{subprodcond}
\GH_{m+l}\subseteq\GH_m\otimes\GH_l,\qquad\forall m,l\in\N_0.
\end{equation}
\end{dfn}
We shall always denote $\GH_1$ by $\GH$ and, throughout this paper, $n\in\N$ will always be the dimension of $\GH$, 
$$
\GH\cong\C^n.
$$
\begin{Example}\label{fullandbosonex}
Given a finite-dimensional Hilbert space $\GH$, we set $\GH_m:=\GH^{\otimes m}$ for each $m$ and refer to it as the ``product system" associated with $\GH$. Another example of a subproduct system is obtained by taking $\GH_m:=\GH^{\vee m}$ to be the $m$th symmetric power of $\GH$; this is the ``symmetric subproduct system". 
\end{Example}
\begin{dfn}\label{commutedef}
A subproduct system is \textbf{commutative} if $\GH_m\subseteq\GH^{\vee m}$ for all $m\in\N_0$.
\end{dfn}
\begin{Lemma}{[\cite[Lemma 6.1]{ShSo1}]} Let $\GH_\bullet$ be a subproduct system. Then the projections $p_m:\GH^{\otimes m}\to\GH_m$ satisfy
\begin{equation}\label{subprodprojdom}
p_l(p_m\otimes  p_{l-m})p_l=p_l=p_l( p_{l-m}\otimes p_m)p_l
\end{equation}
whenver $m\leq l$.
\end{Lemma}
\begin{proof} Replacing $l$ by $l-m$ for $m\leq l$, the condition \eqref{subprodcond} reads
\begin{equation}\label{subprodcondalt}
\GH_l\subseteq\GH_m\otimes\GH_{l-m},\qquad\forall m\leq l\in\N_0.
\end{equation}
Writing \eqref{subprodcondalt} in terms of projections,
$$
p_l\leq p_m\otimes p_{l-m},
$$
the result is clear.
\end{proof}
\begin{dfn} The \textbf{Fock space} associated with a subproduct system $\GH_\bullet$ is the Hilbert space 
$$
\GH_\N:=\bigoplus_{m\in\N_0}\GH_m.
$$
We regard $\GH_\N$ as a subspace of full Fock space $\GH^{\otimes \N}:=\bigoplus_{m\in\N_0}\GH^{\otimes m}$ and denote by $p_\N=\sum_{m\in\N_0}p_m$ the projection from $\GH^{\otimes \N}$ onto $\GH_\N$. Thus $p_\N$ is the identity in $\Bi(\GH_\N)$, just as $p_m$ is the identity in $\Bi(\GH_m)$.
\end{dfn}
\begin{Example} If $\GH_\bullet=\GH^{\otimes \bullet}$ is the product system over a fixed Hilbert space $\GH$ then $\GH_\N$ is the full (or ``Boltzmannian") Fock space $\GH^{\otimes \N}$ over $\GH$. 
\end{Example}
\begin{Example}
If $\GH_\bullet=\GH^{\vee \bullet}$ is the full commutative subproduct system then $\GH_\N$ is the symmetric (or ``Bosonic") Fock space $\GH^{\vee \N}:=\bigoplus_{m\in\N_0}\GH^{\vee m}$ over $\GH$.
\end{Example}
\begin{Notation} We denote by $\C\bra \mathbf{z}\ket=\C\bra z_1,\dots,z_n\ket$ the algebra of polynomials in $n$ freely commuting variables. We denote by $\C[\mathbf{z}]=\C[z_1,\dots,z_n]$ the algebra of polynomials in $n$ commuting variables.
\end{Notation}
 For any polynomial $f(z_1,\dots,z_n)=\sum_{j_1,\dots,j_n}f_{j_1,\dots,j_n}z^{j_1}_1\cdots z^{j_n}_n$ in $\C\bra  z_1,\dots,z_n\ket$, evaluation on the basis $e_1,\dots,e_n$ for $\GH=\GH_1$ defines an element in Fock space,
$$
f(e_1,\dots,e_n):=\sum_{j_1,\dots,j_n}f_{j_1,\dots,j_n}e^{j_1}_1\otimes\cdots\otimes e^{j_n}_n\in\GH_\N,
$$
and $f$ is homogeneous iff
$$
f(e_1,\dots,e_n)\in\GH^{\otimes m},\qquad \text{for some }m\in\N_0.
$$
\begin{Lemma}\label{idealsublemma}
Every subproduct system $\GH_\bullet=(\GH_m)_{m\in\N_0}$ defines a homogeneous ideal in $\C\bra\mathbf{z}\ket$, where $n=\dim(\GH)$. Conversely, every homogeneous ideal in $\C\bra \mathbf{z}\ket$ corresponds to a subproduct system \cite[Prop. 7.2]{ShSo1}. For commutative subproduct systems the same is true with $\C[\mathbf{z}]$ instead \cite[§2.3]{DRS1}.
\end{Lemma}
\begin{sop} Given $\GH_\bullet$, define a homogeneous ideal $\Ii$ in $\C\bra \mathbf{z}\ket$ by
$$
\Ii:=\{f\in\C\bra z\ket|f(e_1,\dots,e_n)\in\GH^{\otimes m}\ominus\GH_m\text{ for some }m\in\N_0\}.
$$
Conversely, given a homogeneous ideal $\Ii$, we associate the Hilbert spaces 
$$
\GH_m:=\GH^{\otimes m}\ominus\{f(e_1,\dots,e_n)|f\in\Ii^{(m)}\},
$$
where $\Ii^{(m)}$ is the degree-$m$ component of $\Ii$.
\end{sop}

\subsection{Toeplitz algebras}
Having fixed a subproduct system $\GH_\bullet$, we shall always denote by $S_1,\dots,S_n$ the operators on Fock space $\GH_\N$ defined by
$$
S_k\phi:=p_{m+1}(e_k\otimes \phi),\qquad \forall \phi\in\GH_m
$$ 
for all $m\in\N_0$, where $p_{m+1}:\GH^{\otimes (m+1)}\to\GH_{m+1}$ is the orthogonal projection. They are the compressions to $\GH_\N$ of the left shifts $\psi\to e_k\otimes\psi$ on full Fock space $\GH^{\otimes\N}$. For more about compressed $n$-tuples of shift operators, see \cite{Pop1}, \cite{Pop2}, \cite{ShSo1}, \cite{DRS1}, \cite{DRS2}.
\begin{dfn} The \textbf{Toeplitz algebra} of a subproduct system $\GH_\bullet$ is the unital $C^*$-algebra $\Ti_\GH$ of operators on $\GH_\N$ generated by the shifts $S_1,\dots,S_n$. 
\end{dfn}
The backward shifts on $\GH^{\otimes \N}$ preserves the subspace $\GH_\N$, so the adjoints $S_1^*,\dots,S_n^*$ of $S_1,\dots,S_n$ are just the restrictions to $\GH_\N$ of the backward shifts on $\GH^{\otimes \N}$.

\begin{dfn} The \textbf{vacuum state} on the Toeplitz algebra $\Ti_\GH$ is the restriction $\hat{\ep}:\Ti_\GH\to\C$ of the vector state on $\Bi(\GH_\N)$ defined by the unit vector $\Omega\in\GH_0=\C$. That is,
$$
\hat{\ep}(X):=\bra\Omega|X\Omega\ket,\qquad \forall X\in\Bi(\GH_\N).
$$
\end{dfn}
If $p_0$ denotes the unit in $\Bi(\GH_0)$ then
\begin{equation}\label{vacuumcounit}
Xp_0=\hat{\ep}(X)p_0=p_0X,\qquad \forall X\in\Bi(\GH_\N).
\end{equation}

\begin{Notation}\label{freeunitsemnot}
Let $\F_n^+$ be the free unital semigroup generated by $n$ elements $1,\dots,n$ (the empty word $\emptyset$ is the identity in $\F_n^+$). We write a word $\mathbf{k}\in\F_n^+$ as $\mathbf{k}=k_1\cdots k_m$ and refer to $|\mathbf{k}|:=m$ as the \textbf{length} of $\mathbf{k}$. For the shifts $S_1,\dots,S_n$ and the basis vectors $e_1,\dots,e_n$ we then write
$$
S_\mathbf{k}:=S_{k_1}\cdots S_{k_m},\qquad S_\mathbf{k}^*:=(S_\mathbf{k})^*=S_{k_m}^*\cdots S_{k_1}^*,
$$
$$
e_\mathbf{k}:=e_{k_1}\otimes \cdots \otimes e_{k_m},
$$
and similarly for other $n$-tuples of elements defined below. Finally, $\mathbf{j}\mathbf{k}:=j_1\cdots j_lk_1\cdots k_m$ for $\mathbf{j},\mathbf{k}\in\F_n^+$ with $|\mathbf{j}|=l$ and $|\mathbf{k}|=m$.
\end{Notation}

\subsubsection{The Toeplitz core}
Let $N=\bigoplus_m m\, p_m$ be the \textbf{number operator} on $\GH_\N$. It generates a unitary group on $\GH_\N$ which implements an action $\gamma_\bullet$ of the circle group $\Un(1)$ on $\Ti_\GH$,
\begin{equation}\label{gaugeacttoep}
\gamma_t(S_k):=e^{it}S_k,\qquad \forall e^{it}\in\Un(1),\, k\in\{1,\dots,n\},
\end{equation} 
referred to as the \textbf{gauge action} on $\Ti_\GH$.

The gauge action \eqref{gaugeacttoep} splits $\Ti_\GH$ into the $C^*$-direct sum of the subspaces 
$$
\Ti_\GH^{(k)}:=\{T\in\Ti_\GH|\gamma_t(T)=e^{ikt}T\text{ for all }e^{it}\in\Un(1)\},\qquad k\in\Z.
$$
The fixed-point subalgebra $\Ti_\GH^{(0)}$ (the \textbf{Toeplitz core}) will be of great importance to us. It is generated by polynomials in the shifts $S_j$ and $S_k^*$ which are ``homogeneous of degree zero" in the sense that each term contains equally many forward shifts $S_j$ as backward shifts $S_k^*$. 
 
\subsubsection{The right shifts}
In addition to the ``sinister" shift $S_k$ by the basis vector $e_k\in\GH$, we shall need the ``rectus" shift
\begin{equation}\label{rectoshift}
R_k\psi:=p_{m+1}(\psi\otimes e_k),\qquad\forall \psi\in\GH_m,m\in\N_0
\end{equation}
acting on the same Fock space $\GH_\N$. Note that $R_k$ commutes with each $S_j$, and that $R_k^*$ commutes with each $S_j^*$, but $[R_k,S_j^*]\ne 0$. 

The following formulas will be used extensively.
\begin{Lemma} For all $m,l\in\N$ with $m\leq l$ we have
$$
\sum_{|\mathbf{r}|=m}S_\mathbf{r}S_\mathbf{r}^*\big|_{\GH_l}=p_l=\sum_{|\mathbf{r}|=m}R_\mathbf{r}R_\mathbf{r}^*\big|_{\GH_l}.
$$
Hence
$$
\sum_{|\mathbf{r}|=m}S_\mathbf{r}S_\mathbf{r}^*=\sum_{l\geq m}p_l
$$
and, in particular ($m=1$), 
\begin{equation}\label{Toeplidminusvac}
\sum_{k=1}^nS_kS_k^*=\bone-|\Omega\ket\bra\Omega|.
\end{equation}
\end{Lemma}
\begin{proof} For $\mathbf{r},\mathbf{s}\in\F_n^+$ with $|\mathbf{r}|=m=|\mathbf{s}|$ we have $S_\mathbf{r}S_\mathbf{s}^*|_{\GH_l}=p_l(|e_\mathbf{r}\ket\bra e_\mathbf{s}|\otimes p_{l-m})p_l=p_l(|e_\mathbf{r}\ket\bra e_\mathbf{s}|\otimes p_{l-m})p_l$, so 
\begin{align*}
\sum_{|\mathbf{r}|=m}S_\mathbf{r}S_\mathbf{r}^*\big|_{\GH_l}&=p_l\Big(\sum_{|\mathbf{r}|=m}|e_\mathbf{r}\ket\bra e_\mathbf{r}|\otimes p_{l-m}\Big)p_l
\\&=p_l(p_m\otimes p_{l-m})p_l=p_l,
\end{align*}
and similarly for the right shifts.
\end{proof}
From \eqref{Toeplidminusvac} we see that the vacuum projection $p_0=|\Omega\ket\bra\Omega|$ belongs to the Toeplitz algebra. Considering multiplying $p_0$ from both sides with different shift operators, it is a simple matter to deduce the following.
\begin{cor} The Toeplitz algebra $\Ti_\GH$ contains the $C^*$-algebra $\Ki$ of all compact operators on Fock space $\GH_\N$ as a norm-closed two-sided ideal.
\end{cor}

\subsubsection{Normal ordering}
In \S\ref{CMasprojlimsec} we will obtain the following important result about the Toeplitz algebra, which we state here for emphasis. 
\begin{Lemma}[Normal ordering]\label{Lemmanormorder}
Let $\GH_\bullet$ be a subproduct system. Let $\Ai_\GH$ denote the norm-closed (non-$*$) algebra generated by the shifts $S_1,\dots,S_n$ and the identity in $\Bi(\GH_\N)$. Then
$$
\overline{\operatorname{span}}(\Ai_\GH\Ai_\GH^*)=\Ti_\GH.
$$
In particular, $\overline{\operatorname{span}}(\Ai_\GH\Ai_\GH^*)$ is an algebra.
%\begin{enumerate}[(i)]
%\item{$\overline{\operatorname{span}}(\Ai_\GH\Ai_\GH^*)$ is an algebra,}
%\item{$\Ti_\GH=\overline{\operatorname{span}}(\Ai_\GH\Ai_\GH^*)$, and}
%\item{$\Ti_\GH=\overline{\operatorname{span}}(\Ai_\GH^*\Ai_\GH\cup\Ki)$.}\label{antinormlemmanormord}
%\end{enumerate}
\end{Lemma}
%The proof of (i) and (ii) will be given in Lemma \ref{Toepevconst} and (iiii) will be obtained in \S\ref{CMasprojlimsec}. For the moment we just note that the $C^*$-algebra $\Ki$ of compact operator is contained in $\overline{\operatorname{span}}(\Ai_\GH\Ai_\GH^*)$. 

\subsection{Cuntz--Pimsner algebras}
\begin{dfn}[{\cite[Cor. 3.2]{Vis2}}] The \textbf{Cuntz--Pimsner algebra} of a subproduct system $\GH_\bullet$ is the quotient of the Toeplitz algebra $\Ti_\GH$ by the ideal $\Ki$ of all compact operators on $\GH_\N$,
$$
\Oi_\GH:=\Ti_\GH/\Ki.
$$
\end{dfn}
We denote by $Z_1,\dots,Z_n$ the generators of $\Oi_\GH$, i.e. the images of the shifts $S_1,\dots,S_n$ in the quotient. They satisfy the \textbf{sphere relation}
$$
\sum^n_{k=1}Z_kZ_k^*=\bone,
$$
which suggests viewing $\Oi_\GH$ as the ``boundary" of $\Ti_\GH$; in the latter holds $\sum^n_{k=1}S_kS_k^*\leq\bone$, as we saw in \eqref{Toeplidminusvac}.

The formula \eqref{gaugeacttoep}, but with $Z_k$ replacing $S_k$, defines the gauge action on $\Oi_\GH$, which gives a splitting
$$
\Oi_\GH=\overline{\bigoplus_{k\in\Z}\Oi_\GH^{(k)}}^{\|\cdot\|}
$$ 
of $\Oi_\GH$ into spectral subspaces for this $\Un(1)$-action.

\begin{Remark}[Known examples]\label{Arvvconjrem}
The most straightforward example of a subproduct Cuntz--Pimsner algebra is the Cuntz algebra $\Oi_n$, obtained from $\GH_\bullet=\GH^{\otimes\bullet}$. As a commutative example, $\Oi_\GH$ for the symmetric subproduct system $\GH^{\vee\bullet}$ was shown in \cite{Arv6c} to be isomorphic to the $C^*$-algebra $C(\Sb^{2n-1})$ of continuous functions on the unit sphere $\Sb^{2n-1}\subset\C^n$. Cuntz--Pimsner algebras coming from monomial ideals were described in \cite{KaSh1}. 
\end{Remark}
We can reformulate Conjecture \ref{Arvcon} as a statement about Cuntz--Pimsner algebras:
\begin{Conjecture}[Arveson's conjecture]\label{Arvconreform}
For every commutative subproduct system $\GH_\bullet\subset\GH^{\bullet\vee}$, the Cuntz--Pimsner algebra $\Oi_\GH$ is commutative.
\end{Conjecture}

For any subproduct system $\GH_\bullet$, the spectral subspaces $\Oi_\GH^{(k)}$ for the gauge action on $\Oi_\GH$ are Hilbert $C^*$-bimodules over the fixed-point subalgebra $\Oi_\GH^{(0)}$, with left and right inner products
\begin{equation}\label{innerprodCuntzhilbmod}
\bra\xi|\eta\ket_{\rig}:=\xi^*\eta,\qquad \bra\xi|\eta\ket_{\lef}:=\xi\eta^*
\end{equation}
for $\xi,\eta\in\Oi_\GH^{(k)}$.

\section{Review of quantization of projective varieties}\label{quantcommsec}
Let us formulate Berezin quantization of complex submanifolds of projective $n$-space $\C\Pb^{n-1}$ in terms of subproduct systems, just to make it clear how the results of the subsequent sections relate to the classical ones. 

\subsection{Berezin quantization with prequantum condition}\label{quantkaehlersec}
Recall that Chow's theorem says that a submanifold of projective space $\Pb[\C^n]=\C\Pb^{n-1}$ is a nonsingular (i.e. smooth) projective variety, i.e. the zero-set of some finitely generated homogeneous ideal in $\C[z_1,\dots,z_n]$. These manifolds can be characterized without even referring to $\C\Pb^{n-1}$ (see Lemma \ref{submanvariety} below), but for this we need to recall some complex geometry. A \textbf{Kähler manifold} is a pair $(M,\omega)$ consisting of a complex manifold $M$ and a closed nondegenerate $2$-form $\omega$ on $M$ which equals the imaginary part of a Hermitian metric on $M$.
\begin{Remark}[Poisson bracket]
The Kähler form $\omega$ is in particular a symplectic (i.e. closed and nondegenerate) form, making $M$ a symplectic manifold. The nondegeneracy of $\omega$ allows us to use the inverse $\omega^{-1}$ to define a Poisson bracket on $C^\infty(M)$ by
\begin{equation}\label{Poissbrack}
\{f,g\}:=\omega^{j,k}\frac{\pd f}{\pd x_j}\frac{\pd g}{\pd x_k},
\end{equation}
if we denote by $\omega^{j,k}$ the coefficients of $\omega^{-1}$ in local Darboux coordinates $x_j$. 
\end{Remark}
 Recall that for a holomorphic line bundle $L$ with a fixed choice of Hermitian metric $h$, there is a unique connection, the ``Chern connection", which is compatible with both the metric and the holomorphic structure in a suitable sense \cite[Prop. 4.2.14]{Huy}. If we locally represent $h$ by a matrix-valued function, the curvature of this connection is given by $\bar{\pd}\pd\log h$.
\begin{dfn}[{\cite[§2.1]{BeSl1}}] A compact Kähler manifold $(M,\omega)$ is \textbf{quantizable} if there is a holomorphic Hermitian line bundle $(L,h)$ over $M$ such that the curvature $\bar{\pd}\pd\log h$ of the Chern connection satisfies the \textbf{prequantum condition}
\begin{equation}\label{prequantucond}
\sqrt{-1}\bar{\pd}\pd\log h=\omega.
\end{equation}
Then $(L,h)$ gives a \textbf{quantization} of the Kähler manifold $(M,\omega)$.
\end{dfn}
Condition \eqref{prequantucond} is there to ensure 
ensures the following. 
\begin{Lemma}[{\cite{Schl2}, \cite[\S2.1]{BeSl1}}]\label{submanvariety}
Every quantization $(L,h)$ of a compact Kähler manifold $(M,\omega)$ gives an embedding of $M$ as a submanifold of $\C\Pb^{n-1}$ for some $n\in\N$, and hence $M$ can be regarded as a projective algebraic variety (by Chow's theorem). Conversely, every smooth projective algebraic variety is a quantizable compact Kähler manifold. 
\end{Lemma}

\begin{Example}[The hyperplane bundle]\label{hyperex}
The ``tautological line bundle" over $\C\Pb^{n-1}$ is the holomorphic line bundle, usually denoted by $\Oi(-1)$, whose fiber over a point in $\C\Pb^{n-1}$ is the corresponding line in $\C^n$. The dual $\Oi(1):=\Oi(-1)^*$ of this line bundle is the \textbf{hyperplane line bundle} over $\C\Pb^{n-1}$. The triple $(\C\Pb^{n-1},\Oi(1),\omega_{\text{FS}})$, where $\omega_{\text{FS}}$ is the Fubini--Study form and $\Oi(1)$ is equipped with the Fubini--Study metric, is the prototypical example of a quantizable Kähler manifold $(M,L,\omega)$ \cite{Schl1}.
\end{Example}
The embedding into $\C\Pb^{n-1}$ mentioned in Lemma \ref{submanvariety} requires a sufficiently positive line bundle, and \eqref{prequantucond} says only that $L$ is positive (``ample" ). It may therefore be necessary to use some tensor power $L^{\otimes m}$ of the line bundle $L$. However, by replacing $L$ by $L^{\otimes m}$ and rescaling the Kähler form $\omega$ to $m\omega$ we can, and shall, assume that $L$ is itself sufficiently positive (``very ample").

The important requirement in Lemma \ref{submanvariety} is that $M$ \emph{admits} a Kähler metric, but we do not need to choose one in order to an embedding of $M$ into $\C\Pb^{n-1}$ as a complex-analytic submanifold (similarly we do not have the choose a Hermitian metric on $L$). Also, if we choose a Kähler metric $\omega$ on $M$ and a Hermitian metric $h$ on $L$, so that we can embed $M$ as a Kähler submanifold of $\C\Pb^{n-1}$, there is no need to require the prequantum relation \eqref{prequantucond} between $h$ and $\omega$; it is just the existence of metrics satisfying the prequantum condition which is needed, to ensure that there is an ample line bundle on $M$. Therefore, we will often speak of a \textbf{polarized manifold}, i.e. a pair $(M,L)$ where $M$ is a compact Kähler manifold (with no choice of Kähler metric) and $L$ is a positive line bundle on $M$ (which we shall assume very ample for convenience, with no choice of Hermitian metric specified).

Let $(M,\omega)$ be a compact Kähler manifold and let $(L,h)$ be a Hermitian line bundle over $M$. The space $H^0(M;L)$ of global holomorphic sections of $L$ is finite-dimensional and hence made into a Hilbert space after fixing any inner product on $H^0(M;L)$. In Berezin quantization one looks at the limit of large $m$ for the spaces $H^0(M;L^{\otimes m})$ of sections of the tensor powers of $L$, and therefore the inner products on these spaces should be comparable in some way. A consequence of the fact that the Kähler form $\omega$ is symplectic is that $\omega^d/d!$ (where $d:=\dim_\C M$) is a volume form on $M$ (the ``Liouville form") and we can take the inner product
\begin{equation}\label{Ltwoinnerprod}
\bra\phi|\psi\ket_{h,\omega}:=m^d\int_Mh^m(\phi(x),\psi(x))\frac{\omega(x)^d}{d!},\qquad\forall \psi,\phi\in H^0(M,L^m),
\end{equation}
where $h^m$ is the Hermitian metric on $L^{\otimes m}$ induced by $h$. Note that \eqref{Ltwoinnerprod} is determined for all $m$ by the choice of inner product on $H^0(M;L)$. If $(L,h)$ is a quantization of $(M,\omega)$ then it is reasonable to leave out either $h$ or $\omega$ from the notation in $\bra\cdot|\cdot\ket_{h,\omega}$.

%When $H^0(M,L)$ is endowed with the structure of Hilbert space defined by \eqref{Ltwoinnerprod}, we denote it by $\GH$. More generally, we write
%\begin{equation}\label{Kahlerhilbs}
%\GH_m:=(H^0(M,L^{\otimes m}),\bra\cdot|\cdot\ket),\qquad m\in\N
%\end{equation}
%and we set $\GH_0:=\C$. 

\begin{Example}[Sections of the hyperplane bundle]\label{hypersectionsex}
Recall the hyperplane line bundle $\Oi(1)$ over $\Pb[\GH^*]$ introduced in Example \ref{hyperex}. Fixing a basis for $\GH^*\cong\C^n$, the global holomorphic sections of the $m$th tensor power $\Oi(m)$ of $\Oi(1)$ are identified with the degree-$m$ homogeneous polynomials on $\C^n$. In particular, the space of holomorphic sections of $\Oi(1)$ is just the space $\GH$ of continuous linear functionals on $\GH^*$. If we define $\GH_m$ to be the Hilbert space of holomorphic sections of $\Oi(m)$ with inner product \eqref{Ltwoinnerprod} then for $L=\Oi(1)$ we simply have $\GH_m=\GH^{\vee m}$ (symmetrized tensor product). Thus, the $\GH_m$'s form the full symmetric subproduct system (see Example \ref{fullandbosonex}).
\end{Example}

Let $(L,h)$ be a quantization of $(M,\omega)$ and endow $H^0(M;L^m)$ with the inner product \eqref{Ltwoinnerprod}. The Hilbert space $L^2(M;L^{\otimes m})$ of all square-integrable sections of the line bundle $L^{\otimes m}$ is much larger than the holomorphic subspace $H^0(M;L^m)$. If $\Pi_m$ denotes the projection from $L^2(M;L^{\otimes m})$ to $H^0(M;L^m)$ then every function $f\in C(M)$ defines an operator $\breve{\varsigma}^{(m)}(f)$ on $H^0(M;L^m)$ by
\begin{equation}\label{Toeplitzop}
\breve{\varsigma}^{(m)}(f)\phi:=\Pi_m (f\phi),\qquad \forall \phi\in H^0(M;L^m),
\end{equation}
i.e. acting as multiplication by $f$ (recall that a section of a line bundle can be multiplied by continuous functions to yield a new section) followed by projection back to $H^0(M;L^m)$ (the latter step is needed since $f$ is not holomorphic unless it is constant). 
In the case $(L,h)$ is a quantization of $(M,\omega)$, we have the following. 
\begin{prop}[{\cite[Thms. 4.1,4.2, \S5]{BMS}}]\label{strictquant} Let $(L,h)$ be a quantization of a compact Kähler manifold $(M,\omega)$ and define a structure of Hilbert space on $H^0(M;L^m)$ by \eqref{Ltwoinnerprod}. Then the collection of endomorphism algebras $\End_\C H^0(M;L^m)$ and maps $C^\infty(M)\ni f\to \breve{\varsigma}^{(m)}(f)\in\End_\C H^0(M;L^m)$ defined by \eqref{Toeplitzop} gives a strict quantization of the algebra $C^\infty(M)$ in the sense of \cite[Def. 1.1.1]{Lan1}, i.e. for all $f,g\in C^\infty(M)$ we have
\begin{enumerate}[(i)]
\item{$\lim_{m\to\infty}\|\breve{\varsigma}^{(m)}(f)\|=\|f\|$ \textnormal{(Rieffel condition)},}
\item{$\lim_{m\to\infty}\|\breve{\varsigma}^{(m)}(fg)-\breve{\varsigma}^{(m)}(f)\breve{\varsigma}^{(m)}(g)\|=0$ \textnormal{(von Neumann condition)},}
\item{$\lim_{m\to\infty}\|m^{-1}[\breve{\varsigma}^{(m)}(f),\breve{\varsigma}^{(m)}(g)]-\{f,g\}\|=0$ \textnormal{(Dirac condition)},}
\end{enumerate} 
and every operator in $\End_\C H^0(M;L^m)$ is of the form $\breve{\varsigma}^{(m)}(f)$ for some $f\in C^\infty(M)$ \cite[Prop. 4.2]{BMS}. Here $\{\cdot,\cdot\}$ is the Poisson bracket \eqref{Poissbrack}.  
\end{prop}
For any volume form $\omega^d/d!$ of $(M,L)$, we have the associated Lebesgue space $L^2(M,\omega)$. The algebra $C(M)$ identifies with a subspace of $L^2(M,\omega)$. Let $\varsigma^{(m)}:\Bi(\GH_m)\to C(M)$ denote the adjoint of $\breve{\varsigma}^{(m)}:C(M)\to\Bi(\GH_m)$ with respect to the $L^2$-inner product on $L^2(M,\omega)$ and the normalized Hilbert-Schmidt inner product on $\Bi(\GH_m)$. 

For $A\in\Bi(\GH_m)$, the function $\varsigma^{(m)}(A)$ is also called the \textbf{Berezin covariant symbol} of $A$, and if $A=\breve{\varsigma}^{(m)}(f)$ then $f$ is a (non-unique) \textbf{contravariant symbol} of $A$. The map 
$$
\varsigma^{(m)}\circ\breve{\varsigma}^{(m)}:C(M)\to C(M)
$$
is the \textbf{Berezin transform} at level $m$. Similar to the famous expansion of the integral kernel for the Bergman projections $\Pi_m$ \cite{Zeld1}, the Berezin transform $\varsigma^{(m)}\circ\breve{\varsigma}^{(m)}$ has an asymptotic expansion at large $m$ \cite{KaSc1}. 

Since the Toeplitz maps $\breve{\varsigma}^{(m)}$ are surjective by Proposition \ref{strictquant}, their adjoints $\varsigma^{(m)}$ are injective. Hence we may regard the $\Bi(\GH_m)$'s as embedded in $C^\infty(M)$ as vector subspaces. One may then ask if it is possible to use the covariant symbols to approximate the whole $C^*$-algebraic structure of $C(M)$ by that of finite-dimensional matrix algebras $\Bi(\GH_m)$. 
%For homogeneous polarized manifolds $(M,L)$ (namely coadjoint orbits for compact Lie groups), the Berezin transforms converge to the identity map on $C(M)$ as $m$ goes to infinity \cite[Thm. 6.1]{Rie2}. This convergence result, which is stronger than the mere Toeplitz convergence in Proposition \ref{strictquant}, relies on the fact coadjoint orbits are ``balanced" (see below). %It is natural to believe, in view of the results in this paper, that such a strong covergence is true for a polarized manifold $(M,L)$ if it admits balanced embeddings $M\hookrightarrow\Pb[H^0(M;L^m)^*]$ using Hermitian metrics $h_m$ on $L^m$ and the rescaled metrics $m^{-1}\sqrt{-1}\pd\bar{\pd}\log h_m$ converge to a Kähler form on $M$. 

We shall see that  by changing the definition of the Toeplitz maps $\breve{\varsigma}^{(m)}$ we can in fact obtain $C(M)$  as an inductive limit for \emph{any} polarized manifold $(M,L)$.

\subsection{Projectively induced quantization}\label{projquantsec}
The vector spaces $H^0(M;L^m)$ equipped with the inner products \eqref{Ltwoinnerprod} do not always form a subproduct system. For that one has to choose $\omega$ and $h$ appropriately, and for most polarized manifolds $(M,L)$ one cannot choose them to satisfy the prequantum condition \eqref{prequantucond} at the same time.

If we do \emph{not} require that $h$ and $\omega$ are related as in \eqref{prequantucond} then any two of (i) an inner product on $H^0(M;L)$, (ii) a Hermitian metric on $L$ with positive curvature and (iii) a volume form $\omega^d/d!$ on $M$ determines the third via \eqref{Ltwoinnerprod}. By the Calabi--Yau theorem \cite{Yau1}, any volume form on $M$ can be obtained as $\omega^d/d!$ for some $\omega$ in the cohomology class $c_1(L)$ (for any choice of polarization $L$, after normalization of the volume form). If we do not require $h$ to have positive curvature then there are infinitely many Hermitian metrics $h$ giving rise to the same inner product via the same volume form (see \cite[Eq. (3.25)]{LMS1}).  

For us, the choice of inner product $\bra\cdot|\cdot\ket$ on $H^0(M;L)$ will be the important input, and it will not matter which Hermitian metric on $L$ and volume form on $M$ was used to define it. %Therefore, we can simply speak about a polarized manifold $(M,L)$.  

Given a polarized manifold $(M,L)$, a choice of basis for the $n$-dimensional vector space $H^0(M;L)$ allows use to embed $M$ into the projectivization $\Pb[H^0(M;L)^*]$ of the vector space dual to $H^0(M;L)$. The elements of the basis for $H^0(M;L)$ become the restrictions of the homogeneous coordinate functions $z_1,\dots,z_n$ on $\Pb[H^0(M;L)^*]$ to the embedded $M$.
Choosing an inner product $\bra\cdot|\cdot\ket$ on $H^0(M;L)$ we obtain an $n$-dimensional Hilbert space $\GH$ which after a choice of orthonormal basis identifies with $\C^n$, and so $M$ embeds into $\Pb[\GH^*]=\C\Pb^{n-1}$. 
Whatever inner product on $H^0(M;L)$ we used to define the Hilbert space $\GH$, it will produce the symmetric subproduct system $\GH^{\vee\bullet}$ of holomorphic sections of the hyperplane bundle on $\Pb[\GH^*]$ as in Example \ref{hypersectionsex}. 
What Lemma \ref{submanvariety} says is that the ideal determined by the algebraic relations among the $z_j$'s, appearing when we restrict them to the submanifold $M$, is homogeneous. The subspaces $H^0(M;L^m)\subset H^0(M;L)^{\vee m}$ of holomorphic sections of the tensor powers of $L$ endowed with the inner product as a subspace of $\GH^{\vee m}$ will be denoted by 
$$
\GH_m=(H^0(M;L),\bra\cdot|\cdot\ket).
$$
Here $\bra\cdot|\cdot\ket$ is thus the inner product \eqref{Ltwoinnerprod} in the special case when $\omega$ and $h$ are the restrictions to $M$ of the Fubini--Study metrics on $\Pb[\GH^*]$, depending only on the inner product on $H^0(M;L)$ which defines the one-particle Hilbert space $\GH$. We set $\GH_0:=\C$. 

We therefore have a description of polarized manifolds $(M,L)$ with the extra datum of an inner product on $H^0(M;L)$ as a collection of Hilbert spaces $\GH_m$ satisfying \eqref{subprodcond} with $\GH_m\subseteq\GH^{\vee m}$, where $\vee$ is the symmetrized tensor product (recall Lemma \ref{idealsublemma}). The subproduct system $\GH_\bullet$ is obtained from $\GH^{\vee\bullet}$ by quotiening out by the ideal in $\C[z_1,\dots,z_n]$ which defines the embedded $M$. 
\begin{cor}\label{commutecorkahler} Every commutative subproduct system $\GH_\bullet=(\GH_m)_{m\in\N_0}$ (see Definition \ref{commutedef}) determines (via its associated homogeneous ideal in $\C[z_1,\dots,z_n]$) a polarized manifold $(M,L)$ with a fixed structure of Hilbert space on $H^0(M;L)$. Conversely, every such datum $(M,L,\bra\cdot|\cdot\ket)$ determines a commutative subproduct system. 
\end{cor}
We stress that for obtaining the subproduct system $\GH_\bullet$, the inner product on $H^0(M;L)$ is arbitrary; we do not require $h$ and $\omega$ in \eqref{Ltwoinnerprod} to satisfy the pre-quantum condition. Even if we did require $\omega=\sqrt{-1}\bar{\pd}\pd\log h$, the inner product on $\GH_m\subset\GH^{\vee m}$ for $m\geq 2$ would in general differ from the inner product $\bra\cdot|\cdot\ket_h$ defined by the initial $\omega$ and $h$.  
\begin{Lemma}\label{CYlemma}
Let $(M,L)$ be a polarized manifold with $d:=\dim_\C M$. Then for any inner product $\bra\cdot|\cdot\ket$ on $H^0(M;L)$, there exists a unique volume form on $M$, which we can express as $\omega^d$ for a Kähler metric in the class $c_1(L)$, such that $\bra\cdot|\cdot\ket$ is $\omega^d$-\textbf{balanced} in the sense of \cite[\S2.2]{Don3}, i.e. if $Z_1,\dots,Z_n$ are the homogeneous coordinates on $M\subset\Pb[\GH^*]$ associated with any orthonormal basis for $\GH=(H^0(M;L),\bra\cdot|\cdot\ket)$, normalized to $\sum_{k=1}^nZ_kZ_k^*=\bone\in C^\infty(M)$, then
$$
\frac{1}{\vol(M,L)}\int_MZ_j^*Z_k\frac{\omega^d}{d!}=\frac{\delta_{j,k}}{n},\qquad\forall j,k=1,\dots,n,
$$
where $\vol(M,L):=\int_M\omega^d/d!$. 
\end{Lemma} 
\begin{proof}
The statement follows from the Calabi--Yau theorem \cite{Yau1} and the fact that every polarized manifold $(M,L)$ admits a unique $\omega^d$-balanced metric for every volume form $\omega^d$ on $M$ \cite{BLY1}, \cite[\S2.2]{Don3} (equivalently, every line bundle is stable and from this it follows that every very ample line bundle is ``balanced as a line bundle" in the sense of \cite{Wa1}).
\end{proof}
To say that $(M,L)$ is \textbf{balanced} as a polarized manifold \cite{Don1} means precisely that there exists a Hilbert space structure $\GH$ on $H^0(M;L)$ such that the volume form $\omega^d$ in Lemma \ref{CYlemma} is the restriction to $M$ of the Fubini--Study volume form on $\Pb[\GH^*]$. Not every polarized manifold is balanced, and so in general one cannot form a subproduct system from a quantization $(L,h)$ of a compact Kähler manifold $(M,\omega)$ in the sense of the last section. In fact that would require $(M,L^m)$ to be balanced for each $m\in\N$, in which case one says that the quantization is \textbf{regular} \cite{CGR2}. The only polarized manifolds known to admit a regular quantization are coadjoint orbits; cf. \cite{ArLo1}. Since the subproduct condition clearly will lead to a stronger kind of quantization (as will be shown in this paper), to be able to use it for any polarized manifold we therefore consider a new kind of quantization.

We shall see that the normalized traces on the $\Bi(\GH_m)$'s converge to a faithful state $\omega$ on $C(M)$, which we call the \textbf{limit state} of the subproduct system $\GH_\bullet$. 
\begin{dfn}\label{projinddef}
A \textbf{projectively induced} quantization of a polarized manifold $(M,L)$ is the datum of a subproduct system $\GH_\bullet\subseteq\GH^{\vee\bullet}$ associated with some choice of inner product $\bra\cdot|\cdot\ket$ on $H^0(M;L)$ %, and the covariant symbol maps $\varsigma^{(m)}:\Bi(\GH_m)\to C(M)$ which it determines.
together with the covariant and contravariant symbol maps specified by $\GH_\bullet$. That is, we define $\varsigma^{(m)}:\Bi(\GH_m)\to C(M)$ as the Berezin covariant symbol map (see e.g. \cite{KMS1}) and we take the Toeplitz map $\breve{\varsigma}^{(m)}:C(M)\to\Bi(\GH_m)$ to be the adjoint of $\varsigma^{(m)}$ with respect to the limit state $\omega$ and the normalized trace on $\Bi(\GH_m)$.
%$$
%\breve{\varsigma}^{(m)}(f)\phi:=\Pi_m(f\phi),\qquad\forall \phi\in\GH_m,
%$$
%where $\Pi_m:L^2(M,\omega_m;L^m,h^m)\to\GH_m$ is the orthogonal projection, and we let $\varsigma^{(m)}:\Bi(\GH_m)\to L^2(M,\omega)$ be its adjoint. 
%The Hermitian metric $h$ on $L$ used to define $L^2(M;L^m)$ is then forced to be the Fubini--Study metric associated with $\bra\cdot|\cdot\ket$ (cf. \cite{Don2}), i.e. the unique Hermitian metric $h=\FS(\bra\cdot|\cdot\ket)$ on $L$ for which any orthonormal basis $Z_1,\dots,Z_n$ for $\GH$ satisfies
%$$
%\sum^n_{k=1}h(Z_k,Z_k)=\bone\in C(M).
%$$
\end{dfn}
\begin{Remark}\label{regularremark}
The terminology in Definition \ref{projinddef} is slightly nonstandard unless $(M,L)$ is balanced; indeed, $(M,L)$ is a balanced polarized manifold (in the sense of \cite{Don1}) if and only if there exists a Hermitian metric $h$ on $L$ such that the quantization $(L,h)$ of $(M,\bar{\pd}\pd\log h)$ is projectively induced for some choice of inner product on $H^0(M;L)$. In that sense the term ``projectively induced" appeared in \cite{CGR1} (it says precisely that the epsilon function discussed there is constant at each level, i.e. that we have a ``regular" quantization in the sense of \cite{CGR2}). In the less restricted sense of Definition \ref{projinddef}, which works for any polarized manifold $(M,L)$ since we do not require the prequantization condition, the families covariant symbols maps associated with projectively induced quantizations were referred to as ``Berezin--Bergman quantizations" in \cite[\S5]{LMS1} (this is the only work we know of where it has been discussed for not necessarily balanced manifolds). In order to choose Toeplitz operators one needs 
\end{Remark}
An explicit formula for the covariant symbol map $\varsigma^{(m)}$ is easy to write down; see Theorem \ref{twocovsymthm}. In order to define Toeplitz operators one also needs to choose a state $\omega:C(M)\to\C$. 
The choice of state $\omega$ is very important if one wants the Toeplitz maps $\breve{\varsigma}^{(m)}$ to give a strict quantization. We will show that there is a canonical choice of $\omega$, appearing as the limit of the normalized traces on the $\Bi(\GH_m)$'s, which produces a strict quantization of $(M,L)$ with covariant symbols determned by the subproduct system $\GH_\bullet$. We will show that a projectively induced quantization gives
\begin{enumerate}[(i)]
\item{$C(M)$ as the ``generalized inductive limit'' of the $C^*$-algebras $\Bi(\GH_m)$, and}
\item{a strict quantization of $(M,L)$.}
\end{enumerate}
In this paper we will mainly discuss (ii) in the case of (quantum) homogeneous manifolds, but due to our results concerning point (i) in the next sections we can fill in the details to obtain (ii) for any polarized manifold. The latter will be very important in future work where more details will be given. 

Both (i) and (ii) could be satisfied without the subproduct condition. However, the inductive system in (i) has very nice properties in the subproduct case which strictly relies on this assumption. Concerning (ii) it is not necessarily the case that the subproduct condition gives something extra. Rather, the important fact is that one can have a strict quantization at the same time as an inductive system.

% Note that $\breve{\varsigma}^{(m)}$ will be defined as adjoint w.r.t. $\omega$ on $C(M)$, and the $*$-operation on $C(M)$ is that of the Fubini--Study metric. Since $\breve{\varsigma}^{(m)}$ also defined w.r.t. $\phi_m$, it does compares $\omega$ and $\phi_m$.  

%We shall see that using projectively induced quantizations, a polarized manifold $(M,L)$ need not be balanced in order to recover $C^\infty(M)$ using Berezin quantization. We stress that the restriction to quantizations with this choice of volume form is not an artifact of the operator-algebraic approach: it is needed if we want the the stronger convergence of Berezin transforms, and in order to recover the algebra $C^\infty(M)$ not just approximately. 

\subsection{The circle bundle}\label{circbundsec}
% Need to use patched Hardy space. So $C^0(\Sb)$ is not $*$-homomorphically represented unless regular quantization.

Recall the Toeplitz quantization maps $\breve{\varsigma}^{(m)}:C(M)\to\Bi(\GH_m)$ (for definiteness and later relevance we will focus on the case of a projectively induced quantization). We can assemble them into a single map
\begin{equation}\label{tottoeplitz}
\breve{\varsigma}:C(M)\to\prod_{m\in\N_0}\Bi(\GH_m),
\end{equation}
where the $C^*$-algebra on the right-hand side is the $C^*$-direct product. 

Let $L^*$ be the dual line bundle of $L$. Under the embedding of $M$ into $\C\Pb^{n-1}$, when $L$ becomes the restriction of the hyperplane line bundle, $L^*$ becomes the restriction of the tautological line bundle. Denote by
\begin{equation}\label{circlebundleeq}
\Sb_M:=\{\zeta\in L^*|\ \|\zeta\|=1\}
\end{equation}
the total space of the associated principal $\Un(1)$-bundle. The $\Un(1)$-action on $\Sb_M$ induces a $\Z$-grading 
$$
C(\Sb_M)=\overline{\bigoplus_{k\in\Z}C(\Sb_M)^{(k)}}^{\|\cdot\|}
$$
of the $C^*$-algebra of continuous functions on $\Sb_M$, with $C(\Sb_M)^{(0)}=C(M)$ and more generally
$$
C(\Sb_M)^{(k)}=\Gamma(M;L^k)
$$
as vector spaces, where $\Gamma(M;L^m)$ is the space of all global continuous sections of $\Li^k$. The strongly $\Z$-graded algebraic structure on $C(\Sb_M)$ comes from the $C(M)$-module structure on each $\Gamma(M,L^m)$ and the tensor operation 
$$
\Gamma(M;L^j)\otimes_{C(\Sb_M)}\Gamma(M;L^k)=\Gamma(M;L^{j+k}).
$$
The $*$-operation on $C(\Sb_M)$ is obtained by endowing each module $\Gamma(M,L^k)$ with the structure of a Hilbert module given by $h^k$, where $h$ is the Fubini--Study Hermitian metric on $L$. That is, the $C^*$-algebra $C(\Sb_M)$ is generated by the homogeneous coordinate functions $Z_1,\dots,Z_n$ on $M\subset\C\Pb^{n-1}$ satisfying
$$
\sum^n_{k=1}Z_kZ_k^*=\bone=\sum^n_{k=1}Z_k^* Z_k.
$$
These ``sphere'' relations are a manifestation of the fact that $\Sb_M$ is nothing but the preimage of $M$ under the map $\Sb^{2n-1}\to\C\Pb^{n-1}$ which defines projective space. Thus we can make an identification
$$
\Sb_M\subset\Sb^{2n-1}.
$$
%The $*$-structure on $C(\Sb_M)$ given by the tensor-product metrics $h^k$ ensures that 

Now let $\omega:C(M)\to\C$ be any faithful state. It extends canonically to a faithful state on $C(\Sb_M)$ by setting
$$
\omega(Z_\mathbf{j}Z_\mathbf{k}^*):=0,\qquad \forall\mathbf{j},\mathbf{k}\in\F_n^+,\ |\mathbf{j}|\ne|\mathbf{k}|. 
$$
As just mentioned, we have $C(\Sb_M)=\overline{\bigoplus_{k\in\Z}\Gamma(M;L^k)}^{\|\cdot\|}$ as a $C^*$-algebra if we endow each $\Gamma(M,L^k)$ with the metric $h^k$. The GNS space of $\omega$ is the $L^2$-space $L^2(\Sb,\omega)$ of $\omega$-square-integrable functions on $\Sb$, and the closed subspace $H^2(\Sb,\omega)$ spanned by the coordinate functions $Z_1,\dots,Z_n$ on $\Sb$ is a Hardy-type space. 

Both $H^2(\Sb,\omega)$ and the Fock space $\GH_\N$ are completions of the homogeneous coordinate ring $\Ai=\bigoplus_{m\in\N_0}\Ai_m$ of $\M\subset\C\Pb^{n-1}$ in such a way that elements of $\Ai_m$ are orthogonal to elements of $\Ai_l$ whenever $m\ne l$. If $\M=\G/\K$ is a coadjoint orbit and $\GH_\bullet$ is a regular quantization of $\G/\K$ then one simply has (see the proof of Lemma \ref{Qsubnormlemma})
$$
\bra\psi|\varphi\ket_{L^2}=\frac{1}{\dim\GH_m}\bra\psi|\varphi\ket.
$$
So in this case Fock space can be embedded into $H^2(\Sb_M)$,
$$
\GH_\N=\overline{\bigoplus_{m\in\N_0}\GH_m}^{\bra\cdot|\cdot\ket}\subset H^2(\Sb_M)\subset L^2(\Sb_M).
$$
In general the relation between $\GH_\N$ and $H^2(\Sb_M,\omega)$ is more involved. Hence $H^2(\Sb_M,\omega)$ is far from sitting inside $H^2(\Sb^{2n-1},\omega_{\rm FS})$ as a coinvariant subspace. For this reason, the Fock space $\GH_\N$ is better suited for studying how well the geometric properties of $M$ are compatible with an embedding $M\hookrightarrow\C\Pb^{n-1}$ and the associated pullback quantities. 

Suppose now that $\omega$ is the state on $C(\Sb_M)$ associated with a Kähler form on $M$ (denoted by the same symbol $\omega$), and let $\Pi:L^2(\Sb_M,\omega)\to\GH_\N$ be the orthogonal projection. Then, with $\breve{\varsigma}$ as in \eqref{tottoeplitz}, we have
$$
\breve{\varsigma}(f)\psi=\Pi (f\psi),\qquad \forall \psi\in\GH_\N
$$
if we identify $f\in C(M)$ with the multiplication operator it defines on $L^2(\Sb_M,\omega)$. Then $f$ is just the \textbf{contravariant symbol} of the operator $\breve{\varsigma}(f)$ in the general sense of \cite{Bere3}. The interior of $\Sb_\M$ (the disk bundle) is a bounded symmetric domain and Berezin quantization on spaces such as $\Sb_M$ has been studied even more than in the setting of compact Kähler manifolds, see e.g. \cite{UnUp1}, \cite{Bere2}. 

The circle bundle comes with the structure of a Cauchy--Riemann manifold, the details of which can be found in \cite[§2]{Schl1}, \cite[§2]{Zeld1}. 

We can define the Berezin transform and the covariant symbol of an operator on $\GH_\N$, just as we did on the components $\GH_m$, using the fact that $\GH_\N$ is a reproducing kernel Hilbert space. Again the covariant symbol map $\varsigma$ is the adjoint of $\breve{\varsigma}$. In the noncommutative setting we will just calculate the adjoint of $\breve{\varsigma}$ and take that as the definition of the covariant symbol map.

%Now let us instead let $h_m$ be a Hermitian metric on $L^m$ for which the inner product on $\GH_m$ is recovered as
%$$
%\bra\phi|\psi\ket_{\GH_m}=(\dim\GH_m)\omega(h_m(\phi,\psi)).
%$$
%The Hilbert space $\GK_m$

%Products of the form $Z_{j_1}\cdots Z_{j_m}Z_{k_l}^*\cdots Z_{k_1}^*$ constitute a basis for the $L^2$-space $L^2(\Sb_M)$ 
%(see \cite[§5.1]{Schl1} for the definition of the measure on $\Sb_M$) and we define the ``Hardy space" $H^2(\Sb_M)$ to the be subspace of $L^2(\Sb_M)$ spanned by the products $Z_{j_1}\cdots Z_{j_m}$ for all $m\in\N_0$. The Hardy space can therefore be obtained as the Hilbert space direct sum
%$$
%H^2(\Sb_M)=\overline{\bigoplus_{m\in\N_0}\GH_m}^{\bra\cdot|\cdot\ket_{L^2}},
%$$
%just as the Fock space $\GH_\N$ associated with the subproduct system $\GH_\bullet$. 

\subsection{Singular varieties}
We have seen that Berezin quantization of quantizable Kähler manifolds is really the quantization of smooth projective varieties. It was suggested in \cite{Schl2} that it may be possible to quantize also singular (non-smooth) projective varieties in the same fashion. We shall see that this is in fact so: it will be covered by the constructions in the next two sections, as the case when the subproduct system $\GH_\bullet$ is commutative (Corollary \ref{strictsingcor}).

\section{Inductive limits}\label{indlimsec}
Recall \cite[Def. 1.1.11]{ArMa1} that a sequence $(\Bi_m)_{m\in\N_0}$ of $*$-algebras $\Bi_m$ forms an ``inductive system" if there are $*$-homomorphisms $\iota_{m,l}:\Bi_m\to\Bi_{m+l}$ for each $m\leq l$, and that the \textbf{algebraic inductive limit} of such a sequence is the algebra $\bigcup_{m\in\N_0}\Bi_m$ obtained as the quotient of the algebra of eventually constant sequences of elements in the $\Bi_m$'s,
$$
\Big\{(b_j)_{j\in\N_0}\in \prod_{j\in\N_0}\Bi_j\Big|\exists m\in\N_0\text{ such that } b_j=b_m\text{ for all } j\geq m\Big\},
$$ 
by its ideal of sequences $(b_j)_{j\in\N_0}$ which are eventually $0$. If each $\Bi_m$ is a $C^*$-algebra, $\bigcup_{m\in\N_0}\Bi_m$ can be completed in a canonical $C^*$-norm to obtain a $C^*$-algebra which, if the $\iota_{m,l}$'s are injective, can be identified with the non-disjoint union $\overline{\bigcup_{m\in\N_0}\Bi_m}^{\|\cdot\|}$ \cite[§1.2]{ArMa1}.

It was observed in \cite{Hawk1} that the sequence of algebras $\Bi_m:=\Bi(\GH_m)$ arising in quantization has a structure resembling that of an inductive system, although the map from $\Bi_m$ to $\Bi_{m+1}$ is not a homomorphism in the category of $C^*$-algebras. If we want to obtain a $C^*$-algebra $C(M)$ of continuous functions on a manifold as an inductive limit of finite-dimensional matrix algebras, then requiring $\Bi_m\subset\Bi_{m+1}$ says by definition that $C(M)$ is an AF algebra. This 
forces $M$ to be totally disconnected. Hence we must relax the notion of inductive limit.  

\subsection{Relaxed definition of inductive limits}
Blackadar and Kirchberg introduced a more general inductive-limit-type construction \cite{BlKi1}. Although never pointed out in the literature, the system of finite-dimensional $C^*$-algebras $\Bi(\GH_m)$ obtained from a projective quantization $(M,\omega,L)$ fits perfectly into their framework. This is most apparent in \cite{Hawk1} where similar notions were introduced independently. We will follow the notation of \cite{Hawk1} as closely as possible. 

\begin{Notation}\label{notatdirectprod}
If $\Bi_\bullet=(\Bi_m)_{m\in\N_0}$ is a sequence of $C^*$-algebras, we write
$$
\Gamma_b(\Bi_\bullet)=\prod_{m\in\N_0}\Bi_m
$$
for the full $C^*$-\textbf{direct product} of the $\Bi_m$'s, i.e. the set of sequences $X_\bullet=(X_m)_{m\in\N_0}$ of elements $X_m\in\Bi_m$ with finite supremum norm 
$$
\|X_\bullet\|:=\sup_{m\in\N_0}\|X_m\|_{\Bi_m}<\infty.
$$
The multiplication and $*$-operation in $\Gamma_b(\Bi_m)$ is pointwise. We also write
$$
\Gamma_0(\Bi_\bullet)=\overline{\bigoplus_{m\in\N_0}\Bi_m}
$$
for the $C^*$-\textbf{direct sum}, the closed two-sided ideal in $\Gamma_b(\Bi_\bullet)$ consisting of the sequences converging to zero in norm. We simply write $\Gamma_b:=\Gamma_b(\Bi_\bullet)$ etc. if it is clear which sequence $\Bi_\bullet$ it concerns. We let
$$
\pi:\Gamma_b\to\Gamma_b/\Gamma_0
$$
be the quotient map
\end{Notation}
Since we will only deal with a special kind of the ``generalized inductive systems" defined in \cite{BlKi1} (namely the ``NF" ones), we will simply refer to them as ``inductive systems". See also \cite[§11]{BrOz1}.
\begin{dfn}\label{indlimdef}
An \textbf{inductive system} is a sequence $(\Bi_\bullet,\iota_\bullet)$ of full matrix algebras $\Bi_m=\Mn_{k(m)}(\C)$ and unital completely positive maps $\iota_{m,l}:\Bi_m\to\Bi_{m+l}$ for $l\geq m$ (with $\iota_{m,m}:=\id$) satisfying 
\begin{equation}\label{coherecond}
\iota_{m,l}=\iota_{r,l}\circ\iota_{m,r},\qquad\text{if }m\leq r\leq l
\end{equation}
and which are \textbf{asymptotically multiplicative} in the sense that for all $A,B\in\Bi(\GH_m)$, $\ep>0$, there are $r\leq l$ such that
\begin{equation}\label{asymptmult}
\|\iota_{r,l}(\iota_{m,r}(A)\iota_{m,r}(B))-\iota_{m,l}(A)\iota_{m,l}(B)\|<\ep.
\end{equation}
The \textbf{inductive limit} of an inductive system $(\Bi_\bullet,\iota_\bullet)$ is the $C^*$-algebra 
$$
\Bi_\infty\subset\Gamma_b(\Bi_\bullet)/\Gamma_0(\Bi_\bullet)
$$
generated by the elements
\begin{equation}\label{indlimberezinatX}
\varsigma^{(m)}(A):=\pi((\iota_l(A))_{l\geq m}),\qquad A\in\Bi(\GH_m)
\end{equation}
for all $m\in\N_0$, where $\pi:\Gamma_b\to\Gamma_b/\Gamma_0$ is the quotient map.
\end{dfn}
\begin{Remark}[Norm] A norm on the quotient $C^*$-algebra $\Gamma_b/\Gamma_0$ is given by
$$
\|\pi(X_\bullet)\|=\limsup_{m\to\infty}\|X_m\|,\qquad \forall X_\bullet\in\Gamma_b,
$$
and this norm satisfies the $C^*$-identity, hence it is the unique $C^*$-norm on $\Gamma_b/\Gamma_0$. Moreover, since the $\iota_{m,l}$'s are norm-decreasing, 
$$
\limsup_{m\to\infty}\|\varsigma^{(m)}(A)\|=\lim_{l\to\infty}\|\iota_{m,l}(A)\|,\qquad\forall A\in\Bi(\GH_m),
$$
so the norm on $\Bi_\infty$ is just the ``norm-at-infinity" of $\pi^{-1}(\Bi_\infty)$. 
\end{Remark}

It follows that the maps 
\begin{equation}\label{indlimberezin}
\varsigma^{(m)}:\Bi_m\to\Bi_\infty
\end{equation}
are completely positive, and we refer to $\varsigma^{(m)}$ as the \textbf{covariant Berezin symbol map} at level $m$. The motivation for this terminology will become clear below. Due to \eqref{coherecond}, the covariant symbol maps satisfy
\begin{equation}\label{compatibcovsymb}
\varsigma^{(l)}\circ\iota_{m,l}=\varsigma^{(m)},\qquad \forall m\leq l\in\N_0.
\end{equation}
We may say that a sequence $A_\bullet\in\Gamma_b$ is \textbf{eventually constant} under $\iota_{\bullet,\bullet}$ if there is a large enough $m\in\N_0$ such that $A_l=\iota_{r,l}(A_r)$ for all $l\geq r\geq m$. Then $\Bi_\infty$ is the image under $\pi$ of the norm closure of the algebra of eventually constant sequences. 
\begin{Remark}[Asymptotic multiplicativity]\label{asmultrem}
The condition \eqref{asymptmult} is chosen precisely to ensure that $\varsigma^{(m)}(A)\varsigma^{(m)}(B)$ belongs to the $C^*$-algebra $\Bi_\infty$ for all $A,B\in\Bi(\GH_m)$, without requiring it to be close to $\varsigma^{(m)}(AB)$. Conversely, if $\varsigma^{(m)}(A)\varsigma^{(m)}(B)$ belongs to $\Bi_\infty$ for all $A$ and $B$ then each $\varsigma^{(m)}(A)\varsigma^{(m)}(B)$ is an eventually constant sequence, so \eqref{asymptmult} must hold.
\end{Remark}
\begin{Remark}[Continuous fields]\label{contfieldrem}
Suppose that $\Bi(\bullet)$ is a continuous field of matrix algebras over $\N_0\cup\{\infty\}$. For each $A\in\Bi(\infty)$, there is a continuous section $x\to A(x)$ of $\Bi(\bullet)$ with $A(\infty)=A$, and this section defines an element of $\prod_{m\in\N_0}\Bi(m)$. Two sections evaluating to $A$ at $\infty$ differ by an element in $\bigoplus_{m\in\N_0}\Bi(m)$. Hence \cite[Prop. 2.2.3]{BlKi1}
$$
\Bi(\infty)\subset \Big(\prod_{m\in\N_0}\Bi(m)\Big)\Big/\Big(\bigoplus_{m\in\N_0}\Bi(m)\Big).
$$
In fact, a $C^*$-algebra is an inductive limit (in the sense of Definition \ref{indlimdef}) if and only if it is a nuclear separable $C^*$-algebra which is of the form $\Bi(\infty)$ for some continuous field of matrix algebras over $\N_0\cup\{\infty\}$ \cite[Thm. 5.2.2]{BlKi1}.
\end{Remark}

\begin{Remark}[Quasi-diagonality]\label{quasidiagrem}
If a $C^*$-algebra $\Bi_\infty$ is an inductive limit in the sense of Definition \ref{indlimdef}, there exists a short-exact sequence
$$
0\longto \Ki\longto\Di\longto\Bi_\infty
$$
which is an ``essential quasi-diagonal extension" of $\Bi_\infty$, meaning that $\Di$ is a $C^*$-algebra of quasi-diagonal operators containing the $C^*$-algebra $\Ki$ of compact operators as an essential ideal. In fact, such an extension of a $C^*$-algebra $\Bi_\infty$ exists if and only if $\Bi_\infty$ can be embedded into $\Gamma_b(\Bi_\bullet)/\Gamma_0(\Bi_\bullet)$ for some sequence of full matrix algebras $\Bi_m$ \cite{BlKi1}. 
\end{Remark}
\begin{Remark}[Nuclearity]\label{nucremark}
Since $\Bi_\infty$ is nuclear, the Choi--Effros lifting theorem \cite[IV.2.3.4]{Blac1} says that the identity mapping $\id:\Bi_\infty\to\Bi_\infty$ can be lifted to a unital completely positive map $\breve{\varsigma}:\Bi_\infty\to\Gamma_b$ such that if
$$
\breve{\varsigma}^{(m)}(f):=\breve{\varsigma}(f)p_m\in\Bi(\GH_m),\qquad \forall f\in\Bi_\infty,
$$
then the sequence $\varsigma^{(m)}\circ\breve{\varsigma}$ converges in the point-norm topology to $\id:\Bi_\infty\to\Bi_\infty$. We shall calculate $\breve{\varsigma}$ and its inverse $\varsigma$ explicitly in §\ref{adjointtotsec}. 
\end{Remark}

\subsection{Inductive limits from subproduct systems}
\begin{Lemma}\label{indlimasshiftslemma}
Let $\GH_\bullet$ be a subproduct system and define completely positive maps $\iota_{m,l}:\Bi_m\to\Bi_{l}$ by
\begin{equation}\label{indlimconnectionmaps}
\iota_{m,l}(A):=p_l(A\otimes\bone_{\GH^{\otimes (l-m)}})p_l,\qquad \forall A\in\Bi(\GH_m),
\end{equation}
where $p_l:\GH^{\otimes l}\to\GH_l$ is the projection.
Then for all $A\in\Bi(\GH_m)$ we have the formulae
\begin{align}
\iota_{m,l}(A)&=\sum_{|\mathbf{k}|=l-m}R_\mathbf{k}AR_{\mathbf{k}}^*\big|_{\GH_{l}}\label{flipinduct}
\\&=\sum_{|\mathbf{j}|=m=|\mathbf{k}|}A_{\mathbf{j},\mathbf{k}}S_{\mathbf{j}}S_{\mathbf{k}}^*\big|_{\GH_{l}}\label{indlimasshifts}
\end{align}
where $R_\mathbf{k}$ is the right shift by the vector $e_\mathbf{k}$ as in \eqref{rectoshift} and $A_{\mathbf{j},\mathbf{k}}:=\bra e_\mathbf{j}|Ae_\mathbf{k}\ket$. 
\end{Lemma}
\begin{proof}
Formula \eqref{indlimasshifts} is immediate from
\begin{align*}
S_{\mathbf{j}}S_{\mathbf{k}}^*|_{\GH_l}&=p_l(|p_me_{\mathbf{j}}\ket\bra p_me_{\mathbf{k}}|\otimes\bone_{\GH_{l-m}})p_l.
\end{align*}
Similarly, the expression \eqref{flipinduct} is deduced from straightforward calculations.
\end{proof}
\begin{thm}
 Every subproduct system $\GH_\bullet$ defines a generalized inductive system $(\Bi_\bullet,\iota_{\bullet,\bullet})$ by setting $\Bi_m:=\Bi(\GH_m)$ and letting $\iota_{m,l}:\Bi_m\to\Bi_l$ be as in \eqref{indlimconnectionmaps}.
\end{thm}
\begin{proof} It is clear that each $\iota_{m,l}:\Bi_m\to\Bi_l$ is unital and completely positive. For $m\leq r\leq l$ and $A\in\Bi(\GH_m)$ we have
\begin{align*}
\iota_{r,l}\circ\iota_{m,r}(A)&=\iota_{m,l}(p_r(A\otimes\bone_{\GH^{\otimes (r-m)}})p_r)=p_l(p_r(A\otimes\bone_{\GH^{\otimes (r-m)}})p_r\otimes\bone_{\GH^{\otimes(l-r)}})p_l
\\&=p_l(A\otimes\bone_{\GH^{\otimes(l-m)}})p_l,
\end{align*}
where the last equality is due to \eqref{subprodprojdom}. That is, the coherence condition \eqref{coherecond} holds. 

It remains to show that $\iota_{\bullet,\bullet}$ is asymptotically multiplicative, i.e. that it satisfies \eqref{asymptmult}. We have
\begin{align*}
\iota_{r,l}(\iota_{m,r}(A)\iota_{m,r}(B))&=\iota_{r,l}\big(p_r(A\otimes\bone)p_r(B\otimes\bone)p_r\big)
\\&=p_l\big(\big(p_r(A\otimes\bone)p_r(B\otimes\bone)p_r\big)\otimes\bone\big)p_l
\\&=p_l\big(\big((A\otimes\bone)p_r(B\otimes\bone)\big)\otimes\bone\big)p_l,
\end{align*}
so it is the failure of $A\otimes\bone$ to commute with the projection $p_r$ which spoils multiplicativity. 
It seems hard to show directly from norm estimates that the maps \eqref{indlimconnectionmaps} satisfy the asymptotic multiplicativity condition \eqref{asymptmult}. We shall instead obtain that by showing that the set of elements \eqref{indlimberezinatX} forms an algebra. 

Consider the positively graded algebraic part $\Ri$ of the Toeplitz algebra $\Ti_\GH$, i.e. the $\N_0$-graded ring
$$
\Ri=\bigoplus_{m\in\N_0}\Ri_m:=\bigoplus_{m\in\N_0}\Ti_\GH^{(m)}.
$$
Denote by $\Gr(\Ri)$ the Abelian category of $\Z$-graded right $\Ri$-modules, with morphisms the grading-preserving morphisms in the category of $\R$-modules. Write $\Ri_{\geq m}$ for the $\Ri$-module $\bigoplus_{l\geq m}\Ri_l$. 
Then it is straightforward to see that
$$
\End_{\Gr(\Ri)}(\Ri_{\geq m})=\Bi(\GH_m)
$$
as rings. Indeed, we have the left $\Ri$-action on $\Ri_{\geq m}$ and the grading-preseving elements can all be obtained by taking linear combinations of the elements $S_\mathbf{j}S_\mathbf{k}^*$ with $|\mathbf{j}|=m=|\mathbf{k}|$. Namely, the operator $S_\mathbf{j}S_\mathbf{k}^*$ on $\GH_\N$ is the direct sum $T\oplus 0$ of an operator $T\in\in\Bi(\GH_{\geq m})$ and the
zero operator $0\in\Bi(\GH_{<m})$. So $S_\mathbf{j}S_\mathbf{k}^*$ can be viewed as an operator of the $\Ri$-module $\Ri_{\geq m}$, and $S_\mathbf{j}S_\mathbf{k}^*$ is right $\Ri$-linear because it acts by multiplication from the left. 
Now the family $(S_\mathbf{j}S_\mathbf{k}^*)_{|\mathbf{j}|=m=|\mathbf{k}|}$ forms an overcomplete set of matrix units in $\Bi(\GH_m)$, so their $\C$-linear span identifies with $\Bi(\GH_m)$. We have an inductive system
$$
\End_{\Gr(\Ri)}(\Ri_{\geq m})\ni X|_{\GH_{\geq m}}\to Y|_{\Ri_{\geq l}}\in\End_{\Gr(\Ri)}(\Ri_{\geq l}),\qquad \forall m\leq l\in\N_0
$$
obtained by restriction to shorter tails $\Ri_{\geq l}\subset\Ri_{\geq m}$. From Equation \eqref{indlimasshifts} we see that this is precisely the algebraic inductive system underlying $\iota_{m,l}:\Bi(\GH_m)\to\Bi(\GH_l)$, under the identification $\End_{\Gr(\Ri)}(\Ri_{\geq m})=\Bi(\GH_m)$. 

 By \cite[\S IX.1]{Sten1} (see also \cite[Example 5.4]{ArZh1}), the algebraic inductive limit 
$$
^{0}\Bi_\infty=\lim_{m\to\infty}\End_{\Gr(\Ri)}(\Ri_{\geq m})
$$
is a ring (and an algebra over $\C$ since the maps $\iota_{m,l}$ are $\C$-linear). Therefore the norm closure $\Bi_\infty$ of $^{0}\Bi_\infty$ is an algebra as well. In particular, the asymptotic multiplicativity condition \eqref{asymptmult} is satisfied. 
\end{proof}
Thus, subproduct systems give rise to generalized inductive limits of $C^*$-algebras with the special property that the algebraic direct limit $^{0}\Bi_\infty$ is already an algebra (no need for norm closure). Still, the asymptotic multiplicativity condition \eqref{asymptmult} cannot be formulated in a weaker fashion even for subproduct systems, since the set of eventually constant sequences under $\iota_{\bullet,\bullet}$ is only an algebra after taking norm closure. 

Our aim is to identify the inductive limit $\Bi_\infty$ with the Cuntz--Pimsner core $\Oi_\GH^{(0)}$. For that we need a lemma.
\begin{Lemma}\label{Toepevconst}
Let $\pi^{-1}(\Bi_\infty)$ be the norm closure of the subset of $\prod_m\Bi(\GH_m)$ consisting of sequences which are eventually constant under the coherent system $\iota_{\bullet,\bullet}$ defined by \eqref{indlimconnectionmaps}. Then $\pi^{-1}(\Bi_\infty)$ coincides with the normally ordered part of the Toeplitz core $\Ti_\GH^{(0)}$. Hence the normally ordered part of $\Ti_\GH^{(0)}$ is an algebra, and must coincide with all of $\Ti_\GH^{(0)}$, so
$$
\pi^{-1}(\Bi_\infty)=\Ti_\GH^{(0)}.
$$
\end{Lemma}
In this way we have proven Lemma \ref{Lemmanormorder}. 
\begin{proof}
% The right shift $R_r$ commutes with $S_j^*$ and $S_k$ outside the vacuum subspace $\GH_0$. It follows from \eqref{flipinduct} that both $S_j^*S_k$ and $S_kS_j^*$ are constant under $\iota_{\bullet,\bullet}$. \textbf{Since we know this is false, it must be that $R^*_j$ does never commute with $S_k$ even outside the vacuum.}
%\begin{align*}
%\[R_k,S_j^*](\psi\otimes\phi)=\psi_jR_k\phi-S_j^*(\psi\otimes\phi\otimes e_k)=\psi_jp_{m+1}(\phi\otimes e_k)-\psi_jp_{m+1}(\phi\otimes e_k)=0.
%\\end{align*}
%\No we cannot calculate $S_j^*$ explicitly. 
%\\begin{align*}
%\\bra\phi|[R_k,S_j^*]\psi\ket&=\bra R_k^*\phi|S_j^*\psi\ket-\bra S_j\phi|R_k\psi\ket=\overline{\bra S_j^*\psi| %\R_k^*\phi\ket}-\bra S_j\phi|R_k\psi\ket
%\\\&=\bra R_k^*\phi|S_j^*\psi\ket-\bra S_j\phi|R_k\psi\ket
%\\end{align*}
It is clear that every normally ordered element of $\Ti_\GH^{(0)}$ defines a sequence $A_\bullet=(A_m)_{m\in\N_0}$ of operators $A_m\in\Bi(\GH_m)$ with $\iota_{m,l}(A_m)=A_l$ for sufficiently large $m\leq l$. For example, 
$$
S_jS_k^*=(S_jS_k^*|_{\GH_m})_{m\in\N_0}\in\prod_{m\in\N_0}\Bi(\GH_m).
$$
Suppose now that $A=(A_m)_{m\in\N_0}$ is any element of $\prod_{m\in\N_0}\Bi(\GH_m)$ which is eventually constant. Since $\Bi(\GH_m)$ is contained in $\Ti_\GH^{(0)}$ for each $m$, we may for simplicity just as well look at the case where $\iota_{r,l}(A_r)=A_l$ for all $r\leq l$ for some $r\in\N_0$ while $A_m=0$ for $m\leq r$. Then \eqref{indlimasshifts} in Lemma \ref{indlimasshiftslemma} shows that $A$ is a combination of shift operators. 

From the fact that $\Bi_\infty$ is an algebra we have that $\pi^{-1}(\Bi_\infty)$ is an algebra, whence the last statement. 
\end{proof}
%Now we can complete the proof of the statement that the maps $\iota_{m,l}$ form an inductive system. Namely, we have seen that the norm closure of the subset of $\Gamma_b$ consisting of eventually constant sequences forms an algebra (namely $\Ti^{(0)}_\GH$). Since $\Gamma_0$ is an ideal in $\Gamma_b$, this is equivalent to the statement that the $\varsigma^{(m)}(\Bi_m)$'s form an algebra, which is in turn equivalent to the asymptotic multiplicativity of the  $\iota_{m,l}$'s (see Remark \ref{asmultrem}).

\begin{Remark}
Let $\Bi(\bullet)$ be the continuous field of $C^*$-algebras over $\N_0\cup\{\infty\}$ such that the fiber over $m\in\N$ is $\Bi(m)=\Bi(\GH_m)$ and the fiber over $\infty$ is $\Bi(\infty)=\Bi_\infty$, the inductive limit (cf. Remark \ref{contfieldrem}). Then Lemma \ref{Toepevconst} says that $\Ti_\GH^{(0)}$ is the algebra of continuous sections of this field. For commutative case see also \cite[Thm. 3.3]{Hawk2}. 
\end{Remark}

\begin{thm}\label{indlimpimsnerthm}
Let $\GH_\bullet$ be a subproduct system and let $\Oi_\GH^{(0)}$ denote the $\Un(1)$-invariant part of the Cuntz--Pimsner algebra of $\GH_\bullet$. Then we have
$$
\Oi_\GH^{(0)}\cong\Bi_\infty,
$$
where the right-hand side is the inductive limit defined by the inductive system $\iota_{\bullet,\bullet}$ in \eqref{indlimconnectionmaps}.
\end{thm}
\begin{proof} The Toeplitz core $\Ti_\GH^{(0)}$ is the norm closure of linear combinations of elements of the form $S_\mathbf{j}S_\mathbf{k}^*$ with $|\mathbf{j}|=|\mathbf{k}|$ as well as their products with the vacuum projection $p_0=|\Omega\ket\bra\Omega|$. The elements which are products with $|\Omega\ket\bra\Omega|$ belong to $\Gamma_0=\Ki\cap\Ti_\GH^{(0)}$. Hence, 
$$
\Bi_\infty=\Ti_\GH^{(0)}/\Gamma_0=\Ti_\GH^{(0)}/(\Ki\cap\Ti_\GH^{(0)})=\Oi^{(0)}_\GH,
$$
as asserted. 
\end{proof}
\begin{Remark} The quasi-diagonal extension of $\Bi_\infty$ mentioned in Remark \ref{quasidiagrem} can now be taken as $\Di=\Ti^{(0)}_\GH+\Ki$ (cf. \cite[V.4.2.16]{Blac1}).
\end{Remark}

\subsection{Cuntz--Pimsner algebras from inductive limits}\label{pimsindlimsec}
For $m>0$, let $\GH_{-m}:=\overline{\GH}_{m}$ denote the conjugate Hilbert space of $\GH_{m}$. For all $k\in\Z$ we can consider the $\Bi(\GH_m)$-module
$$
\Ei_m^{(k)}:=\Bi(\GH_m,\GH_{m+k}),
$$
and the maps $\iota_{m,l}^{(k)}:\Ei_m^{(k)}\to \Ei_l^{(k)}$ defined by
$$
\iota_{m,l}^{(k)}(X):=\sum_{|\mathbf{r}|=l-m}R_\mathbf{r}XR_\mathbf{r}^*\big|_{\GH_l},\qquad\forall X\in\Ei_m^{(k)}.
$$
We define the $C^*$-algebras $\Gamma_b(\Ei_\bullet^{(k)})$ and $\Gamma_0(\Ei_\bullet^{(k)})$ in the same way as in Notation \ref{notatdirectprod} and we denote by $\pi^{(k)}:\Gamma_b(\Ei_\bullet^{(k)})\to\Gamma_b(\Ei_\bullet^{(k)})/\Gamma_0(\Ei_\bullet^{(k)})$ the quotient map.

Define $^{0}\Ei^{(k)}$ to be the vector space consisting of all elements of the form
$$
\varsigma^{(m,k)}(X):=\pi^{(k)}\big(\iota_{m,l}^{(k)}(X)\big)_{l\geq m}\big),\qquad X\in\Ei_m^{(k)}
$$
for all $m\in\N_0$. In particular, $^{0}\Ei^{(0)}\equiv {^{0}\Bi_\infty}$ is the algebraic part of $\Bi_\infty\equiv\Ei^{(0)}\cong\Oi_\GH^{(0)}$. Each $^{0}\Ei^{(k)}$ is a module over $^{0}\Bi_\infty$. The linear span of 
$$
\Bi_\infty{^{0}\Ei^{(k)}}:=\{f\psi|f\in\Bi_\infty,\ \psi\in{^{0}\Ei^{(k)}}\}
$$
is a module over $\Bi_\infty$, which we denote by $\Ei^{(k)}$.

\begin{thm} The Cuntz--Pimsner algebra $\Oi_\GH$ is isomorphic to the $C^*$-algebra generated by $\Ei^{(1)}$ and $\Bi_{\infty}$. It allows the decomposition 
$$
\Oi_\GH\cong\overline{\bigoplus_{k\in\Z}\Ei^{(k)}}^{\|\cdot\|}
$$
and $\Ei^{(k)}\cong\Oi_\GH^{(k)}$ is the spectral subspace for the gauge action corresponding to $k\in\Z$. 
\end{thm}
\begin{proof} The vector space $\Bi(\GH_m,\GH_{m+k})$ has an overcomplete basis given by the operators $S_\mathbf{k}|_{\GH_m}$ for all $\mathbf{k}\in\F_n^+$ with $|\mathbf{k}|=k$. In particular, $\Bi(\GH_m,\GH_{m+1})$ is spanned by $S_j|_{\GH_m}$ for $j=1,\dots,n$. Recalling that $S_j$ is the shift by the basis vector $e_j\in\GH$, we see that
\begin{align*}
\iota_{m,l}^{(k)}(S_j|_{\GH_m})&=\sum_{|\mathbf{r}|=l-m}R_\mathbf{r}S_jR_\mathbf{r}^*\big|_{\GH_l}
\\&=\sum_{|\mathbf{r}|=l-m}S_jR_\mathbf{r}R_\mathbf{r}^*\big|_{\GH_l}=S_j|_{\GH_l}.
\end{align*}
We can identify a sequence $X_\bullet=(X_m)_{m\in\N_0}$ of operators $X_m\in\Bi(\GH_m,\GH_{m+k})$ with an operator on Fock space $\GH_\N$. The effect of the quotient map $\pi^{(k)}$ on such a sequence $X_\bullet$ is to take it to its image in the Calkin algebra $\Bi(\GH_\N)/\Ki$. 
From \eqref{compatibcovsymb} we therefore have (for $m\geq 1$)
$$
\varsigma^{(m,1)}(S_j|_{\GH_m})=\varsigma^{(1,1)}(S_j|_{\GH})=\pi^{(1)}\big((S_j\big|_{\GH_m})_{m\in\N_0}\big)=\pi^{(1)}(S_j)=Z_j,
$$
where $Z_1,\dots,Z_n$ are the generators of $\Oi_\GH$. Similarly one gets that $\varsigma^{(l,k)}(S_\mathbf{k}|_{\GH_m})$ is just $Z_\mathbf{k}$ for all $\mathbf{k}\in\F_n^+$ with $|\mathbf{k}|=k$ and all $l\geq m$. The adjoints $S_\mathbf{k}^*$ define elements of $\Gamma_b(\Ei^{(-m)}_\bullet)$ for $|\mathbf{k}|=m$. So $\Ei^{(k)}\cong\Oi_\GH^{(k)}$ holds for all $k\in\Z$.%, as sets. 

%The maps $\iota_{m,l}^{(k)}:\Ei_m^{(k)}\to \Ei_l^{(k)}$ form an inductive system of sets but not an inductive system of modules. We need to show that the module multiplication 
%$$
%X\cdot A:=XA,\qquad\forall A\in\Bi(\GH_m),\ X\in\Bi(\GH_m,\GH_{m+k})
%$$
%is asymptotically intertwined by the maps $\iota_{m,l}^{(k)}$. Suppose $k>0$. Make the identification $\Bi(\GH_m,\GH_{m+k})=\Bi(\GH_m)\otimes\C^{n_{m+k}}$ and identify $\bigoplus_{m\in\N_0}\Bi(\GH_m)\otimes\C^{n_{m+k}}$ with a closed subspace $\GH_\N^{L^k}$ of the vector-valued Fock space $\GH_\N\otimes\C^{n_k}$. Then $\GH_\N^{L^k}$ is invariant under the diagonal action of $R_1,\dots,R_n$ on $\GH_\N\otimes\C^{n_k}$, and so the projection $P_{L^k}$ onto $\GH_\N^{L^k}$ satisfies (see e.g. \cite[\S4]{Pop7})
%$$
%\sum^n_{k=1}R_kP_{L^k}R_k^*\leq P_{L^k}.
%$$

\end{proof}

\subsection{Formulas for covariant symbols}
Our discussion about inductive limits associated to subproduct system has been based on shift operators on Fock space. We now observe that what we are doing is in fact a generalization of Berezin quantization. First we show that there is a very simple expression for the maps $\varsigma^{(m)}$. 
\begin{thm}\label{twocovsymthm}
Let $Z_1,\dots,Z_n$ be the images of the shifts of the subproduct system $\GH_\bullet$. Then the covariant symbol map $\varsigma^{(m)}:\Bi(\GH_m)\to\Oi_\GH^{(0)}$ defined in \eqref{indlimberezinatX} can be expressed as
$$
\varsigma^{(m)}(A)=\sum_{|\mathbf{j}|=m=|\mathbf{k}|}A_{\mathbf{j},\mathbf{k}}Z_{\mathbf{j}}Z_{\mathbf{k}}^*.
$$
\end{thm}
\begin{proof}
From \eqref{indlimberezinatX} we see that we need to express $\iota_m$ in terms of the Toeplitz operators $S_1,\dots,S_n$; applying $\pi$ transforms these into $Z_1,\dots,Z_n$. But that is easily done using Lemma \ref{indlimasshiftslemma}: the ``second quantization" of $A$,
$$
\sum_{|\mathbf{j}|=m=|\mathbf{k}|}A_{\mathbf{j},\mathbf{k}}S_{\mathbf{j}}S_{\mathbf{k}}^*\in\Bi(\GH_\N),
$$
acts as $\iota_{m,l}(A)$ on $\GH_l$ for $l\geq m$ and as $0$ on $\GH_l$ for $l<m$. Applying the quotient map $\pi$ to it, we obtain
$$
\pi\Big(\sum_{|\mathbf{j}|=m=|\mathbf{k}|}A_{\mathbf{j},\mathbf{k}}S_{\mathbf{j}}S_{\mathbf{k}}^*\Big)=\pi((\iota_{m,l}(A))_{l\geq m})=\varsigma^{(m)}(A),
$$
and on the other hand,
$$
\pi\Big(\sum_{|\mathbf{j}|=m=|\mathbf{k}|}A_{\mathbf{j},\mathbf{k}}S_{\mathbf{j}}S_{\mathbf{k}}^*\Big)=\sum_{|\mathbf{j}|=m=|\mathbf{k}|}A_{\mathbf{j},\mathbf{k}}Z_{\mathbf{j}}Z_{\mathbf{k}}^*.
$$
\end{proof}
\begin{cor} For all $\mathbf{j},\mathbf{k}\in\F_n^+$ with $|\mathbf{j}|=|\mathbf{k}|=m$ and all $l\geq m$ we have
$$
\varsigma^{(l)}(S_\mathbf{j}S_\mathbf{k}^*|_{\GH_l})=Z_{\mathbf{j}}Z_{\mathbf{k}}^*.
$$
\end{cor}

\begin{Example}\label{commutecovsymbasindlim}
Let $\GH_\bullet$ be a commutative subproduct system and let $M$ be the compact manifold it defines (see Corollary \ref{commutecorkahler}). Then $\varsigma^{(m)}:\Bi(\GH_m)\to\Oi_\GH^{(0)}$ coincides with the Berezin covariant symbol map $\varsigma^{(m)}:\Bi(\GH_m)\to C(M)$ (mentioned in \S\ref{projquantsec}).
\end{Example}

\begin{Notation}\label{unitaryforallHnot}
Fix a faithful representation of $\Oi_\GH$ on a Hilbert space $\Hi$ and let 
$$
u\in\Mn_n(\C)\otimes\Bi(\Hi)
$$
be a unitary $n\times n$ matrix with values $u_{j,k}\in\Bi(\Hi)$ such that the first row of $u$ is given by the generators $Z_1,\dots,Z_n$ of $\Oi_\GH$. Let $u^c$ be the matrix obtained from $u$ by taking adjoints of each entry $u_{j,k}$. Denote by $u_m$ the restriction of $u^{\otimes m}$ from $\GH^{\otimes m}$ to $\GH_m$ and similarly for $u^c$. Finally, 
$$
\alpha^{(m)}:\Bi(\GH_m)\to\Bi(\GH_m)\otimes\Bi(\Hi)
$$
will be the map which takes $A\in\Bi(\GH_m)$ to $u_m(A\otimes\bone)u_m^*$. 
\end{Notation}
The following formulas are known from the classical case to define the ``Berezin covariant symbol" in case we quantize a coadjoint orbit $M=G/K$ (cf. \cite{Per1}, \cite{Lan2}).
\begin{prop}\label{covsymstandardexprposm}
Assume that $u^{\otimes m}$ preserves the subspace $\GH_m\subset\GH^{\otimes m}$, in the sense that $u_m=u^{\otimes m}(p_m\otimes\bone)$. Let $|e_1^{\otimes m}\ket\bra e_1^{\otimes m}|$ be the rank-$1$ projection onto the line spanned by $e^{\otimes m}_1$, where $e_1$ is the first basis vector in $\GH$. Then for all $A\in\Bi(\GH_m)$ we have
\begin{align}
\varsigma^{(m)}(A)&=(\Tr\otimes\id)\big((A\otimes\bone)u_m^{c*}(|e_1^{\otimes m}\ket \bra e_1^{\otimes m}|\otimes\bone)u_m^c\big)\label{firstumforcov}
\\&=(\Tr\otimes\id)\big(\alpha^{(m)}(A)(|e_1^{\otimes m}\ket \bra e_1^{\otimes m}|\otimes\bone)\big).\label{secondumforcov}
\end{align}
\end{prop}
\begin{proof} We have
\begin{align*}
u_m^{c*}(|e_1^{\otimes m}\ket \bra e_1^{\otimes m}|\otimes\bone)u_m^c&=\sum_{|\mathbf{j}|=m=|\mathbf{k}|}S_\mathbf{j}S_\mathbf{k}^*|_{\GH_m}\otimes Z_\mathbf{k}Z_\mathbf{j}^*
\end{align*}
so \eqref{firstumforcov} is clear. For \eqref{secondumforcov} we can use the formula 
\begin{align*}
\alpha^{(m)}(A)&=\sum_{|\mathbf{j}|=m=|\mathbf{k}|}A_{\mathbf{r},\mathbf{s}}S_\mathbf{j}S_\mathbf{k}^*\big|_{\GH_m}\otimes u_{\mathbf{j},\mathbf{r}}u_{\mathbf{k},\mathbf{s}}^*.
\end{align*}
\end{proof}

\subsection{Commutative case and the Arveson conjecture}
One of the most striking applications of our results is the Arveson conjecture (see Remark \ref{Arvvconjrem}). %The validity of this conjecture has been proven by other means in some cases \cite{DoWa2}, \cite{EnEs1}. 
\begin{cor}\label{Arvvconjcor} Arveson's conjecture holds for all homogeneous ideals $\Ii\subset\C[z_1,\dots,z_n]$, 
i.e. the Cuntz--Pimsner algebra $\Oi_\GH$ of the subproduct system $\GH_\bullet$ associated to $\Ii$ (as in Lemma \ref{idealsublemma}) is commutative. 
\end{cor}
\begin{proof}
Lemma \ref{Lemmanormorder} together with \cite[Prop. 4.14]{KeSh1} gives the result. 
\end{proof}

In \cite{Vas1}, \cite{Vas3} it was shown that for any continuous line bundle $L\to M$, the Cuntz--Pimsner algebra $\Oi_\Ei$ (defined in \cite{Pims1}) of the Hilbert $C(M)$-bimodule $\Ei$ of continuous sections of $L$ is isomorphic to the $C^*$-algebra $C(\Sb_L)$ of continuous functions on the total space of the circle bundle $\Sb_L$ associated to $L^*$. Recall that in the definition of $\Oi_\Ei$ (which is Pimsner's original one) the tensor products are taken over the coefficient algebra $C(M)$. We shall now see that, in the case $(M,L)$ is a polarized manifold, from a projectively induced quantization we can also obtain $C(\Sb_M):=C(\Sb_L)$ as the Cuntz--Pimsner algebra $\Oi_\GH$ of the associated subproduct system $\GH_\bullet$. %If we let $\Ei$ be as in  vector space $\GH_m$ is the ``holomorphic part" of the vector space $\Ei^{\otimes m}$ (where $\otimes:=\otimes_{C(M)}$).

\begin{prop}\label{propcommutecircle}
Let $(M,L)$ be a polarized (not necessarily smooth) variety and let $\GH_\bullet$ be a projectively induced quantization of $(M,L)$ (the definition still makes sense in the non-smooth case), and endow $M$ with the complex (Hausdorff) topology \cite[§2]{Serre1}. Then
\begin{equation}\label{circlebundlecontfunc}
C(\Sb_M)\cong\Oi_\GH,
\end{equation}
and for all $k\in\Z$,
$$
\Gamma(M;L^{\otimes k})\cong\Oi^{(k)}_\GH
$$  
as a Hilbert $C(M)$-bimodule. In particular, $k=0$ gives $C(M)$ as an inductive limit.
\end{prop}
\begin{proof} As in the general noncommutative case, the commutative algebra $\Oi_\GH$ is built up from Hilbert modules over $\Oi_\GH^{(0)}$ and the latter is generated by the images of the covariant symbol maps $\varsigma^{(m)}$. An argument given in \cite[\S4]{CGR1} shows that the $\varsigma^{(m)}(A)$'s (for all $m\in\N_0$ and all $A\in\Bi_m$) separate points. 
%Choose linearly independent sections $e_1,\dots,e_{n_m}$ giving an orthonormal basis for $\GH_m$. Then $K:=\sum_ke_k\otimes\overline{e}_k$ is a positive-definite kernel with associated reproducing kernel space equal to $\GH_m$ (see e.g. \cite{BeHi1}). Let $|x\ket:=K_x$ denote the associated coherent states. The adjoint $\varsigma^{(m)}:\Bi(\GH_m)\to C(M)$ of $\breve{\varsigma}^{(m)}$ can in this setting be expressed as $\varsigma^{(m)}(A)(x)=\bra x|A|x\ket/\bra x|x\ket$ for $x\in M$ and $A\in\Bi(\GH_m)$. An argument given in \cite[\S4]{CGR1} shows that the $\varsigma^{(m)}(A)$'s (for all $m\in\N_0$ and all $A\in\Bi_m$) separate points. 
The Stone--Weierstrass theorem gives that they form a dense subalgebra of $C(M)$. The inclusion of $\GH_m$ into $\GH_{m+1}$ coincide by construction with our $\iota_{m,m+1}$, and the supremum norm on $C(M)$ is seen to coincide with the norm on the inductive limit $\Bi_\infty$. Therefore, $C(M)\cong\Bi_\infty$. The result now follows from Theorem \ref{indlimpimsnerthm} and the well-know decomposition of $C(\Sb_M)$ into the $\Gamma(M;L^{\otimes m})$'s (see §\ref{circbundsec}). 

\end{proof}
 In fact, given Corollary \ref{Arvvconjcor} we get Proposition \ref{propcommutecircle} from the calculation of the space of multiplicative $\Un(1)$-valued functionals on $\Oi_\GH$ done in \cite[Prop. 2.4]{KeSh1}. Indeed, the circle bundle $\Sb_L$ can be identified with the preimage of $M$ under the map $\C^n\setminus\{0\}\to\C\Pb^{n-1}$ which defined projective space, and this is the boundary of the analytic variety discussed \cite{KeSh1}. 

Note that it is not so obvious that we could recover $M$ completely (as a topological space) from $\GH_\bullet$ because the homogeneous coordinate ring $\bigoplus_{m\in\N_0}H^0(M;L^{\otimes m})$ of $M\subset\Pb[\GH^*]$ is not a ring of functions on $M$. The choice of basis on $H^0(M;L^{\otimes m})$ corresponds to a choice of algebra structure on the ring $C(M)$ and the inner product on $H^0(M;L^{\otimes m})$ to a choice of $C^*$-algebra structure on $C(M)$, but all possible $C^*$-algebra structures obtain in this way are isomorphic. 

In the following we use the terminology from Lemma \ref{CYlemma}.
\begin{cor}\label{limitvolumecor}
Let $(M,L)$ be a polarized manifold and let $\GH_\bullet$ be a projectively induced quantization of $(M,L)$. Then the Fubini--Study metric $\FS(\bra\cdot|\cdot\ket)$ on $L$ associated with the inner product on $\GH$ coincides with the $*$-operation which defines the $C^*$-algebra $\Oi_\GH$, and is thus equal to the inductive limit of the Hermitian pairings
$$
\Bi(\GH_{m+1},\GH_m)\times\Bi(\GH_{m+1},\GH_m)\to\Bi(\GH_m),\qquad (A,B)\to A^*B.
$$
%he limit state on $C(M):=\Bi_\infty$ coincides with the unique volume form on $M$ which balances the inner product on $\GH$. 
Consequently, $(M,L)$ is balanced if and only if the limit state $\omega$ coincides with the the unique $\FS(\bra\cdot|\cdot\ket)$-balancing state. If $M=G/K$ is a coadjoint orbit one also has (using Notation \ref{freeunitsemnot}) for all $\mathbf{j},\mathbf{k}\in\F_n^+$ with $|\mathbf{j}|=m=|\mathbf{k}|$ that
\begin{equation}\label{regularlimstate}
\frac{1}{\vol(M,L)}\int_MZ_\mathbf{j}^*Z_\mathbf{k}\frac{\omega^d}{d!}=\frac{p_{\mathbf{j},\mathbf{k}}}{\Tr(p_m)},
\end{equation}
where $p_m:\GH^{\otimes m}\to\GH_m$ is the orthogonal projection and $\GH_\bullet\subset \GH^{\vee \bullet}$ is the subproduct system of $(M,L,\bra\cdot|\cdot\ket)$ and we denote by $\omega$ also the Fubini--Study Kähler form on $M\subset\C\Pb^{n-1}$. 
\end{cor}
\begin{proof}
Recall that the generators $Z_1,\dots,Z_n$ of $\Oi_\GH$ satisfy the relation of the ideal which defines $M$, so they can be identified with the homogeneous coordinates on $M$. Recall also that $\sum^n_{k=1}Z_kZ_k^*=\bone$, which says that the adjoint operation on the operator system in $\Oi_\GH$ spanned by $Z_1,\dots,Z_n$ is the Fubini--Study metric $h$ on $L$. Since products of $m$ of the generators $Z_1,\dots,Z_n$ identify with elements of $H^0(M;L^m)$ and since the $*$-operation on $\Oi_\GH$ is given by $(Z_{k_1}\cdots Z_{k_m})^*=Z_{k_m}^*\cdots Z_{k_1}^*$, we get that it is induced by the tensor-product metric $h^m$ on $H^0(M;L^m)$ for all $m\in\N$.  

If $M=G/K$ is a coadjoint orbit then the limit state $\omega=\omega_{p_1}:C(M)\to\C$ satisfies $\omega(Z_j^*Z_k)=\delta_{j,k}/n$. The formula for $\omega$ on the products $Z_\mathbf{j}^*Z_\mathbf{k}$ then gives \eqref{regularlimstate}. %if we recall that (by Hirzebruch--Riemann--Roch)
%$$
%\vol(M,L)=\lim_{m\to\infty}\frac{\Tr(p_m)}{m^d}.
%$$
\end{proof}
%As mentioned, the existence of the $\bra\cdot|\cdot\ket$-balancing volume form $\omega_{p_1}$ is related to the Calabi--Yau theorem, and it was previously known that $\omega_{p_1}$ can be obtained using finite-dimensional approximation \cite{CaKe1}, \cite[\S2.2]{Don3}. This procedure comes out automatically in a slightly different guise in our approach using inductive limits; the limit state on $\Bi_\infty$ is a generalization of the volume form associated with the Calabi--Yau metric. 

%By \cite{Zeld1}, the Hermitian metric $h$ on $L$ used to define the inner product on $\GH$ is recovered the $C^\infty$-topology as the $m\to\infty$ limit of the pullbacks of the Fubini--Study metrics via the Kodaira embeddings. In a rather different fashion, Corollary \ref{limitvolumecor} gives the Fubini--Study Hermitian metric $\text{FS}(\bra\cdot|\cdot\ket)$ on $L$ as a limit of the inner products of the $\GH_m$'s. That is, we obtain a Hermitian metric on $L$ from a sequence of matrix-valued \emph{inner products} on Hilbert modules over finite-dimensional matrix algebra. These two approximation results are rather similar though since the space of pullback metrics with respect to a Kodaira embeddings for $H^0(M;L^m)$ identifies with the space of inner products on $H^0(M;L^m)$. 

\section{Projective limits}\label{projlimsec}
We now want to realize the same algebra $\Oi_\GH^{(0)}$ as a projective limit. For this we need some background information from \cite[§B2]{Hawk1}. 
\subsection{Relaxing the notion of projective limit}
In this paper, a ``projective limit" will always refer to the following object which, in comparison to more conventional $C^*$-algebraic projective limits, is defined in terms of completely positive maps instead of $C^*$-homomorphisms. 
\begin{dfn}[{\cite[§B2]{Hawk1}}]
A \textbf{projective system} $(\Bi_\bullet,\jmath_{\bullet,\bullet})$ is a sequence of finite-dimensional matrix algebras $\Bi_m$ and norm-contracting completely positive mappings $\jmath_{l,m}:\Bi_l\to\Bi_m$ for $m\leq l$ satisfying $\jmath_{l,m}=\jmath_{r,m}\circ\jmath_{l,r}$ for all $m\leq r\leq l$. The \textbf{projective limit} of $(\Bi_\bullet,\jmath_{\bullet,\bullet})$ is the vector space defined by
$$
\Bi^\infty:=\{A_\bullet=(A_m)_{m\in\N_0}\in \Gamma_b(\Bi_\bullet)|\ A_{m-1}=\jmath_{m,m-1}(A_m)\textnormal{ for all }m\in\N\},
$$
equipped with the norm
$$
\|A_\bullet\|:=\lim_{m\to\infty}\|A_m\|.
$$
\end{dfn}
\begin{Remark} The intersection of $\Bi^\infty$ with $\Gamma_0$ is $\{0\}$. We always identify $\Bi^\infty$ with its embedding into $\Gamma_b(\Bi_\bullet)/\Gamma_0(\Bi_\bullet)$ because it is more likely that $\Bi^\infty$ is an algebra when multiplication is taken modulo $\Gamma_0$. If we do so and then pull back $\Bi^\infty$ via the quotient map $\pi:\Gamma_b\to \Gamma_b/\Gamma_0$, we obtain a vector space $\pi^{-1}(\Bi^\infty)$ which is much larger than $\Bi^\infty$, namely
\begin{equation}\label{pullbackprojlim}
\pi^{-1}(\Bi^\infty)=\Bi^\infty\cup\Gamma_0(\Bi_\bullet).
\end{equation}
Importantly, $\Bi^\infty$ is an algebra (hence a $C^*$-algebra) if and only if \eqref{pullbackprojlim} is. 
\end{Remark}
\begin{Remark}\label{projlimremarktwo}
We could also define $\Bi^\infty$ as the set of elements 
$$
f=\pi\big((\jmath_{\infty,m}(f))_{m\in\N_0}\big)
$$
where the components $\jmath_{\infty,m}(f)\in\Bi_m$ satisfy $\jmath_{\infty,m}(f)=\jmath_{l,m}\circ\jmath_{\infty,l}(f)$. We can regard $\jmath_{\infty,m}:=\lim_{l\to\infty}\jmath_{l,m}$ as the map from $\Bi^\infty$ to $\Bi_m$ which evaluates $A_\bullet=(A_m)_{m\in\N_0}\in\Bi^\infty$ at $m$,
$$
\jmath_{\infty,m}(A_\bullet)=A_m.
$$
\end{Remark}

\subsection{Changing the inner products}\label{notationsec}

From now on $Q\in\Bi(\GH)$ is a positive invertible $n\times n$ matrix and let
$$
Q_m:=p_mQ^{\otimes m}|_{\GH_m}
$$
be the compression of $Q^{\otimes m}\in\Bi(\GH^{\otimes m})$ to the subspace $\GH_m=p_m\GH^{\otimes m}$. 
We choose the orthonormal basis $e_1,\dots,e_n$ for $\GH_1$ such that $Q$ is diagonal and, as before, we let $S_1,\dots,S_n$ be the shifts on $\GH_\N$ by these basis vectors. We shall write
$$
Q_{\mathbf{j},\mathbf{k}}:=(Q_m)_{\mathbf{j},\mathbf{k}},\qquad Q^{\mathbf{j},\mathbf{k}}:=(Q^{-1}_m)_{\mathbf{j},\mathbf{k}}.
$$
We associate to each $Q_m$ a density matrix
$$
\rho^{(m)}=\rho_Q^{(m)}:=\frac{Q_m}{\Tr(Q_m)},
$$
and denote by $\phi_m$ the state on $\Bi(\GH_m)$ defined by
$$
\phi_m(A):=\Tr(\rho^{(m)}A),\qquad\forall A\in\Bi(\GH_m).
$$
Sometimes it will be useful to change the inner product on $\GH_m$ to
\begin{equation}\label{rhoQinnerprod}
\bra\psi|\phi\ket_{\rho^{(m)}}:=\bra\psi|\rho^{(m)}\phi\ket,\qquad \forall \phi,\psi\in\GH_m.
\end{equation}
We stress that (unless $\GH_\bullet=\GH^{\otimes\bullet}$)
$$
\rho^{(m)}\ne \frac{Q_m}{\Tr(Q)^m}=p_m(\rho^{(1)})^{\otimes m}p_m.
$$
\begin{Ass}\label{Qnotat} From now on $Q\in\Bi(\GH)$ is a positive invertible $n\times n$ matrix such that 
\begin{enumerate}[(i)]
\item{the operator $Q^{\otimes m}$ on $\GH^{\otimes m}$ preserves the subspace $\GH_m$ for all $m\in\N_0$, i.e. $Q_m=Q^{\otimes m}|_{\GH_m}$, and}
\item{$\iota_{m,l}$ preserves the $Q$-traces, i.e.
$$
\phi_l\circ\iota_{m,l}=\phi_m,
$$
for all $m\leq l\in\N$.}
\end{enumerate}
\end{Ass}
Property (i) allows the subproduct condition to be maintained with the new inner products $\bra\cdot|\cdot\ket_{\rho^{(m)}}$ (see Proposition \ref{isomprop} below) while property (ii) allows the construction of a state on $\Oi_\GH$ with very nice properties. We will show that in the examples of subproduct systems coming from compact quantum groups there is always a matrix $Q$ which satisfies these assumptions (i) and (ii). It is possible to drop either (or both) of the assumptions (i) and (ii) and many of the constructions in the next section will carry over; in particular there will be a limit state but with weaker quantization properties. We shall elaborate on this slightly in the commutative case, where assumptions (i) and (ii) hold with $Q_m=p_m$ for all $m\in\N$ precisely when $\GH_\bullet$ is a regular quantization (in the sense of Remark \ref{regularremark}). 

\begin{Remark}\label{Qfactsrem}
The property $p_mQ^{\otimes m}p_m=Q^{\otimes m}p_m$ ensures that $Q_m$ is invertible; its inverse is $p_m(Q^{\otimes m})^{-1}p_m$ because
$$
(Q^{\otimes m})^{-1}p_mQ^{\otimes m}p_m=(Q^{\otimes m})^{-1}Q^{\otimes m}p_m=p_m,
$$
$$
Q^{\otimes m}p_m(Q^{\otimes m})^{-1}p_m=Q^{\otimes m}(Q^{\otimes m})^{-1}p_m=p_m,
$$
where we used the fact that the inverse $A^{-1}$ of any invertible matrix $A$ preserves every $A$-invariant subspace. 
We denote by $Q_m^{-1}$ this inverse of $Q_m$. For $l\geq m$ we have
$$
Q_l^{-1}(Q_m\otimes Q_{l-m})=(Q^{\otimes l})^{-1}((Q^{\otimes m})\otimes (Q^{\otimes (l-m)}))p_l=p_l,
$$
where we regard $p_l$ as an operator from $\GH_m\otimes\GH_{l-m}$ onto $\GH_l$. 
\end{Remark}

\subsection{The isometries}
Now that $\GH_m$ is endowed with the inner product \eqref{rhoQinnerprod} we discuss how $\Bi(\GH_m)$ can be mapped into $\Bi(\GH_l)$ when $m<l$ and calculate the explicit isometries $\GH_l\hookrightarrow\GH_{l-m}\otimes\GH_m$ and $\GH_l\hookrightarrow\GH_m\otimes\GH_{l-m}$. This construction would fail without the assumption that $Q^{\otimes m}$ preserves $\GH_m$. 

\begin{prop}\label{isomprop}
The isometry from $\GH_l$ into $\GH_{l-m}\otimes\GH_m$ is given by
\begin{equation}\label{isomfromGHlinto}
\bar{V}_{m,l}\psi=\sqrt{\frac{\Tr(Q_m)\Tr(Q_{l-m})}{\Tr(Q_l)}}\sum_{|\mathbf{r}|=l-m}p_{l-r}e_\mathbf{r}\otimes S_\mathbf{r}^*\psi,\qquad \forall \psi\in\GH_l,
\end{equation}
and its adjoint by
\begin{equation}\label{adjointofisomfromGHlinto}
\bar{V}_{m,l}^*=\sqrt{\frac{\Tr(Q_m)\Tr(Q_{l-m})}{\Tr(Q_l)}}(\rho^{(l)})^{-1}(\rho^{(l-m)}\otimes \rho^{(m)}).
\end{equation}
\end{prop}
\begin{proof} We proceed by first calculating the adjoint of the given operator \eqref{isomfromGHlinto}. Let $\lambda_{m,l}:=\sqrt{\frac{\Tr(Q_m)\Tr(Q_{l-m})}{\Tr(Q_l)}}$. For all $\xi\in\GH_{l-m},\eta\in\GH_m$ and all $\psi\in\GH_l$ we have
\begin{align*}
\bra \bar{V}_{m,l}^*(\xi\otimes\eta)|\psi\ket_{\rho^{(l)}}&=\bra\xi\otimes\eta|\bar{V}_{m,l}\psi\ket_{\rho^{(l-m)}\otimes\rho^{(m)}}
\\&=\lambda_{m,l}\sum_{|\mathbf{r}|=l-m}\bra \xi\otimes\eta|p_l(e_\mathbf{r}\otimes S_\mathbf{r}^*\psi)\ket_{\rho^{(l-m)}\otimes\rho^{(m)}}
\\&=\lambda_{m,l}\sum_{|\mathbf{r}|=l-m}\bra \xi\otimes\eta|S_\mathbf{r}S_\mathbf{r}^*\psi\ket_{\rho^{(l-m)}\otimes\rho^{(m)}}
\\&=\lambda_{m,l}\bra (\rho^{(l)})^{-1}(\rho^{(l-m)}\xi\otimes\rho^{(m)}\eta)|\psi\ket_{\rho^{(l)}}
\end{align*}
and hence \eqref{adjointofisomfromGHlinto} holds. Finally, using Remark \ref{Qfactsrem},
\begin{align*}
\bar{V}_{m,l}^*\bar{V}_{m,l}\psi&=\frac{\Tr(Q_m)\Tr(Q_{l-m})}{\Tr(Q_l)}\sum_{|\mathbf{r}|=l-m}(\rho^{(l)})^{-1}(\rho^{(l-m)}e_\mathbf{r}\otimes \rho^{(m)}S_\mathbf{r}^*\psi)
\\&=\sum_{|\mathbf{r}|=l-m}Q^{-1}_l(Q_{l-m}e_\mathbf{r}\otimes Q_mS_\mathbf{r}^*\psi)
=\sum_{|\mathbf{r}|=l-m}p_l(e_\mathbf{r}\otimes S_\mathbf{r}^*\psi)
\\&=\sum_{|\mathbf{r}|=l-m}S_\mathbf{r}S_\mathbf{r}^*\psi=\psi,
\end{align*}
so $V$ is the desired isometry. Let us also calculate the final projection:
\begin{align*}
\bar{V}_{m,l}\bar{V}_{m,l}^*&=\sum_{|\mathbf{r}|=l-m}S_\mathbf{r}S_\mathbf{r}^*Q_l^{-1}(Q_{l-m}\otimes Q_m)
=p_l(p_{l-m}\otimes p_m).
\end{align*}
\end{proof}
Now let $A\in\Bi(\GH_l)$ and define
\begin{equation}\label{bariotaandisom}
\bar{\iota}_{m,l}(A):=\bar{V}_{m,l}^*(\bone_{\GH_{l-m}}\otimes A)\bar{V}_{m,l}.
\end{equation}
We have
\begin{align*}
\bar{V}_{m,l}^*(\bone_{\GH_{l-m}}\otimes A)\bar{V}_{m,l}\psi&=\frac{\Tr(Q_m)\Tr(Q_{l-m})}{\Tr(Q_l)}\sum_{|\mathbf{r}|=l-m}(\rho^{(l)})^{-1}(\rho^{(l-m)}e_\mathbf{r}\otimes \rho^{(m)}AS_\mathbf{r}^*\psi)
\\&=\sum_{|\mathbf{r}|=l-m}Q^{-1}_l(Q_{l-m}e_\mathbf{r}\otimes Q_mAS_\mathbf{r}^*\psi)
\\&=\sum_{|\mathbf{r}|=l-m}p_l(e_\mathbf{r}\otimes AS_\mathbf{r}^*\psi)
\\&=\sum_{|\mathbf{r}|=l-m}S_\mathbf{r}AS_\mathbf{r}^*\psi,
\end{align*}
so $\bar{\iota}_{m,l}$ is a ``chirality-flipped" version of $\iota_{m,l}$, i.e. the $R_r$'s are replaced by $S_j$'s (cf. \eqref{flipinduct}). If we use the inductive system $\bar{\iota}_{\bullet,\bullet}$ instead of $\iota_{\bullet,\bullet}$, the roles of the left and right Toeplitz algebras are interchanged. 
\begin{prop} Define a coherent system of maps $\bar{\iota}_{m,l}:\Bi(\GH_m)\to\Bi(\GH_l)$ by
\begin{align}
\bar{\iota}_{m,l}(A)&:=\bar{V}_{m,l}^*(\bone_{\GH_{l-m}}\otimes A)\bar{V}_{m,l}
\\&=\sum_{|\mathbf{s}|=l-m}S_\mathbf{s}AS_\mathbf{s}^*\big|_{\GH_{l}}.
\end{align}
Then $(\Bi(\GH_\bullet),\bar{\iota}_{\bullet,\bullet})$ is an inductive system, and the $C^*$-subalgebra of $\Gamma_b(\Bi_\bullet)$ consisting of norm limits of the eventually constant sequences for $\bar{\iota}_{\bullet,\bullet}$ is equal to the $\Un(1)$-invariant part of the right Toeplitz algebra $C^*(R_1,\dots,R_n)$.
\end{prop}
\begin{proof} The proof is very similar to the case of the $\iota_{m,l}$'s (the ``left case"). 
\end{proof}
We furthermore note that if we expand $A\in\Bi(\GH_m)$ as $A=\sum_{|\mathbf{j}|=m=|\mathbf{k}|}A_{\mathbf{j},\mathbf{k}}R_\mathbf{j}R_\mathbf{j}^*|_{\GH_m}$ then
$$
\bar{\iota}_{m,l}(A)=\sum_{|\mathbf{j}|=m=|\mathbf{k}|}A_{\mathbf{j},\mathbf{k}}R_\mathbf{j}R_\mathbf{k}^*\big|_{\GH_l},
$$
again similar to the left case. More will be said on the ``chiral duality" between $\bar{\iota}_{\bullet,\bullet}$ and $\iota_{\bullet,\bullet}$ in Remark \ref{chiralrem}.

We now want to find an isometric implementation of $\iota_{m,l}$ similar to \eqref{bariotaandisom}. For this we need to flip the tensor factors. 
\begin{prop}\label{isomprop} The isometry from $\GH_l$ into $\GH_{m}\otimes\GH_{l-m}$ is given by
\begin{equation}\label{rightisomfromGHlinto}
V_{m,l}\psi=\sqrt{\frac{\Tr(Q_m)\Tr(Q_{l-m})}{\Tr(Q_l)}}\sum_{|\mathbf{r}|=l-m}R_\mathbf{r}^*\psi\otimes p_{l-r}e_\mathbf{r},\qquad \forall \psi\in\GH_l,
\end{equation}
and its adjoint by
\begin{equation}\label{adjointofrightisomfromGHlinto}
V_{m,l}^*=\sqrt{\frac{\Tr(Q_m)\Tr(Q_{l-m})}{\Tr(Q_l)}}(\rho^{(l)})^{-1}(\rho^{(m)}\otimes\rho^{(l-m)}).
\end{equation}
\end{prop}
\begin{proof} For all $\xi\in\GH_{l-m},\eta\in\GH_m$ and all $\psi\in\GH_l$ we have
\begin{align*}
\bra V_{m,l}^*(\eta\otimes\xi)|\psi\ket_{\rho^{(l)}}&=\bra\eta\otimes\xi|V_{m,l}\psi\ket_{\rho^{(m)}\otimes\rho^{(l-m)}}
\\&=\lambda_{m,l}\sum_{|\mathbf{r}|=l-m}\bra \eta\otimes\xi|R_\mathbf{r}R_\mathbf{r}^*\psi\ket_{\rho^{(m)}\otimes\rho^{(l-m)}}
\\&=\lambda_{m,l}\bra (\rho^{(l)})^{-1}(\rho^{(m)}\eta\otimes\rho^{(l-m)}\xi)|\psi\ket_{\rho^{(l)}},
\end{align*}
and the rest is similar to the proof of Proposition \ref{isomprop}. 
\end{proof}
\begin{cor} The inductive system $\iota_{\bullet,\bullet}$ is implemented by the system $V_{\bullet,\bullet}$ of isometries:
$$
\iota_{m,l}(A)=V_{m,l}^*(A\otimes \bone_{\GH_{l-m}})V_{m,l},\qquad\forall A\in\Bi(\GH_m).
$$
\end{cor}
\begin{proof} Follows from formula \eqref{flipinduct} and the calculation (with $\psi\in\GH_l$)
\begin{align*}
V_{m,l}^*(A\otimes \bone_{\GH_{l-m}})V_{m,l}\psi&=\sqrt{\frac{\Tr(Q_m)\Tr(Q_{l-m})}{\Tr(Q_l)}}\sum_{|\mathbf{r}|=l-m}(\rho^{(l)})^{-1}(\rho^{(m)}AR_\mathbf{r}^*\psi\otimes\rho^{(l-m)}e_\mathbf{r})
\\&=\sum_{|\mathbf{r}|=l-m}Q^{-1}_l(Q_mAR_\mathbf{r}^*\psi\otimes Q_{l-m}e_\mathbf{r})
\\&=\sum_{|\mathbf{r}|=l-m}p_l(AR_\mathbf{r}^*\psi\otimes e_\mathbf{r})
\\&=\sum_{|\mathbf{r}|=l-m}R_\mathbf{r}AR_\mathbf{r}^*\psi.
\end{align*}
\end{proof}

\subsection{Projective system for subproduct systems}
Let $\GH_\bullet$ be a subproduct system. We let $Q\in\Bi(\GH)$ be as in Assumption \ref{Qnotat} and endow $\GH_m$ with the inner product defined by the density matrix $\rho^{(m)}:=Q_m/\Tr(Q_m)$. 
\begin{Lemma}\label{subprodprojlimsyst}
Define maps $\jmath_{l,m}:\Bi(\GH_l)\to\Bi(\GH_m)$ for $m\leq l$ by
\begin{equation}\label{projectsyst}
\jmath_{l,m}(A):=\frac{\Tr(Q_m)}{\Tr(Q_l)}\sum_{|\mathbf{k}|=l-m}(Q^{\otimes m})_{\mathbf{k},\mathbf{k}}R_\mathbf{k}^*A R_\mathbf{k}\big|_{\GH_m},\qquad \forall A\in\Bi(\GH_l).
\end{equation}
Then, with $V_{m,l}$ as in Proposition \ref{isomprop}, we have the formula
\begin{align}\label{projlimaspartialtrace}
\jmath_{l,m}(A)&=(\id_{\Bi_m}\otimes\phi_{l-m})(V_{m,l}AV_{m,l}^*)
\end{align}
and $\jmath_{\bullet,\bullet}$ is a projective system. 
\end{Lemma}
\begin{proof} First of all, for all $A\in\Bi(\GH_l)$ we have
\begin{align*}
\jmath_{r,m}\circ\jmath_{l,r}(A)&=\frac{\Tr(Q_r)}{\Tr(Q_l)}\frac{\Tr(Q_m)}{\Tr(Q_r)}\sum_{|\mathbf{j}|=r-m,|\mathbf{k}|=l-r}(Q^{\otimes m})_{\mathbf{j},\mathbf{j}}(Q^{\otimes m})_{\mathbf{k},\mathbf{k}}R_\mathbf{j}^*R_\mathbf{k}^*A R_\mathbf{k}R_\mathbf{j}\big|_{\GH_m}
\\&=\frac{\Tr(Q_m)}{\Tr(Q_l)}\sum_{|\mathbf{kj}|=l-m}(Q^{\otimes m})_{\mathbf{kj},\mathbf{kj}}R_\mathbf{kj}^*A R_\mathbf{kj}\big|_{\GH_m}
\\&=\frac{\Tr(Q_m)}{\Tr(Q_l)}\sum_{|\mathbf{r}|=l-m}(Q^{\otimes m})_{\mathbf{r},\mathbf{r}}R_\mathbf{r}^*A R_\mathbf{r}\big|_{\GH_m}=\jmath_{l,m}(A),
\end{align*}
and it is obvious that each $\jmath_{l,m}$ is completely positive. The norm-contracting property holds because $\jmath_{l,m}$ is in fact unital. To see this we first prove the alternative formula \eqref{projlimaspartialtrace}. We have
\begin{align*}
\jmath_{l,m}(A)&=\frac{\Tr(Q_m)}{\Tr(Q_l)}\sum_{|\mathbf{r}|=l-m}(Q^{\otimes m})_{\mathbf{r},\mathbf{r}}R_\mathbf{r}^*AR_\mathbf{r}\big|_{\GH_m}
\\&=\frac{\Tr(Q_m)}{\Tr(Q_l)}\sum_{|\mathbf{r}|=l-m}\sum_{|\mathbf{j}|=m=|\mathbf{k}|}(Q^{\otimes m})_{\mathbf{r},\mathbf{r}}\bra e_\mathbf{j}|R_\mathbf{r}^*AR_\mathbf{r}e_\mathbf{k}\ket S_\mathbf{j}S_\mathbf{k}^*\big|_{\GH_m}
\\&=\frac{\Tr(Q_m)}{\Tr(Q_l)}\sum_{|\mathbf{r}|=l-m}\sum_{|\mathbf{j}|=m=|\mathbf{k}|}(Q^{\otimes m})_{\mathbf{r},\mathbf{r}}\bra e_\mathbf{j}\otimes e_{\mathbf{r}}|A(e_\mathbf{k}\otimes e_{\mathbf{r}})\ket S_\mathbf{j}S_\mathbf{k}^*\big|_{\GH_m}
\\&=\frac{\Tr(Q_m)}{\Tr(Q_l)}\sum_{|\mathbf{r}|=l-m}\sum_{|\mathbf{j}|=m=|\mathbf{k}|}(Q^{\otimes m})_{\mathbf{r},\mathbf{r}}(A)_{\mathbf{j}\mathbf{r},\mathbf{k}\mathbf{r}} S_\mathbf{j}S_\mathbf{k}^*\big|_{\GH_m}
\\&=(\id\otimes\phi_{l-m})(V_{m,l}AV_{m,l}^*),
\end{align*}
where in the last equality we used that, for all $\xi_1,\xi_2\in\GH_m$, $\eta_1,\eta_2\in\GH_{l-m}$,
\begin{align*}
&\bra V^*_{m,l}(\xi_1\otimes \eta_1)|AV^*_{m,l}(\xi_2\otimes Q_{l-m}\eta_2)\ket
\\&=\frac{\Tr(Q_m)\Tr(Q_{l-m})}{\Tr(Q_l)}\bra \rho^{(l)-1}(\rho^{(m)}\xi_1\otimes\rho^{(l-m)}\eta_1)|A\rho^{(l)-1}(\rho^{(m)}\xi_2\otimes\rho^{(l-m)}Q_{l-m}\eta_2)\ket
\\&=\bra \xi_1\otimes\eta_1|A\rho^{(l)-1}(\rho^{(m)}\xi_2\otimes\rho^{(l-m)}Q_{l-m}\eta_2)\ket
\\&=\frac{\Tr(Q_l)}{\Tr(Q_{l-m})\Tr(Q_m)}\bra \xi_1\otimes\eta_1|A(\xi_2\otimes Q_{l-m}\eta_2)\ket,
\end{align*}
so that summing such inner products over a basis for $\GH_m\otimes\GH_{l-m}$ and multiplying with $\frac{\Tr(Q_m)}{\Tr(Q_l)}S_\mathbf{j}S_\mathbf{k}^*\big|_{\GH_m}$ is the same thing as partially tracing $V_{m,l}AV^*_{m,l}$ with $Q_{l-m}/\Tr(Q_{l-m})$. 

%To show that $\jmath_{\bullet,\bullet}$ is a projective system we have to show that each map $\jmath_{l,m}$ is contractive.
 The formula \eqref{projlimaspartialtrace} shows that $\jmath_{l,m}$ is the adjoint of $\iota_{m,l}$ with respect to $\phi_l$ and $\phi_m$ (see details in \ref{propwithstatescompat} below). Our assumption $\phi_l\circ\iota_{m,l}=\phi_m$ is then equivalent to the unitality
$$
\jmath_{l,m}(p_l)=p_m,
$$ 
and hence $\jmath_{l,m}$ is contractive. 
\end{proof}
Notice that since $\jmath_{l,m}$ intertwines the normalized traces, the unitality assumption on $\jmath_{l,m}$ is equivalent to assuming that $\jmath_{l,m}$ is contractive.
\begin{prop}\label{propwithstatescompat}
Let $m,l\in\N_0$ with $m\leq l$. Then $\jmath_{l,m}$ is the adjoint of $\iota_{m,l}$: for all $A\in\Bi(\GH_l)$ and all $B\in\Bi(\GH_m)$ we have
\begin{equation}\label{limmapsadjoint}
\phi_l(A\iota_{m,l}(B)\big)=\phi_m\big(\jmath_{l,m}(A)B\big).
\end{equation}
In particular, taking $A=p_l$ respectively $B=p_m$ we obtain the equivalences
\begin{equation}\label{iotaandphi}
\phi_l\circ\iota_{m,l}=\phi_m\iff\jmath_{l,m}(p_l)=p_m,
\end{equation}\begin{equation}\label{jmathandphi}
\phi_l=\phi_m\circ\jmath_{l,m}\iff\iota_{m,l}(p_m)=p_l,
\end{equation}
where \eqref{jmathandphi} holds for any subproduct system while \eqref{iotaandphi} is our standing assumption for this section.
\end{prop}
\begin{proof} We have
\begin{align*}
\phi_l(A\iota_{m,l}(B))&=\phi_l(AV_{m,l}^*(B\otimes p_{l-m})V_{m,l})
\\&=(\phi_m\otimes\phi_{l-m})(V_{m,l}AV_{m,l}^*(B\otimes p_{l-m}))
\\&=\phi_m\big((\id_{\Bi_m}\otimes\phi_{l-m})(V_{m,l}AV_{m,l}^*(B\otimes p_{l-m}))\big)
\\&=\phi_m\big((\id_{\Bi_m}\otimes\phi_{l-m})(V_{m,l}AV_{m,l}^*)B\big),
\end{align*}
which equals $\phi_m\big(\jmath_{l,m}(A)B\big)$ by \eqref{projlimaspartialtrace}.
\end{proof}

\begin{cor} The states $\phi_m$ satisfy the ``right invariance" condition
\begin{equation}\label{invcondofstateiota}
\phi_m(\jmath_{l,m}\circ\iota_{m,l}(A))=\phi_m(A),\qquad\forall A\in\Bi(\GH_m).
\end{equation}
\end{cor}
\begin{proof} Just use \eqref{jmathandphi} and then \eqref{iotaandphi}.
\end{proof}

\begin{Remark}\label{chiralrem} Similarly one shows that 
\begin{equation}\label{statesrightcompatible}
\phi_m=\phi_l\circ\bar{\iota}_{m,l}.
\end{equation}
for the ``right" inductive system $\bar{\iota}_{\bullet,\bullet}$. The adjoint of $\bar{\iota}_{m,l}$ is
\begin{align*}
\bar{\jmath}_{l,m}(A)&=\frac{\Tr(Q_m)}{\Tr(Q_l)}\sum_{|\mathbf{r}|=l-m}(Q^{\otimes m})_{\mathbf{r},\mathbf{r}}S_\mathbf{r}^*AS_\mathbf{r}\big|_{\GH_m}=(\phi_{l-m}\otimes\id)(\bar{V}_{m,l}A\bar{V}_{m,l}^*).
\end{align*}
A ``left invariance" condition similar to \eqref{invcondofstateiota} is also deduced using the $\bar{\iota}_{m,l}$'s and their adjoints.
\end{Remark}

\subsection{The state on $\Oi_\GH$}\label{quasifreesection}

\begin{cor} The limit 
$$
\phi_\infty:=\lim_{m\to\infty}\phi_m
$$
is a well-defined state on $\Ti_\GH^{(0)}$. It annihilates $\Ti_\GH^{(0)}\cap\Ki$, so it descends to a state $\omega_Q$ on $\Oi_\GH^{(0)}=\Bi_\infty$.
\end{cor}
\begin{proof} For well-definedness we use \eqref{iotaandphi} and recall that the elements of $\Ti_\GH^{(0)}$ are norm limits of eventually constant under $\iota_{\bullet,\bullet}$. The fact that $\phi_\infty$ descends to $\Ti_\GH^{(0)}/(\Ti_\GH^{(0)}\cap\Ki)$ follows from
$$
|\phi_m(A)|\leq\|A\|,\qquad \forall A\in\Bi(\GH_m),m\in\N_0,
$$
since this shows that $\lim_m\|A_m\|=0$ implies $\phi_\infty(A_\bullet)=0$ for all $A_\bullet=(A_m)_{m\in\N_0}\in\Ti_{\GH}^{(0)}$. \end{proof}

\begin{prop}
The state $\omega_Q:\Bi_\infty\to\C$ is KMS, with modular automorphism group $\sigma_\bullet=(\sigma_t)_{t\in\R}$ given by
$$
\sigma_t\circ\varsigma^{(m)}(A)=\varsigma^{(m)}(Q_m^{it}AQ_m^{-it})
$$
for all $A\in\Bi(\GH_m)$ and all $m\in\N_0$, and $\omega_Q$ satisfies
\begin{equation}\label{quasifreeQ}
\omega_Q(Z_\mathbf{j}Z_\mathbf{k}^*)=\frac{Q_{\mathbf{k},\mathbf{j}}}{\Tr(Q_m)}
\end{equation}
for all $\mathbf{j},\mathbf{k}\in\F_n^+$ with $|\mathbf{j}|=|\mathbf{k}|=m$. Moreover, the covariant symbol map $\varsigma^{(m)}:\Bi(\GH_m)\to\Bi_\infty$ intertwines $\omega_Q$ and $\phi_m$:
\begin{equation}\label{covsymbandlimitstate}
\omega_Q\circ\varsigma^{(m)}=\phi_m.
\end{equation}
\end{prop}
\begin{proof} Due to \eqref{iotaandphi} we have, if $|\mathbf{j}|=|\mathbf{k}|=m$, 
\begin{align*}
\omega_Q(Z_\mathbf{j}Z_\mathbf{k}^*)&=\phi_\infty(S_\mathbf{j}S_\mathbf{k}^*)
\\&=\lim_{m\leq l\to\infty}\phi_l(S_\mathbf{j}S_\mathbf{k}^*p_l)
\\&=\phi_m(S_\mathbf{j}S_\mathbf{k}^*p_m),
\end{align*}
and so the first formula in \eqref{quasifreeQ} follows from
\begin{align*}
\Tr(Q_m)\phi_m(S_\mathbf{j}S_\mathbf{k}^*)&=\sum_{|\mathbf{r}|=m}(Q^{\otimes m})_{\mathbf{r},\mathbf{r}}\bra e_\mathbf{r}|p_me_\mathbf{j}\ket\bra e_\mathbf{k}|p_me_\mathbf{r}\ket
\\&=\bra e_\mathbf{k}|Q_me_\mathbf{j}\ket.
\end{align*}
The definition of $\varsigma^{(m)}$ immediately gives \eqref{covsymbandlimitstate}, again using \eqref{iotaandphi}.

That $\omega_Q$ is KMS follows from \eqref{covsymbandlimitstate}, in view of the fact that the $*$-algebra generated by the covariant symbols $\varsigma^{(m)}(A)$ is dense in $\Bi_\infty$ and that each $\phi_m$ is KMS. Finally, for $t\in\R$ the modular automorphism $\sigma_t^{\phi_m}$ of $\phi_m$ takes $A\in\Bi(\GH_m)$ to $(\rho^{(m)})^{it}A(\rho^{(m)})^{-it}=Q_m^{it}AQ_m^{-it}$. 
\end{proof}

We can extend $\omega_Q$ to a state, still denoted by $\omega_Q$, on the whole Cuntz--Pimsner algebra by defining it to be zero on each spectral subspace $\Oi^{(k)}_\GH$ except $\Oi^{(0)}_\GH$.
\begin{Remark}\label{faithfulremark}
The property \eqref{covsymbandlimitstate} of the limit state relies on Assumption \ref{Qnotat}(ii) and ensures that $\omega$ is faithful. Without this assumption we could still obtain a limit $\omega_Q$ of the states $\phi_m$ but it is not clear what would guarantee its faithfulness. We could go the GNS representation of $\Bi_\infty$ associated with $\omega_Q$ and use the faithful state induced by $\omega_Q$ on the image of $\Bi_\infty$, which is a quotient of $\Bi_\infty$ (recall that $\omega_Q$ is KMS). Then analogous results hold for the image of $\Bi_\infty$ in the GNS representation.
\end{Remark}

\begin{Example} For the product system $\GH^{\otimes\bullet}$, the Cuntz--Pimsner algebra $\Oi_\GH$ is the Cuntz algebra $\Oi_n$ and $\omega_Q$ is the quasi-free state on $\Oi_n$ defined by the density matrix $Q/\Tr(Q)$ \cite{Ev1}.
\end{Example}

\subsection{Contravariant symbols}
\begin{prop} 
The adjoint $\breve{\varsigma}^{(m)}:\Bi_\infty\to\Bi(\GH_m)$ of the covariant symbol map $\varsigma^{(m)}:\Bi(\GH_m)\to\Bi_\infty$, defined by the relation
$$
\omega_Q(\varsigma^{(m)}(A)^*f)=\phi_m\big(A^*\breve{\varsigma}^{(m)}(f)\big),\qquad\forall  A\in\Bi(\GH_m),f\in\Bi_\infty,
$$
is given by
\begin{align}
\breve{\varsigma}^{(m)}(f)&=\Tr(Q_m)\sum_{|\mathbf{j}|,|\mathbf{k}|=m}(Q^{\otimes m})^{-1}_{\mathbf{j},\mathbf{j}}\omega_Q(Z_\mathbf{j}Z_\mathbf{k}^*f)S_\mathbf{k} S_\mathbf{j}^*\big|_{\GH_m}\label{sumexprcontravarpos}
%\\&=\Tr(Q_m)\sum_{|\mathbf{j}|,|\mathbf{k}|=m}\omega_Q(Z_\mathbf{k}^*fZ_\mathbf{j})S_\mathbf{k} S_\mathbf{j}^*\big|_{\GH_m}\label{seconbreveform}.
\end{align}
\end{prop}
\begin{proof} Let $\breve{\varsigma}^{(m)}(f)$ be defined by \eqref{sumexprcontravarpos}. Then
\begin{align*}
\omega_Q(\varsigma^{(m)}(A)^*f)&=\sum_{|\mathbf{j}|=m=|\mathbf{k}|}(A^*)_{\mathbf{j},\mathbf{k}} \omega_Q(Z_\mathbf{j}Z_\mathbf{k}^*f)
\\&=\sum_{|\mathbf{j}|=m=|\mathbf{k}|}(Q^{\otimes m})_{\mathbf{j},\mathbf{j}}(Q^{\otimes m})^{-1}_{\mathbf{j},\mathbf{j}}(A^*)_{\mathbf{j},\mathbf{k}}\omega_Q(Z_\mathbf{j}Z_\mathbf{k}^*f)
\\&=\phi_m\big(A^*\breve{\varsigma}^{(m)}(f)\big).
\end{align*}
%We then obtain \eqref{seconbreveform} using the KMS condition. 
\end{proof}
\begin{cor} We have
\begin{equation}\label{omegaandphiposm}
\omega_Q=\phi_m\circ\breve{\varsigma}^{(m)}
\end{equation}
and, moreover, each $\breve{\varsigma}^{(m)}$ is unital.
\end{cor}
\begin{proof} Equation \eqref{omegaandphiposm} is a direct consequence of the fact that $\breve{\varsigma}^{(m)}$ is adjoint to the unital map $\varsigma^{(m)}$. Unitality of $\breve{\varsigma}^{(m)}$ follows from $\omega_Q(\varsigma^{(m)}(A)^*\bone)=\phi_m(A^*)$, which we know from \eqref{covsymbandlimitstate}. %We can also deduce $\breve{\varsigma}^{(m)}(\bone)=p_m$ from \eqref{seconbreveform} if we use \eqref{quasifreeQ}. 
\end{proof}

We can now assembly the $\breve{\varsigma}^{(m)}$'s to a map
\begin{equation}\label{totcovarsymb}
\breve{\varsigma}:=\prod_{m\in\N_0}\breve{\varsigma}^{(m)}:\Bi_\infty\to\prod_{m\in\N_0}\Bi(\GH_m),
\end{equation}
which is a noncommutative generalization of the total Toeplitz map \eqref{tottoeplitz}. We can recover its components as
$$
\breve{\varsigma}^{(m)}(f)=\breve{\varsigma}(f)p_m.
$$

The following result which relies on the fact that $\omega_Q$ is faithful (cf. Remark \ref{faithfulremark}).
\begin{Lemma}\label{faithToepllemma}
 No nonzero element of $\Bi_\infty$ is mapped to $\Gamma_0=\Ti^{(0)}_\GH\cap\Ki$ under the map $\breve{\varsigma}$. 
\end{Lemma}
\begin{proof} We have $\phi_m\big(\breve{\varsigma}^{(m)}(f)\big)=\omega_Q(f)$, so if $\breve{\varsigma}^{(m)}(f)\to 0$ as $m\to\infty$ then $\omega_Q(f)=0$. Hence if $f\geq 0$ then $f=0$ and the result follows. 
\end{proof}
Let $\Mi=\pi_{\omega_Q}(\Bi)''$ be the von Neumann algebra generated by the inductive limit $\Bi_\infty$ in the GNS representation of the limit state $\omega_Q$. Then we can define $\breve{\varsigma}(f)\in\Gamma_b$ also for elements in $\Mi$, and Lemma \ref{faithToepllemma} extends to $\Mi$.

\begin{Lemma}\label{Toeplcompatprojlimlemma} For all $f\in\Mi$ and all $l\geq m$,
\begin{equation}\label{projectsystcontrsymb}
\breve{\varsigma}^{(m)}(f)=\jmath_{l,m}\circ\breve{\varsigma}^{(l)}(f).
\end{equation}
Hence the image of $\Mi$ under the total Toeplitz map $\breve{\varsigma}$ is contained in the projective limit $\Bi^\infty$, and in fact we have equality
$$
\breve{\varsigma}(\Mi)=\Bi^\infty.
$$
Therefore $\Bi^\infty$ can be identified with the weak-$*$-closed operator system of elements of the form
$$
\breve{\varsigma}(f)=(\breve{\varsigma}^{(m)}(f))_{m\in\N_0},\qquad f\in\Mi
$$
and, as in Remark \ref{projlimremarktwo}, 
$$
\breve{\varsigma}^{(m)}=\jmath_{\infty,m}
$$
is the map which evaluates $(X_m)_{m\in\N_0}\in\Bi^\infty$ at $m\in\N_0$. The norm-closed subset $\breve{\varsigma}(\Bi_\infty)$ equals the anti-normally ordered part of the Toeplitz core. 
\end{Lemma}
\begin{proof}
We know that $\breve{\varsigma}$ is injective (Lemma \ref{faithToepllemma}). Since we have shown that $\jmath_{l,m}$ is adjoint to $\iota_{m,l}$, we obtain \eqref{projectsystcontrsymb} by taking adjoints of
$$
\varsigma^{(l)}\circ \iota_{m,l}=\varsigma^{(m)}.
$$
From \eqref{projectsystcontrsymb} follows that $\breve{\varsigma}(f)\in\Bi^\infty$ for all $f\in\Mi$. Moreover, $\breve{\varsigma}:\Mi\to\Gamma_b$ is onto $\Bi^\infty$ because each $\breve{\varsigma}^{(m)}$ is onto. Thus $\Bi^\infty$ is in bijection with $\Mi$ via $\breve{\varsigma}$. 

We need to show that $\breve{\varsigma}(\Bi_\infty)$ equals the anti-normally part of the Toeplitz core $\Ti^{(0)}_\GH$. Firstly, since the left and right shifts commute outside the vacuum subspace, for all $r,s=1,\dots,n$ we have
\begin{align*}
\jmath_{l,m}(S_r^*S_sp_l)&=\frac{\Tr(Q_m)}{\Tr(Q_l)}\sum_{|\mathbf{k}|=l-m}(Q^{\otimes m})_{\mathbf{k},\mathbf{k}}R_\mathbf{k}^*S_r^*S_s R_\mathbf{k}\big|_{\GH_m}
\\&=\frac{\Tr(Q_m)}{\Tr(Q_l)}\sum_{|\mathbf{k}|=l-m}(Q^{\otimes m})_{\mathbf{k},\mathbf{k}}S^*_rR_\mathbf{k}^* R_\mathbf{k}S_s\big|_{\GH_m}=S_r^*S_s|_{\GH_m}
\end{align*}
(where we used that $\jmath_{l,m}(p_l)=p_m$), which shows that the anti-normally ordered elements of $\Ti^{(0)}_\GH$ are constant under $\jmath_{\bullet,\bullet}$. Secondly, an explicit calculation using \eqref{sumexprcontravarpos} shows that $\breve{\varsigma}(f)$ is anti-normally ordered for each $f\in\Bi_\infty$. 
\end{proof}

\begin{Remark} Now we can give an alternative proof for the fact that the contravariant symbol map $\breve{\varsigma}^{(l)}:\Bi_\infty\to\Bi(\GH_m)$ intertwines $\omega_Q$ with $\phi_l$,
$$
\omega_Q=\phi_l\circ \breve{\varsigma}^{(l)}.
$$
Recall that $\omega_Q$ denotes the limit state $\phi_\infty:=\lim_{m\to\infty}\phi_m$ when regarded as a state on the quotient $\Bi_\infty$ of $\Ti^{(0)}_\GH$. Then $\omega_Q=\phi_l\circ \breve{\varsigma}^{(l)}$ follows from the compatibility $\phi_m=\phi_l\circ\jmath_{m,l}$ (see \eqref{iotaandphi}) and the fact that $\jmath_{\infty,l}= \breve{\varsigma}^{(l)}$. 
\end{Remark}
%The following is clear.
%\begin{Lemma}\label{Toeplisisometric} The total Toeplitz map $\breve{\varsigma}:\Mi\to\Bi^\infty$ is an isometry: for all $f\in\Mi$ we have
%$$
%\lim_{m\to\infty}\|\breve{\varsigma}^{(m)}(f)\|=\|f\|.
%$$
%\end{Lemma}
%\begin{proof}
%Clear from the identification of $\Bi^\infty$ with the weak closure of the anti-normally ordered part of the Toeplitz algebra. 
%\end{proof}

\subsection{The asymptotic multiplication}\label{asympmultsec}
We now endow the projective limit $\Bi^\infty$ with a multiplication which is the $m\to\infty$ limit of the multiplication on $\Bi(\GH_m)$. 
\begin{dfn} The \textbf{projective-limit multiplication} on $\Bi^\infty\subset\Gamma_b$ is defined by
\begin{equation}\label{projlimmult}
\breve{\varsigma}(f)\cdot\breve{\varsigma}(g):=\lim_{m\to\infty}\breve{\varsigma}^{(m)}(fg)
\end{equation}
for all $f,g\in\Bi_\infty$. 
\end{dfn}
The projective limit $\Bi^\infty$ is not an algebra under the projective-limit multiplication, but we shall see that the subset $\breve{\varsigma}(\Bi_\infty)$ is. 

The multiplication on $\Bi^\infty$ taken modulo $\Gamma_0$ is the one where sequences $(\breve{\varsigma}^{(m)}(f))_{m\in\N_0}$ and $(\breve{\varsigma}^{(m)}(g))_{m\in\N_0}$ are multiplied componentwise but the finite-$m$ part is ignored. That is,
\begin{equation}\label{modGammazeromult}
\pi(\breve{\varsigma}(f)\breve{\varsigma}(g))=\lim_{m\to\infty}\breve{\varsigma}^{(m)}(f)\breve{\varsigma}^{(m)}(g).
\end{equation}
We will see momentarily that the products \eqref{projlimmult} and \eqref{modGammazeromult} coincide for $f,g\in\Bi_\infty$. Comparing the two formulas one then concludes that the Toeplitz maps $\breve{\varsigma}^{(m)}$ are ``asymptotically multiplicative". Again the projective limit $\Bi^\infty$ is not an algebra under the multiplication modulo compacts, while $\breve{\varsigma}(\Bi_\infty)$ will be shown to be so.

\begin{Remark}[Filters] A projective-limit multiplication can be defined using any filter $\omega$ on $\N$. On the $C^*$-level this corresponds to considering not a subalgebra of $\Gamma_b/\Gamma_0$ but a subalgebra of $\Gamma_b/\Gamma_\omega$ where $\Gamma_\omega$ is the ideal consisting of the sequences $A_\bullet$ with
$$
\lim_\omega\|A_m\|=0.
$$
We recover $\Gamma_0$ if $\omega$ is the free filter of all cofinite subsets of $\N_0$. Confer \cite[§6.2]{RoSt1}. 
\end{Remark}

\subsection{The adjoint of the total Toeplitz map}\label{adjointtotsec}
\begin{Lemma}
For $A\in\Bi(\GH_l)$ we have
\begin{equation}\label{tracefromprojectsyst}
\jmath_{l,0}(A)=\phi_l(A)p_0.
\end{equation}
and hence, if $\hat{\ep}$ denotes the vacuum state restricted to $\Ti_\GH^{(0)}$,
$$
\hat{\ep}\circ \jmath_{l,0}=\phi_l.
$$
The vacuum state $\hat{\ep}$ restricted to $\Bi^\infty$ is equal to $\hat{\ep}\circ \jmath_{\infty,0}$ and coincides with the limit state $\phi_\infty:=\lim_{l\to\infty}\phi_l$,
$$
\hat{\ep}\circ \jmath_{\infty,0}=\phi_\infty.
$$ 
\end{Lemma}
\begin{proof} We use $\phi_m\circ\jmath_{l,m}=\phi_l$ for $m=0$. This gives \eqref{tracefromprojectsyst}. Alternatively, note that for each positive operator $A$ on $\Bi(\GH_m)$ one has
$$
\Tr(A)=\sum_{|\mathbf{k}|=m}A_{\mathbf{k},\mathbf{k}}
$$
where $A_{\mathbf{k},\mathbf{k}}:=\bra p_me_\mathbf{k}|A p_me_\mathbf{k}\ket$. Consequently,
$$
\phi_m(A)=\frac{1}{\dim\GH_m}\sum_{|\mathbf{k}|=m}\bra\Omega|S_\mathbf{k}^*AS_\mathbf{k}\Omega\ket=\phi_0(\jmath_{m,0}(A)).
$$
Letting $m\to\infty$ one obtains
$$
\omega(f)=\phi_0(\breve{\varsigma}^{(0)}(f)),\qquad \forall f\in C^0(\M), \ f\geq 0.
$$
The rest is obvious. 
\end{proof}
We can therefore regard $\hat{\ep}$ as a state on the projective limit modulo compact operators as well, i.e. on the algebra $\pi(\Bi^\infty)\subset\Gamma_b/\Gamma_0$.

Recall that the covariant-symbol map $\varsigma^{(m)}$ is the adjoint of $\breve{\varsigma}^{(m)}$, for each $m\in\N$. We now show that the total Toeplitz map $\breve{\varsigma}$ has an adjoint as well. This should be compared with \cite[Lemma 2.3]{INT1}.
\begin{prop} There exists a completely positive map 
$$
\breve{\varsigma}^*:\breve{\varsigma}(\Bi_\infty)\to\Bi_\infty
$$
such that, for all $X\in\breve{\varsigma}(\Bi_\infty)$ and all $f\in\Bi_\infty$,
\begin{equation}\label{defoftotToepadj}
\omega_Q(\breve{\varsigma}^*(X^*)f)=\hat{\ep}(\pi(X^*)\pi(\breve{\varsigma}(f))).
\end{equation}
Explicitly, this map is given by the point-norm limt $\breve{\varsigma}^*=\lim_{m\to\infty}\varsigma^{(m)}$,
$$
\breve{\varsigma}^*(X)=\lim_{m\to\infty}\varsigma^{(m)}(Xp_m),\qquad \forall X\in\breve{\varsigma}(\Bi_\infty),
$$
and will be denoted by $\varsigma$.
\end{prop}

\begin{proof} We identify $X\in\breve{\varsigma}(\Bi_\infty)$ with a bounded sequence $(X_m)_{m\in\N_0}$ of operators $X_m=Xp_m\in\Bi(\GH_m)$. Using the formula \eqref{modGammazeromult} for the multiplication in $\pi(\Bi^\infty)$ we have, by norm-continuity of the vacuum state, the norm limits
\begin{align*}
\hat{\ep}(\pi(X^*)\pi(\breve{\varsigma}(f)))p_0&=\big\bra \Omega\big|\lim_{m\to\infty}X_m^*\breve{\varsigma}^{(m)}(f)\Omega\big\ket p_0
\\&=\jmath_{\infty,0}\big(\lim_{m\to\infty}X_m^*\breve{\varsigma}^{(m)}(f)\big)
\\&=\lim_{m\to\infty}\jmath_{m,0}\big(X_m^*\breve{\varsigma}^{(m)}(f)\big)
\\&=\lim_{m\to\infty}\phi_m\big(X_m^*\breve{\varsigma}^{(m)}(f)\big)p_0
\\&=\lim_{m\to\infty}\omega_Q\big(\varsigma^{(m)}(X_m^*)f\big)p_0
\\&=\omega_Q\Big(\lim_{m\to\infty}\varsigma^{(m)}(X_m^*)f\Big)p_0.
\end{align*}
Being a point-norm limit of completely positive maps, $\varsigma$ is completely positive.
\end{proof}
Since $\Bi^\infty=\breve{\varsigma}(\Mi)$ contains no compact operator, it is clear from the definition of the covariant symbol that $\varsigma$ restricts to a bijection from $\Bi^\infty$ onto $\Mi$ and that $\varsigma\circ\breve{\varsigma}$ is the identity on $\Mi$. 
%Since both $\varsigma$ and $\breve{\varsigma}$ are unital, they are both contractions. 
Hence $\breve{\varsigma}$ is an isometry. So we have a decomposition of the identity map on $\Bi_\infty$,
$$
\id=\varsigma\circ\breve{\varsigma}=\lim_{m\to\infty}\varsigma^{(m)}\circ\breve{\varsigma}^{(m)},
$$
making Remark \ref{nucremark} explicit. We have now seen that $\breve{\varsigma}:\Bi_\infty\to\breve{\varsigma}(\Bi_\infty)$ is a complete order isomorphism, i.e. a bijective unital completely positive map with completely positive inverse.

There is also a version of this result on the level of von Neumann algebras. 
As we shall see in \S\ref{markovsec}, for any subproduct system $\GH_\bullet$, the weak-$*$-closed operator system $\Bi^\infty$ becomes a von Neumann algebra when equipped with a SOT-version of the projective-limit multiplication \eqref{projlimmult}. When $\GH_\bullet$ is the $\G$-subproduct system (see \S\ref{CMQGappsec} below), $\Bi^\infty$ is an operator system in the group-von Neumann algebra $\Ri(\G)$.
\begin{cor}\label{corinjtoep}
The total Toeplitz map $\breve{\varsigma}$ intertwines the state $\omega_Q$ on $\Mi=\pi_{\omega_Q}(\Bi_\infty)''$ with the vacuum state $\hat{\ep}$ on $\Bi^\infty\subset\Bi(\GH_\N)$,
\begin{equation}\label{intertwinestatetotToepl}
\omega_Q=\hat{\ep}\circ\breve{\varsigma},
\end{equation}
and similarly
\begin{equation}\label{intertwinestateadjointToepl}
\omega_Q\circ\varsigma=\hat{\ep},
\end{equation}
\end{cor}
\begin{proof} Take $X=\bone$ respectively $f=\bone$ in \eqref{defoftotToepadj} to get \eqref{intertwinestatetotToepl} respectively \eqref{intertwinestateadjointToepl}. 
\end{proof}

\subsection{$\Oi_\GH^{(0)}$ as a projective limit}\label{CMasprojlimsec}

\begin{thm}\label{indlimtasprojlimthm}
The operator system $\breve{\varsigma}(\Bi_\infty)\subset\Bi^\infty$ is a $C^*$-algebra with the multiplication taken modulo compacts; indeed $\pi(\breve{\varsigma}(\Bi_\infty))$ is isomorphic as a $C^*$-algebra to the inductive limit $\Bi_\infty\cong\Oi_\GH^{(0)}$.
\end{thm}
\begin{Remark} We saw in Lemma \ref{faithToepllemma} that the image of $\breve{\varsigma}$ does not contain $\Gamma_0$. On the other hand, $\breve{\varsigma}(\Bi_\infty)+\Gamma_0$ is generated by $\operatorname{Ran}\breve{\varsigma}$ alone as a $C^*$-algebra. The theorem implies that 
$$
\breve{\varsigma}(f)\breve{\varsigma}(g)-\breve{\varsigma}(fg)\in \Gamma_0
$$
even for nontrivial $f,g\in\Bi_\infty$. These elements do not belong to $\Bi^\infty$, which is why we need to apply the quotient map $\pi:\Gamma_b\to\Gamma_b/\Gamma_0$ in order to obtain an algebra.  
\end{Remark}
\begin{proof} This is a well-known consequence of the fact that $\breve{\varsigma}:\Bi_\infty\to\breve{\varsigma}(\Bi_\infty)$ is a complete order isomorphism from a $C^*$-algebra $\Bi_\infty$ onto the operator system $\breve{\varsigma}(\Bi_\infty)$; see \cite[Prop. 2.2]{Arv10}. 
\end{proof}
\begin{cor}\label{lemmaprojlimmult}
The projective-limit multiplication on $\breve{\varsigma}(\Bi_\infty)\subset\Gamma_b$ coincides with the multiplication on $\breve{\varsigma}(\Bi_\infty)$ taken modulo $\Gamma_0$.
\end{cor}
\begin{proof}
By uniqueness of the $C^*$-algebraic structure we know that any two multiplications on $\breve{\varsigma}(\Bi_\infty)$ compatible with the norm must be isomorphic. But as remarked in \S\ref{asympmultsec}, the \emph{exact} equality of the two products at hand is equivalent to the statement that $\breve{\varsigma}$ is multiplicative modulo $\Gamma_0$, whence the result. 
\end{proof}

Recall that it is the normally ordered elements of $\Ti^{(0)}_\GH$ which are constant under the inductive system $\iota_{\bullet,\bullet}$. The partial inverse $\varsigma_{(m)}:\varsigma^{(m)}(\Bi_m)\to\Bi_m$ of the covariant symbol map $\varsigma^{(m)}$ gives a normally ordered ``quantization" of $\Bi_\infty$, and the image of $\prod_{m\in\N_0}\varsigma_{(m)}$ is the normally ordered part of the Toeplitz core, which is \emph{all} of $\Ti^{(0)}_\GH$. In contrast, the projective limit $\Bi^\infty$ contains (as an operator space) only the anti-normally ordered elements in $\Ti^{(0)}_\GH$, so ``Toeplitz quantization" gives the anti-normal ordering. Lemma \ref{Lemmanormorder} shows that $\pi^{-1}(\pi(\breve{\varsigma}(\Bi_\infty)))=\breve{\varsigma}(\Bi_\infty)+\Gamma_0$ nevertheless gives all of $\Ti^{(0)}_\GH$.

\subsection{Strict quantization}
Some authors (\cite{Hawk3}, \cite{Rie1}, \cite{Sain}) do not require commutativity of the ``classical limit algebra" in an axiomatic approach to ``strict quantization". Adapting such a definition, we can show that what we have done here is a strict quantization.

Let $^{0}\Bi_\infty$ denote the $*$-algebra generated by the $\varsigma^{(m)}(A)$'s for all $A\in\Bi(\GH_m)$ and all $m\in\N_0$; thus $^{0}\Bi_\infty$ is a dense $*$-subalgebra (the ``algebraic part") of the inductive-limit $C^*$-algebra $\Bi_\infty$.

\begin{dfn} The \textbf{Berezin product} on ${^{0}\Bi_\infty}$ is defined for all $f,g\in {^{0}\Bi_\infty}$ by
$$
f\overset{(m)}{\star}g:=\varsigma^{(m)}\big(\varsigma_{(m)}(f)\varsigma_{(m)}(g)\big),
$$
where $\varsigma_{(m)}:\varsigma^{(m)}(\Bi_m)\to\Bi_m$ denotes the partial inverse of $\varsigma^{(m)}$. The \textbf{Poisson bracket} on ${^{0}\Bi_\infty}$ is defined by (cf. \cite[§D.1]{Hawk1})
\begin{equation}\label{Poissbrackindlim}
\{f,g\}:=\lim_{m\to\infty}\frac{m}{\sqrt{-1}}(f\overset{(m)}{\star}g-g\overset{(m)}{\star}f).
\end{equation}
\end{dfn}
Of course we do not expect \eqref{Poissbrackindlim} to be a Poisson bracket in the ordinary sense, and $\{\cdot,\cdot\}$ is not likely to be interesting unless $\GH_\bullet$ is commutative. 
\begin{cor}\label{strictquantcor} The sequence $(\Bi_\bullet,\breve{\varsigma}^{(\bullet)})=(\Bi_m,\breve{\varsigma}^{(m)})_{m\in\N_0}$ is a strict quantization of ${^{0}\Bi_\infty}$ in the sense that each $\breve{\varsigma}^{(m)}$ is surjective and, for all $f,g\in {^{0}\Bi_\infty}$,
\begin{equation}\label{contrsymbvNcond}
\lim_{m\to\infty}\|\breve{\varsigma}^{(m)}(fg)-\breve{\varsigma}^{(m)}(f)\breve{\varsigma}^{(m)}(g)\|=0,
\end{equation}
\begin{equation}\label{contrsymbRiefcond}
\lim_{m\to\infty}\|\breve{\varsigma}^{(m)}(f)\|=\|f\|,
\end{equation}
\begin{equation}\label{contrsymbDiraccond}
\lim_{m\to\infty}\|m^{-1}[\breve{\varsigma}^{(m)}(f),\breve{\varsigma}^{(m)}(g)]-\{f,g\}]\|=0.
\end{equation}
\end{cor}
\begin{proof} We have seen that $\pi\circ\breve{\varsigma}:\Bi_\infty\to\Ti^{(0)}_\GH/\Gamma_0$ is injective, so
$$
\breve{\varsigma}(f)\breve{\varsigma}(g)-\breve{\varsigma}(fg)\in\Gamma_0,
$$
which is equivalent to von Neumann's condition \eqref{contrsymbvNcond}. Rieffel's condition \eqref{contrsymbRiefcond} coincides with the definition of the norm on $\Bi^\infty$. Similarly, the Dirac condition \eqref{contrsymbDiraccond} is tautology in view of our definition of $\{\cdot,\cdot\}$ in \eqref{Poissbrackindlim}.
\end{proof}

We have seen that a subproduct system $\GH_\bullet$ comes with a sequence $\Bi_\bullet=(\Bi_m)_{m\in\N_0}$ of finite-dimensional algebras $\Bi_m:=\Bi(\GH_m)$, to which we can add $\Bi_\infty\cong\Oi_\GH^{(0)}$, and two sequences $\varsigma^{(\bullet)}=(\varsigma^{(m)})_{m\in\N_0}$ and  $\breve{\varsigma}^{(\bullet)}=(\breve{\varsigma}^{(m)})_{m\in\N_0}$ of positive unital maps
$$
\varsigma^{(m)}:\Bi_m\to\Bi_\infty,\qquad \breve{\varsigma}^{(m)}:\Bi_\infty\to\Bi_m
$$
such that $\varsigma^{(m)}\circ \breve{\varsigma}^{(m)}$ converges to the identity map on $\Bi_\infty$. As in \cite[Prop. 2.2]{Sain} we can associate to this data a continuous field of $C^*$-algebras, making explicit the assertion in Remark \ref{contfieldrem}.
\begin{cor}
The $C^*$-algebra $\Ti_\GH^{(0)}$ can be identified with the space of continuous sections of the continuous field $\N_0\cup\{\infty\}\ni m\to \Bi_m$, i.e.
$$
\Ti_\GH^{(0)}\cong\Big\{(X_m)_{m\in\N_0\cup\{\infty\}}\in\prod_{m\N_0\cup\{\infty\}}\Bi_m\Big|\ X_\infty=\lim_{m\to\infty}\varsigma^{(m)}(X_m)\Big\}.
$$
\end{cor}
\begin{proof} We have seen that $\varsigma:\breve{\varsigma}(\Bi_\infty)\to \Bi_\infty$ can be obtained as
$$
\varsigma(X)=\lim_{m\to\infty}\varsigma^{(m)}(Xp_m).
$$
Thus, the $C^*$-algebra of continuous sections of $\Bi_\bullet$ consists of the image of $\breve{\varsigma}(\Bi_\infty)$ under $\varsigma$ together with the sequences $(X_m)_{m\in\N_0\cup\{\infty\}}$ such that $X_\infty=0$. Hence the result follows from the facts that $\breve{\varsigma}(\Bi_\infty)+\Gamma_0=\Ti_\GH^{(0)}$ and that $\varsigma$ is an isomorphism. 
\end{proof}

\subsection{$\Oi_\GH$ assembled from projective limits}
We now define modules over the projective limit $\Bi^\infty$.

Recall that we defined in §\ref{pimsindlimsec} an inductive system $\iota_{m,l}^{(k)}:\Bi(\GH_m,\GH_{m+k})\to\Bi(\GH_l,\GH_{l+k})$. Define the adjoint $\jmath_{l,m}^{(k)}:\Bi(\GH_{l+k},\GH_l)\to\Bi(\GH_{m+k},\GH_m)$ of $\iota_{m,l}^{(k)}$ by the property that 
$$
\phi_m\big(\jmath_{l,m}^{(k)}(X)Y\big)=\phi_l\big(X\iota_{m,l}^{(k)}(Y)\big)
$$
for all $X\in \Bi(\GH_{l+k},\GH_l)$ and all $Y\in\Bi(\GH_m,\GH_{m+k})$. We deduce that
$$
\jmath_{l,m}^{(k)}(X)=\frac{\Tr(Q_m)}{\Tr(Q_l)}\sum_{|\mathbf{r}|=l-m}(Q^{\otimes m})_{\mathbf{r},\mathbf{r}}R_\mathbf{r}^*X R_\mathbf{r}\big|_{\GH_m},\qquad \forall X\in\Bi(\GH_{l+k},\GH_l),
$$
and that the opertors on $\GH_\N$ which are constant with respect to the system $\jmath_{\bullet,\bullet}^{(k)}$ are precisely those of the form $\breve{\varsigma}_k(f)=(\breve{\varsigma}_k^{(l)}(f))_{l\in\N_0}$ for some $f\in\Bi_\infty$, where 
$$
\breve{\varsigma}_k^{(l)}(f)=\Tr(Q_l)\sum_{|\mathbf{j}|=l,|\mathbf{k}|=l+k}(Q^{\otimes m})^{-1}_{\mathbf{j},\mathbf{j}}\omega_Q(Z_\mathbf{j}Z_\mathbf{k}^*f)S_\mathbf{j}S_\mathbf{k}^*\big|_{\GH_l}.
$$
Define $\Ei_{(k)}$ to be the set of $\breve{\varsigma}_k(f)$'s for all $f\in\Bi_\infty$. Then $\Ei_{(k)}$ is an operator system. 

Let $\Gamma_0^{(k)}$ be the vector space of sequences of operators in $\Bi(\GH_{l+k},\GH_l)$ which converge to zero as $l\to\infty$. 
\begin{prop} The vector space $\Ei_{(k)}+\Gamma_0^{(k)}$ coincides with the subspace $\Ti^{(-k)}_\GH$ of the Toeplitz algebra. Hence $\Ei_{(k)}$ is a module over the $C^*$-algebra $\Bi^\infty$, the module action taken modulo $\Gamma_0^{(k)}$. In fact, $\Ei_{(k)}$ is isomorphic as a Hilbert bimodule to $\Ei^{(-k)}\cong\Oi_\GH^{(-k)}$.
\end{prop}
\begin{proof} The first statements are proven in the same way as for $k=0$. For the last assertion, note that the algebras of compact module operators $\Ki_{\Bi_\infty}(\Ei^{(-k)})$ and $\Ki_{\Bi^\infty}(\Ei_{(k)})$ are isomorphic, namely to $\Bi_\infty\cong\Bi^\infty$. Hence the modules $\Ei_{(k)}$ and $\Ei^{(-k)}$ are isomorphic \cite{Frank1}.
\end{proof}

\subsection{Commutative case}
If $\GH_\bullet$ is associated with the radical homogeneous ideal which defines a projective variety $M\subset\C\Pb^{n-1}$, Assumption \ref{Qnotat} is satisfies if and only if $M$ is a coadjoint orbit $G/K$ under some Lie group $G$, and then one may take $Q=p_1$ (the identity operator on $\GH_1$). However, there are many commutative and noncommutative subproduct systems $\GH_\bullet$ for which a projective limit can be constructed in a weaker sense. 

Let us focus on the case when commutative $\GH_\bullet$. Then we know that the right shifts $R_1,\dots,R_n$ coincide with the left shifts $S_1,\dots,S_n$, and from our discussion about inductive limits we know that $[S_j^*,S_k]$ is compact for all $j,k=1,\dots,n$. Therefore, the maps $\jmath_{l,m}$, and hence the Toeplitz maps $\breve{\varsigma}^{(m)}$, become asymptotically multiplicative. That is, $\breve{\varsigma}$ is a homomorphism modulo compacts. Furthermore, the asymptotic unitality of the maps $\jmath_{m+1,m}$ allows us to define a norm on $\Bi^\infty$ by the same formula as before. The map $\breve{\varsigma}$ then becomes isometric, and it is adjoint to the map $\varsigma$ which identifies $\Ti_\GH^{(0)}/\Gamma_0$ with $\Bi_\infty=C^0(M)$. One easily sees that no Toeplitz operator is in $\Gamma_0$ and that
$$
\varsigma\circ\breve{\varsigma}=\id.
$$
The main difference from the special case $M=G/K$ is that for general $M$ one has
$$
\breve{\varsigma}(\bone)\ne\bone
$$
and the passage from $S$ to a spherical isometry is more involved since $\sum^n_{k=1}S_k^*S_kp_m$ is not just a scalar multiple of $p_m$ for each $m$.
 
As we have seen (recall Proposition \ref{strictquant}), from the version of Berezin quantization with prequantum condition one obtains a strict quantization of $C(M)$. With projectively induced quantization we obtain from Corollary \ref{strictquantcor}
a strict quantization of $C(M)$, and we do not require $M$ to be smooth.  
\begin{cor}\label{strictsingcor}
For any projective variety $M$, the sequence $(\Bi_m,\breve{\varsigma}^{(m)})_{m\in\N_0}$ is a strict quantization of the dense $*$-subalgebra of $C(M)$ generated by the $\varsigma^{(m)}(\Bi_m)$'s.
\end{cor}

Corollary \ref{strictsingcor} was inspired \cite{Hawk2}, where the Toeplitz operators where defined geometrically in the case of a prequantum quantization. When $M$ is smooth we see from the proof of \cite[Lemma 4.2]{Hawk2} that the limit state on $\Bi_\infty=C(M)$ is faithful and hence
$C(M)$ is also isomorphic to the subset $\breve{\varsigma}(\Bi_\infty)$ of the projective limit $\Bi^\infty$ with the multiplication taken modulo compacts. For non-smooth $M$ we do not know if the limit state on $\Bi_\infty$ is faithful.

\section{Application to compact matrix quantum groups}\label{CMQGappsec}

\subsection{Compact matrix quantum groups}\label{CMQGprels}
For the theory of compact quantum groups we refer to \cite{KlS}, \cite{MaVD}, \cite{Timm1}. We shall restrict attention to compact \emph{matrix} quantum groups, defined as follows. 
\begin{dfn}[{\cite{Wang3}, \cite{Wor1}}] 
A \textbf{compact matrix quantum group} $\G$ is defined by a $C^*$-algebra $C(\G)$ generated by the entries $u_{j,k}$ of a single unitary matrix $u\in\Mn_n(\C)\otimes C(\G)$ (for some $n\in\N$) such that the map $\Delta:C(\G)\to C(\G)\otimes C(\G)$ defined by
$$
\Delta(u_{j,k}):=\sum_{r=1}^nu_{j,r}\otimes u_{r,k}
$$
is a $*$-homomorphism, and such that the transpose $u^t$ is invertible.
\end{dfn} 
We refer to the generating matrix $u$ as the \textbf{defining representation} of the ``group" $\G$.  

The \textbf{Haar state} on $C(\G)$ (or the \textbf{Haar measure} on $\G$) is the unique state on $C(\G)$ which is \textbf{left $\G$-invariant}, in the sense that
$$ 
(\id\otimes h)\circ\Delta(f)=h(f)\bone.
$$
The Haar state is always faithful on the $*$-algebra generated by the $u_{j,k}$'s but not necessarily so on the norm closure $C(\G)$. There is a canonical construction of a ``reduced version" of $\G$, which is a compact quantum group with faithful Haar state \cite[§2]{BMT1} and has the same dense Hopf $*$-algebra. We shall always work with the reduced version or, what amounts to the same thing, assume that $h$ is faithful on all of $C(\G)$. Then $h$ is a KMS state \cite{Wor4}.

\subsubsection{Representations and actions}\label{repandactsec}
\begin{dfn}
A \textbf{representation} of $\G$ is a \textbf{corepresentation} of $C(\G)$, i.e. an invertible element $v\in\Bi(\GH_v)\otimes C(\G)$, for some Hilbert space $\GH_v$, satisfying (in leg-numbering notation)
$$
(\id\otimes\Delta)(v)=v_{13}v_{23}
$$ 
as elements of $\Bi(\GH_v)\otimes\Bi(\GH_v)\otimes C(\G)$. A representation $v$ is \textbf{irreducible} if the set
$$
\Hom_\G(v,v):=\{T\in\Bi(\GH_v)|(T\otimes\bone)v=v(T\otimes\bone)\}
$$
is trivial. 
\end{dfn} 
\begin{dfn}
Two representations $v\in\Bi(\GH_v)\otimes C(\G)$ and $w\in\Bi(\GH_w)\otimes C(\G)$ are \textbf{equivalent}, denoted $v\simeq w$, if there is a unitary $U:\GH_v\to\GH_w$ such that
$$
(U\otimes\bone)v=w(U\otimes\bone)
$$
(in particular, this requires $\dim\GH_v=\dim\GH_w)$. We denote by $\Irrep\G$ the (countable) set of equivalence classes of irreducible representations of $\G$. We choose a representative $u^{(\lambda)}\in\Bi(\GH_\lambda)\otimes C(\G)$ for each $\lambda\in\hat{\G}$. 
\end{dfn} 
\begin{dfn}
The \textbf{tensor product} of two representations $u\in\Bi(\GH)\otimes C(\G)$ and $v\in\Bi(\GK)\otimes C(\G)$ is the representation 
$$
u\otimes v:=u_{13}v_{23}\in\Bi(\GH\otimes\GK)\otimes C(\G).
$$
In particular, $u^{\otimes m}=u_{1,m+1}\cdots u_{m,m+1}$ is the matrix whose entries in the product basis for $\GH^{\otimes m}$ is given by
$$
u_{\mathbf{j},\mathbf{k}}=u_{j_1,k_1}\cdots u_{j_m,k_m}.
$$
\end{dfn} 

Now let us explain the motivation for the invertible operator $Q\in\Bi(\GH)$ that we incorporated in the Berezin quantization (recall §\ref{notationsec}). 

It is a crucial consequence of the axioms of compact matrix quantum groups that for any finite-dimensional representation $v\in\Bi(\GH_v)\otimes C(\G)$ of $\G$, one can find an invertible matrix $F_v\in\Bi(\GH_v)$ such that 
\begin{equation}\label{conjugatedef}
\bar{v}:=(F_v\otimes\bone)v^c(F^{-1}_v\otimes\bone)
\end{equation}
is unitary, where $v^c=(v^t)^*$ is the matrix whose coefficients are the adjoints of those of $v$.
The equivalence class of $\bar{v}$ is the \textbf{conjugate} of the equivalence class of $v$ (we shall also say that  $\bar{v}$ is a conjugate of $v$). The matrix $F_v$ in \eqref{conjugatedef} is usually chosen such that $Q_v:=F_v^*F_v$ satisfies
$$
\Tr(Q_v^{-1})=\Tr(Q_v)\equiv\dim_q(v),
$$ 
and this quantity is the ``quantum dimension" of $v$. We have $Q_{\bar{v}}=(Q_v^t)^{-1}$. We shall write $Q_\lambda:=Q_{u^{(\lambda)}}$ etc. for irreducibles $\lambda\in\Irrep\G$ and we denote by $\bar{\lambda}$ the conjugate of $\lambda$. 

Every representation of $\G$ decomposes completely into a direct sum of irreducibles. Hence, for each pair of irreps $\lambda,\mu\in\G$ there are integers $\textnormal{mult}(\nu,\lambda\otimes\mu)\in\N_0$ such that
\begin{equation}\label{fusruleseq}
u^{(\lambda)}\otimes u^{(\mu)}\simeq \bigoplus_{\nu\in\Irrep\G}\textnormal{mult}(\nu,\lambda\otimes\mu)u^{(\nu)}.
\end{equation}
\begin{dfn}\label{fusionrulesdef} The equations \eqref{fusruleseq} dictate the \textbf{fusion rules} of $\G$. The fusion rules are \textbf{commutative} if
$$
\textnormal{mult}(\nu,\lambda\otimes\mu)=\textnormal{mult}(\lambda\otimes\mu,\nu),\qquad \forall \lambda,\mu,\nu\in\Irrep\G.
$$ 
\end{dfn}
\begin{Example} Compact groups $\G=G$ have commutative fusion rules. More generally, $q$-deformations of compact Lie groups have commutative fusion rules because the equivalence class of an irreducible representation is determined by the highest weight of the representation. 
\end{Example}
The quantum groups in the next two examples are introduced in Definition \ref{UCMQGdef} below.
\begin{Example} For any $F$, the fusion rules of the quantum group $B_u(F)$ are identical to those of $\SU(2)$; in particular this is true for $\SU_q(2)$. These fusion rules in fact characterize the $B_u(F)$'s among compact quantum groups \cite[Théorème 2]{Ban3}. 
\end{Example}
\begin{Example} The fusion rules of $A_u(Q)$ are far from commutative, see \cite{Ban4}. 
\end{Example}

\begin{dfn}[{\cite[Def. 3.1]{Wang1}}]
A \textbf{left action} of a compact matrix quantum group $\G$ on a $C^*$-algebra $\Bi$ is a unital $*$-homomorphism $\alpha:\Bi\to\Bi\otimes C(\G)$ such that
\begin{enumerate}[(i)]
\item{$(\alpha\otimes\id)\circ\alpha=(\id\otimes\Delta)\circ\alpha$,}\label{leftac}
\item{$(\id\otimes\ep)\circ\alpha=\id$, where $\ep$ is the counit on the dense Hopf-$*$-subalgebra of $C(\G)$, and}
\item{there is a dense $*$-subalgebra $^{0}\Bi$ of $\Bi$ such that $\alpha({^{0}\Bi})\subset{^{0}\Bi}\otimes C^\infty(\G)$.}
\end{enumerate}
Similarly, a \textbf{right action} of $\G$ on $\Bi$ is a unital $*$-homomorphism $\alpha:\Bi\to C(\G)\otimes\Bi$ satisfying the obvious analogues of the properties (i), (ii) and (iii).
\end{dfn} 
If $\Bi$ is a von Neumann algebra then we replace $C(\G)$ by its weak closure $L^\infty(\G)$ in the GNS representation of the Haar state, and only condition (i) is required in the definition of an action. 

Every unitary representation $v\in\Bi(\GH)\otimes C(\G)$ of $\G$ induces a left action of $\G$ on $\Bi(\GH)$ given by
\begin{equation}\label{leftactionfromunitary}
\Ad_v:\Bi(\GH)\to \Bi(\GH)\otimes L^\infty(\G),\qquad \Ad_v(A):=v(A\otimes\bone)v^*.
\end{equation}

\subsubsection{Universal quantum groups}

In the following, for a matrix $u$ with entries in $C(\G)$, we write $u^c$ for the transpose of $u^*$, i.e. $(u^c)_{j,k}:=u_{j,k}^*$ where $u_{j,k}^*$ is the adjoint of $u_{j,k}$ in $C(\G)$.
\begin{dfn}[{\cite{Wang3}, \cite[Déf. 1]{Ban4}}]\label{UCMQGdef}
Let $F\in\GeL(n,C)$ be an invertible matrix and write $Q:=F^*F$. The \textbf{universal unitary quantum group} $\G=A_u(Q)$ is the compact matrix quantum group $\G$ whose algebra of continuous functions $C(\G)$ is generated by the entries of a unitary $n\times n$ matrix $u$ satisfying the relations making $(F\otimes\bone)u^c(F^{-1}\otimes\bone)$ a unitary matrix. 

The \textbf{universal orthogonal quantum group} $\G=B_u(F)$ is the compact matrix quantum group whose algebra $C(\G)$ is the quotient of that of $A_u(Q)$ by the relation $u=(F\otimes\bone)u^c(F^{-1}\otimes\bone)$. 
\end{dfn}
The prototype example of a $B_u(F)$ is the quantum $\SU_q(2)$ group $\G:=\SU_q(2)$. In general, $B_u(F)$ is some kind of higher-dimensional quantum $\SU(2)$ group which has no classical counterpart. 

Suppose that $\Hb$ and $\G$ are compact matrix quantum groups such that $C(\Hb)$ is a quotient of $C(\G)$. If the quotient map $\pi:C(\G)\to C(\Hb)$ fulfills $(\pi\otimes\pi)\circ\Delta_\G=\Delta_{\Hb}\circ\pi$, i.e. if $\pi$ intertwines the comultiplication of $\G$ with that of $\Hb$, then $\Hb$ is a \textbf{quantum subgroup} $\G$. We have seen that $B_u(F)$ is a quantum subgroup of $A_u(Q)$ when $F^*F=Q$. 

The name ``universal" is motivated by the following fact, which we should anticipate from \eqref{conjugatedef}. 
\begin{Lemma}[{\cite{VaDW}}] Any compact matrix quantum group $\G$ is a quantum subgroup of $A_u(Q)$ for some $Q$. If $\G$ in addition has a self-conjugate defining representation, then $C(\G)$ is a quantum subgroup of some $B_u(F)$. We write $\G\subset A_u(Q)$ and $\G\subset B_u(F)\subset A_u(Q)$ for these cases respectively.
\end{Lemma}

Let $u$ be the fundamental representation of $\G\subset A_u(Q)$, with $Q\in\GeL(n,\C)$. Then the elements $z_1:=u_{k,1},\dots,z_n:=u_{k,n}$ of the first row of $u$ satisfy the \textbf{$Q$-sphere relations}
\begin{equation}\label{qspherecond}
(Q^{-1})_{1,1}\bone=\sum^n_{r,s=1}(Q^{-1})_{r,s}z_r^*z_s,\qquad \bone=\sum^n_{s=1}z_sz_s^*.
\end{equation}

\subsubsection{The dual discrete quantum group}\label{dualsec}
Let $\G$ be a compact matrix quantum group such that the GNS representation $C(\G)\to \Bi(L^2(\G))$ of the Haar state is faithful. We shall identify $C(\G)$ with its image in $\Bi(L^2(\G))$ and denote by $L^\infty(\G)$ the von Neumann algebra generated by $C(\G)$ in $\Bi(L^2(\G))$. 
 
In perfect analogy to the case of ordinary compact groups, the $C^*$-algebra $C(\G)$ has a \textbf{Peter--Weyl decomposition}
\begin{equation}\label{PeterforCG}
C(\G)=\bigoplus_{\lambda\in\Irrep\G}\Bi(\GH_\lambda)^*,
\end{equation}
and the completion $L^2(\G)$ of $C(\G)$ in the inner product defined by the Haar state then allows for a similar decomposition. The comultiplication $\Delta$ on $C(\G)\subset\Bi(L^2(\G))$ takes the form
$$
\Delta(f)=W(f\otimes\bone)W^*,\qquad\forall f\in  C(\G)
$$
for a unitary operator $W$ on $L^2(\G)\otimes L^2(\G)$ referred to as the \textbf{multiplicative unitary} of $\G$. The \textbf{dual} of $\G$ is then defined via the (multiplier) Hopf $C^*$-algebra
\begin{equation}\label{cstardualgroup}
c_0(\hat{\G}):=\bigoplus_{\lambda\in\Irrep\G}\Bi(\GH_\lambda),
\end{equation}
with the comultiplication given by 
$$
\hat{\Delta}(X):=W^*(\bone\otimes X)W,\qquad\forall X\in c_0(\hat{\G}).
$$ 
We denote by $p_\lambda$ the identity in $\Bi(\GH_\lambda)$, regarded as an element of $c_0(\hat{\G})$. Then the counit $\hat{\ep}$ on $c_0(\hat{\G})$ is characterized by (cf. \eqref{vacuumcounit})
$$
\hat{\ep}(X)p_0=Xp_0,\qquad\forall X\in c_0(\hat{\G}),
$$
where $0\in\Irrep\G$ is the trivial representation. Every irreducible representation $u^{(\lambda)}$ of $\G$ is obtained from $W$ by means of
$$
u^{(\lambda)}=W(p_\lambda\otimes \bone).
$$

The object $\hat{\G}$ is referred to as a \textbf{discrete quantum group}. In the general theory of ``locally compact quantum groups", there is a canonical dual quantum group associated also to $\hat{\G}$, and this quantum group is precisely $\G$. In particular, the dual of an ordinary compact group $G$ is a discrete quantum group, which is an honest group only if $G$ is abelian. 

The $C^*$-algebra $c_0(\hat{\G})$ is contained in the $C^*$-algebra $\Ki$ of compact operators on $L^2(\G)$. The multiplier algebra of $c_0(\hat{\G})$ can be identified with the $C^*$-direct product
$$
c_b(\hat{\G}):=\prod_{\lambda\in\Irrep\G}\Bi(\GH_\lambda),
$$
and the comultiplication $\hat{\Delta}$ is a map from $c_0(\hat{\G})$ into $c_b(\hat{\G})\otimes c_b(\hat{\G})$. Continuing the analogy with the theory of honest groups, we shall denote by
$$
c_c(\hat{\G}):=\bigoplus_{\lambda\in\Irrep\G}\Bi(\GH_\lambda)
$$
the algebraic sum, and we denote the weak closure of $c_0(\hat{\G})$ in $\Bi(L^2(\G))$ by
$$
\Ri(\G)=\ell^\infty(\hat{\G}):=\ell^\infty\textnormal{-}\prod_{\lambda\in\Irrep\G}\Bi(\GH_\lambda).
$$
This ``group-von Neumann algebra" $\Ri(\G)$ is contained in the dual $C(\G)^*$ of $C(\G)$. 

Finally, the algebra of operators on $L^2(\G)$ affiliated with $\Ri(\G)$ can be identified with the product $\widehat{\prod}_{\lambda\in\Irrep\G}\Bi(\GH_\lambda)$, containing all (not necessarily bounded) sequences of elements in the $\Bi(\GH_\lambda)$'s. Of particular importance is the operator $Q_\G:=(Q_\lambda)_{\lambda\in\Irrep\G}$, where $Q_\lambda$ is as in §\ref{repandactsec}. 

From \eqref{cstardualgroup} we see that the irreducible representations of the $C^*$-algebra $c_0(\hat{\G})$ (also referred to as the irreducible ``corepresentations" of $\hat{\G}$) are parameterized by $\lambda\in\Irrep\G$. In fact, if $u^{(\lambda)}\in\Bi(\GH_\lambda)\otimes C(\G)$ is the irreducible representation of $\G$ with label $\lambda$ (as in §\ref{repandactsec}) then
\begin{equation}\label{dualcorep}
\pi_\lambda(X):=(\id\otimes X)(u^{(\lambda)})
\end{equation}
is the corresponding irreducible corepresentation of $\hat{\G}$, where $X\in c_0(\hat{\G})$ is regarded as a functional on $C(\G)$.  
In general, if $v$ is a unitary representation of $\G$ then \eqref{dualcorep} defines a representation $\pi_v$ of $\hat{\G}$ by substituting $u^{(\lambda)}$ with $v$. Then the commutant of $\pi_v(c_0(\hat{\G}))$ in $\Bi(\GH_v)$ equals
$$
\pi_v(c_0(\hat{\G}))'=\Bi(\GH_v)^\G,
$$
the fixed-point subalgebra under the $\G$-action $\Ad_v$ (recall \eqref{leftactionfromunitary}). 

The same equation \eqref{dualcorep} represents any (possibly unbounded) operator $X$ affiliated to $\Ri(\G)$ on $\Bi(\GH_\lambda)$. For each finite-dimensional representation $v$ of $\G$,
$$
\pi_v(Q_\G)=Q_v^t,
$$
where $Q_v$ is as in §\ref{repandactsec} and $Q_v^t$ denotes its transpose. In particular, $Q_v^t$ commutes with the projection onto any irreducible subrepresentation of $v$. It is well known that $\hat{\Delta}(Q_\G)=Q_\G\otimes Q_\G$, and it gives 
$$
Q_{v\otimes w}=Q_v\otimes Q_w
$$
for all finite-dimensional representations $v$ and $w$.

\subsection{First-row and first-column algebras}

Throughout this section, $\G$ is a compact matrix quantum group. We denote by $(u,\GH)$ the defining representation of $\G$. Thus the $C^*$-algebra $C(\G)$ is generated by the matrix coefficients of the unitary matrix $u\in\Bi(\GH)\otimes C(\G)$. Set $n:=\dim(\GH)$ and fix an orthonormal basis $e_1,\dots,e_n$ of $\GH$ so that $\GH\cong\C^n$, and let $u_{j,k}$ be the matrix coefficients of $u$ in this basis.  

\begin{dfn} The \textbf{first-row algebra} of $\G$ is the $C^*$-algebra $C(\Sb_\G)$ generated by the first row $z_1:=u_{1,1},\dots,z_n:=u_{1,n}$. This defines the quantum homogeneous space $\Sb_\G$. 
\end{dfn}
There is a $\Z$-grading on $C(\Sb_\G)$ obtained by letting the $z_k$'s have degree $1$ while their adjoints are given degree $-1$. We write the decomposition into spectral subspaces for the corresponding $\Un(1)$-action as
$$
C(\Sb_\G)=\overline{\bigoplus_{k\in\Z}C(\Sb_\G)^{(k)}}^{\|\cdot\|}.
$$
\begin{dfn} 
The quantum homogeneous space $\G/\K$ is defined as the noncommutative manifold corresponding to the $C^*$-subalgebra of fixed points in $C(\Sb_\G)$ for the $\Un(1)$-action:
$$
C(\G/\K):=C(\Sb_\G)^{(0)}.
$$
\end{dfn}
It is clear that $C(\G/\K)$ is generated by the $n^2$ elements $\{z_j^*z_k\}_{j,k=1}^n$ (but this is obviously not a minimal set of generators).
\begin{Example}  If $\G=G$ is a compact semisimple Lie group then $\G/\K$ is a coadjoint orbit and $\Sb_\G$ is a principal $\Un(1)$-bundle over $\G/\K$. Indeed, let $(U_{-1},\GH_{-1})$ be the irreducible unitary representation of $G$ with highest weight $(-1,0,\dots,0)$ and let $K$ be the stabilizer of the complex line spanned by the highest-weight vector $\xi_{-1}$. The action $U_{-1}$ of $G$ on $\GH_{-1}$ is unitary, and hence Hamiltonian for the symplectic form given by the imaginary part of the inner product on $\GH_{-1}$. The orbit $U_{-1}(G)\cdot[\xi_{-1}]$ of the line spanned by $\xi_{-1}$ is then a Hamiltonian $G$-homogeneous space, and a characterization of coadjoint orbits shows that $G/K\cong U_{-1}(G)\cdot[\xi_{-1}]$ is a coadjoint orbit.

The hyperplane bundle over the projectivization $\Pb[\GH_{-1}]$ restricts to a holomorphic line bundle $L$ over $G/K$ and we let $\Sb_G$ be the total space of the principal $\Un(1)$-bundle over $G/K$ associated with $L^*$ (cf. §\ref{circbundsec}). A basis for the space of holomorphic sections of $L$ generates $C(\Sb_G)$ and, after fixing a basis $e_1^*,\dots,e_n^*$ for $\GH_{-1}\cong\C^n$, such a basis is provided by the coordinate functions $z_1,\dots,z_n$ of $\C\Pb^{n-1}$ restricted to $G/K$. Choosing the basis such that $e_1^*\propto \xi_{-1}$ we get $\G/\K=G/K$ and $\Sb_\G=\Sb_G$. 

For example, if $G=\SU(n)$ then $K=\Un(1)\times\SU(n-1)$ and $\G/\K=G/K$ is complex projective $n$-space $\C\Pb^{n-1}$, whereas $\Sb_{\SU(n)}=\Sb^{2n-1}$ is the unit sphere in $\C^n$.

\end{Example}
\begin{Example} The preceding discussion carries over to $q$-deformations of $G$, and we get that $\G/\K$ is a quantum flag manifold. For $\G=\SU_q(n)$ we obtain quantum projective $n$-space $\G/\K=\C\Pb_q^{n-1}$, and $\Sb_\G$ is the $q$-deformed $(2n-1)$-sphere $\Sb^{2n-1}_q$ (cf. \cite{DDL1}). 
\end{Example}

\begin{Example}[{\cite[Thm. 3.3]{BaGo1}}]  
Let $\G:=\On^*(n)$ be the half-liberated orthogonal group. Then $C(\G/\K)$ is in fact commutative and $\G/\K$ is just the ordinary complex projective $n$-space $\C\Pb^{n-1}$,
$$
C(\G/\K)\cong C(\C\Pb^{n-1}).
$$
\end{Example}

The spectral subspaces $C(\Sb_\G)^{(k)}$ for the $\Un(1)$ action on $C(\Sb_\G)$ are Hilbert bimodules over the fixed-point subalgebra $C(\G/\K)$, where the multiplication in the ambient algebra $C(\Sb_\G)$ defines left and right $C(\G/\K)$-valued inner products
\begin{equation}\label{nativeinner}
\bra\xi|\eta\ket_{\rig}:=\xi^*\eta,\qquad \bra\xi|\eta\ket_{\lef}:=\xi\eta^*
\end{equation}
 between elements $\xi$ and $\eta$ in $C(\Sb_\G)^{(k)}$.

\begin{Remark}[Row vs column] We can also consider the $C^*$-algebra generated by the first column elements $w_j:=u_{j,1}$ of $u$ (the ``first-column algebra"). Everything proven about the first-row algebra $C(\Sb_\G)$ in this paper has a version where one instead uses the generators of the first-column algebra. 
\end{Remark}
\begin{Lemma}\label{ergonfirstrow}
The first-row algebra $C(\Sb_\G)$ carries an ergodic action of $\G$ which contains every irreducible representation of $\G$ with multiplicity one.
\end{Lemma}
\begin{proof} We define a left action $C(\Sb_\G)\to C(\Sb_\G)\otimes C(\G)$ by restriction of the comultiplication. The Peter--Weyl decomposition \eqref{PeterforCG} of $C(\G)$ gives the decomposition
$$
C(\Sb_\G)\cong\bigoplus_{\lambda\in\Irrep\G}\GH_\lambda,
$$
and the comultiplication restricts to the irreducible $\G$-representation $u^{(\lambda)}$ on each $\GH_\lambda$.  
\end{proof}
In view of Lemma \ref{ergonfirstrow} and the above examples, the quantum homogeneous $\G$-space $\G/\K$ is a natural generalization of the $q$-deformed projective spaces (in particular the standard Podle\'{s} sphere $\Sb^2_q=\C\Pb^1$). In all cases the unique invariant state under the $\G$-action is the restriction to $C(\Sb_\G)$ of the Haar state on $C(\G)$.

\subsection{Subproduct systems of $\G$-representations}
The subproduct system associated with a compact quantum group $\G$ will be a subproduct $\GH_\bullet$ in which the Hilbert space $\GH_m$ is contained in the $m$th tensor power $\GH^{\otimes m}$ of the fundamental representation $\GH$ of $\G$. 

The idea is based on the observation that if $u\in\Mn_n(\C)\otimes C(\G)$ is a representation of a compact matrix quantum group $\G$ then the first row $z_1:=u_{1,1},\dots,z_n:=u_{1,n}$ of $u$ transforms as the representation $u$ under the ``left translation" action $\lambda$ of $\G$ given by restricting the comultiplication $\Delta$:
$$
\lambda(z_k):=\Delta(z_k)=\sum^n_{j=1}z_j\otimes u_{j,k}.
$$
The constructions below can easily be made more general but we shall always assume that $u$ is irreducible. In fact we shall, for simplicity and concreteness, from now on assume that $u$ is the defining representation of $\G$. Let $\GH$ be the $n$-dimensional Hilbert space with basis vectors $z_1,\dots,z_n$. In most cases there are subrepresentations of $\G$ contained in the tensor product $\GH\otimes\GH$. Keeping only the largest $\G$-invariant subspace $\GH_2$ of $\GH\otimes\GH$ we obtain another irreducible representation $u^{(2)}$ of $\G$. Indeed, $\GH_2$ can be identified with the span of $z_jz_k$ for $j,k=1,\dots,n$, and then $u^{(2)}$ is obtained by the restricting the comultiplication, just as for $u$.  Continuing like this we obtain a family $(\GH_m)_{m\in\N_0}$ of Hilbert spaces satisfying the subproduct condition \eqref{subprodcond}. 

\begin{dfn}\label{subprodG}
Let $\G$ be a compact matrix quantum group. The subproduct system $\GH_\bullet$ just described will be referred to as the \textbf{the $\G$-subproduct system}. 
\end{dfn}
If we dropped the requirement that $u$ generates $C(\G)$ then Definition \ref{subprodG} makes no use of the fact that $C(\G)$ is finitely generated, and hence it works for all compact quantum groups. On the other hand, if we do not require irreducibility but $\G$ is a matrix group, we may always find a \emph{self-conjugate} unitary finite-dimensional representation $u$ whose coefficients generate $C(\G)$. Indeed, if $u$ generates $C(\G)$ then so does $u\oplus\bar{u}$, and the latter is self-conjugate. If $u$ is self-conjugate and generates $C(\G)$, every irreducible representation of $\G$ is obtained as $\GH_m$ for a unique $m\in\N_0$. However, for definiteness we shall always suppose that $\G$ is a compact \emph{matrix} quantum group and that $u$ is the defining representation, assumed irreducible.

\begin{Example} For $\G=G$ a classical compact Lie group, it is well known that the representation $\GH_{\lambda+\mu}$ with dominant weight $\lambda+\mu$ occurs exactly once in the tensor product $\GH_\lambda\otimes\GH_\mu$; this is the ``Cartan product" of $\GH_\lambda$ and $\GH_\mu$ \cite{East1}. In particular, if $\GH_\bullet$ denotes the $G$-subproduct system then $\GH_{m+1}$ is the Cartan product of $\GH_m$ and $\GH$. This subproduct system is commutative, i.e. $\GH_\bullet\subseteq\GH^{\vee \bullet}$, and the associated projective variety mentioned in §\ref{quantkaehlersec} is a coadjoint orbit, isomorphic to $G/K$ for some closed subgroup $K$ of $G$. 
\end{Example}
\begin{Example} The $\SU(n)$-subproduct system coincides with the fully symmetric subproduct system $\GH^{\vee\bullet}$ (Example \ref{fullandbosonex}).  
\end{Example}
\begin{Example} The $\G$-subproduct system of the universal quantum group $\G=A_u(Q)$ with positive $Q\in\GeL(n,\C)$ is the product system $\GH^{\otimes\bullet}$ because each power $u^{\otimes m}$ of the defining representation is irreducible \cite{Ban4}.
\end{Example}
\begin{Example} For $\G=B_u(F)$ we have self-conjugacy $u\simeq\bar{u}$ \cite{Ban3}. So if $u$ is irreducible then, as mentioned above, every $\lambda\in\Irrep\G$ occurs as $\GH_\lambda\cong\GH_m$ for some $m\in\N_0$. 
\end{Example}
The $\GH_m$'s are all irreducible and usually pairwise inequivalent. The only examples where they are not all inequivalent are those where $C(\G)$ is finite-dimensional.

\subsection{Berezin quantization of $C(\G/\K)$}

\subsubsection{Covariant symbols as first-row matrix coefficients}
Recall the formula \eqref{firstumforcov} for the covariant symbol derived in Proposition \ref{covsymstandardexprposm}. We now observe that for any compact matrix quantum groups $\G$, the same formula (with $u$ now being the defining representation of $\G$) appears when taking matrix coefficients of the representation $\Ad(u^{(m)})\simeq  u^{(m)}\otimes\bar{u}^{(m)}$ on $\Bi(\GH_m)$. 
\begin{dfn}
We say that an element of $C(\G/\K)$ is \textbf{normally ordered} if all $z_j$'s occur to the left of the $z_k^*$'s.
\end{dfn}
\begin{Lemma}
Suppose that every element in $C(\G/\K)$ can be written in normally ordered form. Then the $C^*$-algebra $C(\G/\K)$ can be generated by covariant symbols 
\begin{equation}\label{coeffcovsymb}
\varsigma^{(m)}_\G(A):=(\Tr\otimes\id)\big((A\otimes\bone)u^{(m)c*}(|e_1^{\otimes m}\ket\bra e_1^{\otimes m }|\otimes\bone)u^{(m)c}\big)
\end{equation}
for all $A\in\Bi(\GH_m)$ and all $m\in\N_0$.
\end{Lemma}
\begin{proof} Write $\phi_A:=\Tr(A\cdot)$. Then the elements
$$
\varsigma^{(m)}_\G(A)=(\phi_A\otimes\id)\big(u^{(m)c*}(|e_1^{\otimes m}\ket\bra e_1^{\otimes m}|\otimes\bone)u^{(m)c}\big),\qquad A\in\Bi(\GH_m)
$$
form a set which consists exactly of those matrix coefficients of the representation $u^{(m)c*}\otimes u^{(m)c}\simeq  u^{(m)}\otimes\bar{u}^{(m)}$ which are contained in $C(\G/\K)$. 
\end{proof}

\subsubsection{Intertwining the actions}
Let $W$ be the multiplicative unitary of $\G$, i.e. the unitary on $L^2(\G)\otimes L^2(\G)$ implementing the comultiplication $\Delta$ (cf. §\ref{dualsec}). Then the restriction of $W$ to $L^2(\G/\K)\otimes L^2(\G)$ identifies with
$$
u^{(\N)}:=(u^{(m)})_{m\in\N_0}\in c_b(\hat{\G})\otimes C(\G).
$$
We have an action of $\G$ on $\Ri(\G)$ given by $\Ad(W)$ (i.e. the same formula as for $\Delta$ but now applied to elements of a different algebra; of course this action extends to all of $\Bi(L^2(\G))$). The restriction of this $\G$-action to $\Bi(\GH_\N)$ is just $\Ad(u^{(\N)})$. Similarly, the extension of $\hat{\Delta}$ to $\Bi(L^2(\G))$ restricts to a right $\hat{\G}$-action on $L^\infty(\G)$.
 
So we have the following actions of $\G$ and $\hat{\G}$ on $C(\G/\K)$.

\begin{dfn} The left $\G$-action $\alpha^\G:C(\G/\K)\to C(\G/\K)\otimes C(\G)$ is defined by
$$
\alpha^\G(\varsigma^{(m)}_\G(A)):=(\varsigma^{(m)}_\G\otimes\id)(u^{(m)}(A\otimes\bone)u^{(m)*})
$$
for all $A\in\Bi(\GH_m)$ and all $m\in\N_0$. The right $\hat{\G}$-action $\hat{\alpha}^\G:C(\G/\K)\to c_b(\hat{\G})\otimes C(\G/\K)$ is defined by
$$
\hat{\alpha}^\G(\varsigma^{(m)}_\G(A)):=(\id\otimes\varsigma^{(m)}_\G)(u^{(m)*}(\bone\otimes A)u^{(m)}).
$$
\end{dfn}
\begin{Lemma}\label{equivlemma}
The map $\varsigma^{(m)}_\G$ is $\G$- and $\hat{\G}$-equivariant,
\begin{equation}\label{covequiv}
\Delta\circ\varsigma^{(m)}_\G(A)=(\varsigma^{(m)}_\G\otimes\id)(u^{(m)}(A\otimes\bone)u^{(m)*}),
\end{equation}
\begin{equation}\label{dualcovequiv}
\hat{\Delta}\circ\varsigma^{(m)}_\G(A)=(\id\otimes\varsigma^{(m)}_\G)(u^{(m)*}(\bone\otimes A)u^{(m)}).
\end{equation}
The actions $\alpha^\G$ and $\hat{\alpha}^\G$ therefore coincide on $C(\G/\K)\subset L^\infty(\G)$ with the left regular action of $\G$ and the right regular action of $\hat{\G}$, respectively. 
\end{Lemma}
\begin{proof} Let $\alpha^{(m)}(A):=u^{(m)}(A\otimes\bone)u^{(m)*}$ and $\hat{\alpha}^{(m)}(A):=u^{(m)*}(\bone\otimes A)u^{(m)}$. We use the defining property of a left action,
$$
(\alpha^{(m)}\otimes\id)\circ\alpha^{(m)}=(\id\otimes\Delta)\circ \alpha^{(m)}.
$$
Formula \eqref{firstumforcov} gives
\begin{align*}
(\varsigma^{(m)}_\G\otimes\id)(\alpha^{(m)}(A))&=(\Tr\otimes\id\otimes\id)\big((\alpha^{(m)}\otimes\id)\circ\alpha^{(m)}(A)(|e_1^m\ket\bra e_1^m|\otimes\bone)\big)
\\&=(\Tr\otimes\id\otimes\id)\big((\id\otimes\Delta)\circ\alpha^{(m)}(A)(|e_1^m\ket\bra e_1^m|\otimes\bone)\big)
\\&=\Delta(\Tr\otimes\id)\big(\alpha^{(m)}(A)(|e_1^m\ket\bra e_1^m|\otimes\bone)\big)=\Delta\circ\varsigma^{(m)}_\G(A).
\end{align*}
The proof of \eqref{dualcovequiv} is identical, using $(\id\otimes\hat{\alpha}^{(m)})\circ\hat{\alpha}^{(m)}=(\hat{\Delta}\otimes\id)\circ\hat{\alpha}^{(m)}$.
\end{proof}

\subsubsection{$C(\G/\K)$ as an inductive limit}
Now we will, for certain compact matrix quantum groups $\G$, realize the first-row algebra $C(\G/\K)$ as a projective limit  $\Bi^\infty$. From our previous results we have then obtained $C(\G/\K)$ as the $\Un(1)$-invariant part of the Cuntz--Pimsner algebra of the $\G$-subproduct system. 

Let $\G\subset A_u(Q)$ be a compact matrix quantum group (with faithful Haar measure), with $Q\in\GeL(n,\C)$, and let $C(\G/\K)$ be the $\Un(1)$-invariant part of the first-row algebra $C(\Sb_\G)$. 

Let $\GH_\bullet=\GH_\bullet^\G$ be the $\G$-subproduct system and let $\Oi_\GH$ be its Cuntz--Pimsner algebra. In §\ref{projlimsec} we defined the Toeplitz quantization $\breve{\varsigma}^{(m)}$ as a map from $\Oi_\GH$ to $\Bi^\infty$. In that way we could realize an isomorphism between $\Oi_\GH^{(0)}\cong\Bi_\infty$ and the projective limit $\Bi^\infty$. In this section we shall use the same strategy but with a map 
\begin{equation}\label{ToepforCMQG}
\breve{\varsigma}^{(m)}_\G:C(\G/\K)\to\Bi(\GH_m),
\end{equation}
i.e. we quantize $C(\G/\K)$ instead of $\Oi_\GH^{(0)}$. We define \eqref{ToepforCMQG} to be the adjoint of the map $\varsigma^{(m)}_\G$ appearing in \eqref{coeffcovsymb}. We shall in this way obtain an isomorphism between $C(\G/\K)$ and $\Bi^\infty$, and hence, due to $\Bi_\infty\cong\Bi^\infty$, we will arrive at the following result.
\begin{thm}\label{mainthmCMQGasindlim}
Assume that normal ordering is possible in $C(\G/\K)$. Then there is a $C^*$-isomorphism between $C(\G/\K)$ and the $\Un(1)$-invariant part of the Cuntz--Pimsner algebra $\Oi_\GH$ of the $\G$-suproduct system:
$$
C(\G/\K)\cong\Oi_\GH^{(0)}.
$$
This isomorphism is given by the total Toeplitz map $\breve{\varsigma}_\G=\prod_m\breve{\varsigma}^{(m)}_\G$, and it intertwines the ergodic $\G$- and $\hat{\G}$-actions as well as the $\G$-invariant states. 
\end{thm}
\begin{Remark}[Normal ordering]\label{normordrem}
For general $\G$, the map $\breve{\varsigma}_\G$ maps the normally ordered part of $C(\G/\K)$ onto $\Oi_\GH^{(0)}$.
 For instance, normal ordering is possible for $\G=B_u(F)$ but not for $A_u(Q)$. In fact, $A_u(Q)$ is our only example where $\breve{\varsigma}_\G$ is not an isomorphism. Commutative fusion rules implies normal ordering (recall Definition \ref{fusionrulesdef}). 
\end{Remark}

So let us begin by defining $\Bi^\infty$ to be the projective limit of the system $(\Bi(\GH_\bullet),\jmath_{\bullet,\bullet})$, where $\GH_\bullet$ is the $\G$-subproduct system. Here the operator $Q$ on $\GH$ which appears in the construction of $\Bi^\infty$ (§\ref{notationsec}) is taken to be the same as the matrix defining $A_u(Q)\supset\G$, assuming $Q$ is equal to its transpose. Thus, $\Bi(\GH_m)$ is equipped with the $\phi_m$-inner product, where $\phi_m=\Tr(\rho^{(m)}_Q\cdot)$ is the state defined by the density matrix $\rho_Q^{(m)}:=Q_m/\Tr(Q_m)$. We may therefore regard $\Bi^\infty\cap\Ti_\GH^{(0)}$ as a subset of $c_b(\hat{\G})$ having trivial intersection with $c_0(\hat{\G})$. 
\begin{Lemma}\label{Qsubnormlemma}
The $\G$-subproduct system $\GH_\bullet$ and the matrix $Q$ satisfy Assumption \ref{Qnotat}.
\end{Lemma}
\begin{proof}
That $Q^{\otimes m}$ preserves the subspace $\GH_m\subset\GH^{\otimes m}$ follows from the fact, mentioned in \S\ref{dualsec}, that $Q^{\otimes m}$ belongs to the commutant of the $\G$-action on $\GH^{\otimes m}$.

Let $L^2(\Sb)$ be the GNS Hilbert space of the restriction of the Haar state to the subalgebra $C(\Sb)\subset C(\G)$. Then the generators $z_1,\dots,z_n$ of $C(\Sb)$ are represented $*$-homomorphically as operators on $L^2(\Sb)$. If we denote these by $Z_1,\dots,Z_n$ we thus have (assuming $Q_{1,1}=1$) from \eqref{qspherecond} that
$$
\sum^n_{k=1}Q^{k,k}Z_k^*Z_k=\bone.
$$
Let $H^0(\Sb)$ be the closed subspace of $L^2(\Sb)$ spanned by $z_1,\dots,z_n$. Then $H^0(\Sb)$ is invariant under $Z_1,\dots,Z_n$. Denote by $T_k$ the restriction of $Z_k$ to $H^0(\Sb)$ for each $k\in\{1,\dots,n\}$. If $P$ is the orthogonal projection of $L^2(\Sb)$ onto $H^2(\Sb)$, we get
\begin{equation}\label{Qsubnormcond}
\sum^n_{k=1}Q^{k,k}T_k^*T_k=P\sum^n_{k=1}Q^{k,k}Z_k^*Z_k\big|_{H^2(\Sb)}=\bone.
\end{equation}
Now the tuple $T_1,\dots,T_n$ is unitarily equivalent to an operator tuple $\tilde{T}_1,\dots,\tilde{T}_n$ on the Fock space $\GH_\N$ satisfying
$$
\tilde{T}_j^*\tilde{T}_k|_{\GH_m}=\frac{\Tr(Q_{m+1})}{\Tr(Q_m)}S_j^*S_k|_{\GH_m},\qquad\forall m\in\N_0,\ j,k\in\{1,\dots,n\}.
$$
where $S_1,\dots,S_n$ are the standard shifts on $\GH_\N$ and we used again that $Q_m$ is simply the restriction of $Q^{\otimes m}$ to an invariant subspace. Equation \eqref{Qsubnormcond} then says that the maps $\jmath_{l,m}$ are unital, i.e. that $\iota_{m,l}$ preserves the $Q$-traces. 
\end{proof}
\begin{Remark}[$Q$-spherical isometries and $Q$-subnormality]
One may say that a tuple of operators $T_1,\dots,T_n$ satisfying Equation \eqref{Qsubnormcond} is a $Q$-\textbf{spherical isometry}'. For $Q=\bone$ we obtain the usual notion of a \textbf{spherical isometry}. The particular $Q$-spherical isometry $T_1,\dots,T_n$ obtained in the proof of Lemma \ref{Qsubnormlemma} is moreover $Q$-\textbf{subnormal} in that it is the restriction to an invariant subspace of a tuple of operators $Z_1,\dots,Z_n$ satisfying $\sum^n_{k=1}Z_kZ_k^*=\bone$. In the commutative case (where we must have $Q=\bone$) we know from \cite[Prop. 2]{Atha3} that a subnormal operator tuple is a spherical isometry if and only if it is subnormal and its normal extension has joint spectrum contained in the unit sphere $\Sb^{2n-1}$. 
%We see that if $\G/\K$ is an ordinary commutative manifold, so that $Q=\bone$, then 
\end{Remark}

From our general results we have an isomorphism $\breve{\varsigma}(\Bi_\infty)\cong\Bi_\infty$ which realizes the projective limit as an inductive limit. We stress again that this isomorphism $\breve{\varsigma}$ is not the same as the map $\breve{\varsigma}_\G$ which we now try to prove is an isomorphism.

The matrix coefficients of the operator $\breve{\varsigma}^{(m)}(f)$ are of the form
\begin{equation}\label{normordexplaineq}
h(z_\mathbf{j}z_\mathbf{k}^*f),\qquad \mathbf{j},\mathbf{k}\in\F_n^+\textnormal{ with }|\mathbf{j}|=m=|\mathbf{k}|.
\end{equation}

\begin{proof}[Proof of Theorem \ref{mainthmCMQGasindlim}]
Since the  ``coefficient map" $\varsigma^{(m)}_\G$ in \eqref{coeffcovsymb} is injective, its adjoint $\breve{\varsigma}^{(m)}_\G$ is surjective. As in the case of $\Bi_\infty$, we get that the image of $L^\infty(\G/\K)$ under $\breve{\varsigma}_\G$ is exactly $\Bi^\infty$ as a set.

We cannot use the reasoning in the proof of Lemma \ref{faithToepllemma} to deduce that $\breve{\varsigma}_\G$ is an injection of $C(\G/\K)$ into the operator system $\Bi^\infty$. On the other hand, we see directly that if $\breve{\varsigma}(f)=0$ then, using that the matrix coefficients of $\breve{\varsigma}(f)$ are given by \eqref{normordexplaineq} for all $m\in\N$, we get that $f$ must be orthogonal to the whole normally ordered part of $C(\G/\K)\subset L^2(\G)$. Since we have assumed that each $f\in C(\G/\K)$ can be normally ordered and that that the Haar state is faithful, this means that $f=0$.
 
Moreover, $\pi^{-1}(\breve{\varsigma}(\Bi_\infty))$ is again equal to $\Ti^{(0)}_\GH$. Namely, the proof in §\ref{CMasprojlimsec} carries over completely.

As before we get that $\breve{\varsigma}^{(m)}=\jmath_{\infty,m}$. Since we know that $\breve{\varsigma}(\Bi_\infty)$ is a $C^*$-algebra (using that it is a quotient of the Toeplitz algebra) with a unique multiplication, we obtain the von Neumann condition, i.e. $\breve{\varsigma}^{(m)}$ is asymptotically a homomorphism (see Corollary \ref{lemmaprojlimmult}). Thus $\breve{\varsigma}$ is an isomorphism for the projective-limit multiplication.

We also know that $\breve{\varsigma}_\G$ intertwines the vacuum state $\hat{\ep}$ with the Haar state $h$ restricted to $C(\G/\K)$. Composing with the isomorphism $\breve{\varsigma}:\Bi_\infty\to\breve{\varsigma}(\Bi_\infty)$ we get that $h$ is intertwined with the limit state $\omega_Q$. Finally, Lemma \ref{equivlemma} shows that $\breve{\varsigma}_\G$ is $\G$-$\hat{\G}$-equivariant.
\end{proof}

\subsubsection{$C(\Sb_\G)$ as an inductive limit}

\begin{cor}\label{firstrowasPims}
Let $\G$ be a compact matrix quantum group with faithful Haar measure $h:C(\G)\to\C$ such that normal ordering is possible in $C(\G/\K)$. Then there is a $\G$-$\hat{\G}$-equivariant isomorphism between the first-row algebra $C(\Sb_\G)$ and the Cuntz--Pimsner algebra $\Oi_\GH$ of the $\G$-subproduct system,
$$
C(\Sb_\G)\cong\Oi_\GH.
$$
In particular, $\Oi_\GH$ carries an ergodic action of $\G$ in which each irreducible representation of $\G$ occurs exactly once.
\end{cor}
\begin{proof} For notation simplicity we identify $\Oi_\GH^{(0)}$ with the inductive limit $\Bi_\infty$ and the modules $\Oi_\GH^{(k)}$ with the modules $\Ei^{(k)}$.  

Since $\Bi_\infty\cong C(\G/\K)$, we know that there is a basis $e_1,\dots,e_n$ for $\GH$ such that the $Q$-sphere condition \eqref{qspherecond} is satisfied by the generators $Z_1,\dots,Z_n$ of $\Oi_\GH$, just as it is for the generators $z_1,\dots,z_n$ of $C(\Sb_\G)$. This says precisely that $Z_1,\dots,Z_n$ and $Q^{-1/2}_{1,1}Z_1,\dots,Q^{-1/2}_{n,n}Z_n$ are standard right and left tight normalized frames for the $\Bi_\infty$-bimodule $\Ei^{(1)}$, respectively; for all $\xi\in\Ei^{(1)}$,
$$
\sum^n_{k=1}\bra \xi|Z_k\ket_{\rig}\bra Z_k|\xi\ket_{\rig}=\bra \xi|\xi\ket_{\rig},
$$
\begin{align*}
\sum^n_{k=1}\bra \xi|Q^{-1/2}_{k,k}Z_k\ket_{\lef}\bra Q^{-1/2}_{k,k}Z_k|\xi\ket_{\lef}
&=\sum^n_{k=1}\xi Q_{k,k}^{-1} Z_k^*Z_k\xi^*
=\xi\xi^*=\bra\xi|\xi\ket_{\lef},
\end{align*}
and identically for $C(\Sb_\G)^{(1)}$ and the $z_j$'s. If we identify $C(\G/\K)$ with $\Bi_\infty$, this means that the projection $P^{(1)}\in\Mn_n(\C)\otimes C(\G/\K)$ which defines the module $C(\Sb_\G)^{(1)}$ coincides with the projection which defines the module $\Ei^{(1)}$. So the modules are the same and the isomorphism $C(\Sb_\G)\cong\Oi_\GH$ is clear.

For the $\G$-$\hat{\G}$-equivariance, we must first define actions on $\Oi_\GH$. But since we know that $C(\Sb_\G)\cong\Oi_\GH$ we can just specify these action on generators $Z_1,\dots,Z_n$ by the same formulas as for $C(\Sb_\G)$. 

The last statement is due to Lemma \ref{ergonfirstrow}.  
\end{proof}

\subsection{Comparison with Poisson and Martin boundaries}

\subsubsection{Poisson integral versus total Toeplitz map}
Let $\G$ be a compact matrix quantum group with commutative fusion rules (see Definition \ref{fusionrulesdef}) and faithful Haar measure. The Poisson boundary to be discussed here is the one defined in \cite{Iz1}, so if we were phrasing things in terms of random walks (we shall not), there would in the background be a representation $u$ of $\G$ whose coefficients generate $C(\G)$ (without any need of the adjoints $u_{j,k}^*$). 

Izumi defines \cite[Lemma 3.8]{Iz1} the \textbf{Poisson integral} to be the unital completely positive map $\Theta:L^\infty(\G)\to\Ri(\G)$ given by
\begin{equation}\label{poissinteg}
\Theta(f):=(\id\otimes h)(W^*(\bone\otimes f)W),
\end{equation}
where $W$ is the fundamental unitary (§\ref{dualsec}). Similar to the projective limit $\Bi^\infty$ which is the image of our total Toeplitz map $\breve{\varsigma}$, the image of map $\Theta$ is an operator system, usually denoted by $H^\infty(\hat{\G})$, which can be made into a von Neumann algebra by replacing the operator multiplication by the new one. Moreover, $\Theta$ is a complete order isomorphism onto its image. A possible definition of the \textbf{Poisson boundary} of $\hat{\G}$ is then as the preimage, say $L^\infty(\G/\T)$, of $\Theta$ in $L^\infty(\G)$. We then refer to the abstract object $\G/\T$ as the Poisson boundary of $\hat{\G}$. The notation $\G/\T$ is chosen to indicate that $\G$ and $\hat{\G}$ act ergodically on $L^\infty(\G/\T)$. 

The Poisson boundary $\G/\T$ is defined in terms of a von Neumann algebra. In order to compare $\G/\T$ with what we have denoted $\G/\K$, note that $\breve{\varsigma}$ extends to a normal completely positive map (denoted by the same symbol) 
\begin{equation}\label{totcovarsymbvN}
\breve{\varsigma}:L^\infty(\G/\K)\to\Bi^\infty
\end{equation}
and this is the ``first-row" version of the Poisson integral. Using it one can carry out Berezin quantization on the level of von Neumann algebras. Inspiring work here is \cite{INT1}.

The Poisson integral \eqref{poissinteg} can be decomposed into components $\Theta_\lambda:L^\infty(\G)\to\Bi(\GH_\lambda)$ for $\lambda\in\Irrep\G$, and doing so one easily calculates the adjoints $\Theta_\lambda^*:\Bi(\GH_\lambda)\to L^\infty(\G)$. Noticing the similarity to Berezin quantization, \cite{INT1} referred to the composition $\Theta_\lambda^*\circ\Theta_\lambda$ as the ``Berezin transform". This terminology is not entirely fortunate because $\Theta_\lambda^*\circ\Theta_\lambda$ does not coincide with the usual notion of Berezin transform when $\G=G$ is an ordinary group. The issue is that $\Theta_\lambda^*$ is obtained by tracing against the invertible operator $Q_\lambda$ (which is of full rank) instead of a rank-$1$ projection. The distinction is the use of ``first-row" versus all of $\G$. This distinction persists even if we, as Izumi does, assume that every irreducible representation of $\G$ is contained in some power of $u$.

It is therefore interesting that the final results ($\G/\K$ and $\G/\T$) are not very different. For $\SU_q(2)$ they even coincide. In general, we should view $\G/\K$ as a (noncommutative) non-maximal flag variety (prototype example being $\C\Pb_q^{n-1}$) while $\G/\T$ is the maximal flag variety (so $\T$ is the ``maximal torus"); cf. \cite{Tom1}. 

The transition between classical and quantum Poisson boundaries is rather involved \cite{NT1}. In fact, if $\G=G$ is an ordinary compact group then the Poisson boundary is trivial: $L^\infty(\G/\T)=\C\bone$ \cite[Cor. 3.9]{Iz1}. In contrast, Berezin quantization carries over in perfect analogy with the commutative case. 

\subsubsection{Markov operator}\label{markovsec}
The set $H^\infty(\Phi)$ of fixed points of a normal completely positive map $\Phi$ on a von Neumann algebra is an ultraweakly closed operator system which can be made into a von Neumann algebra by replacing the operator multiplication by the so-called ``Choi--Effros multiplication" \cite[Thm. 3.1]{Arv10}, \cite{Iz4}. 

The new multiplication on the Poisson boundary $H^\infty(\hat{\G})$ mentioned above is just an example of a Choi--Effros multiplication. The completely positive map on $\Ri(\G)$ whose fixed-point set equals $H^\infty(\hat{\G})$ takes the role of Markov operator for the ``noncommutative random walk" on $\hat{\G}$.  

The following can be summarized by saying that with Berezin quantization one ends up with a random walk on the ``dual" of $\G/\K$ instead of the dual of $\G$. Note however that it works for any subproduct system $\GH_\bullet$. Fix thus a subproduct system $\GH_\bullet$ and denote as usual by $\Gamma_b=\Gamma_b(\Bi_\bullet)$ the von Neumann-algebraic direct sum of the matrix algebras $\Bi_m:=\Bi(\GH_m)$.
\begin{dfn}  The \textbf{Markov operator} on $\Bi_\bullet$ is the unital normal completely positive map $\Phi:\Gamma_b\to\Gamma_b$ defined by
$$
\Phi(X_\bullet):=X_{\bullet-1}=(\jmath_{m,m-1}(X_m))_{m\in\N}.
$$
%$$
%\Phi(X_\bullet):=(\jmath_{m+1,m}(X_{m+1}))_{m\in\N_0}.
%$$
\end{dfn}
%Clearly $\Phi$ fixes every point in $\Bi^\infty$.
\begin{prop}\label{Markovprop} The set $H^\infty(\Phi)$ of $\Phi$-fixed points in $\Gamma_b(\Bi_\bullet)$ is equal to the projective limit $\Bi^\infty$. In particular, $\Bi^\infty$ is a von Neumann algebra.
%The set of $\Phi$-fixed points in $\{A\in\Gamma_b(\Bi_\bullet)|\ \|A_m\|\operatorname{ converges}\}$ is equal to $\Bi^\infty$ (as a set) 
On the subset $\breve{\varsigma}(\Bi_\infty)\subset\Bi^\infty$, the Choi--Effros multiplication coincides with the projective-limit multiplication.  
\end{prop}
\begin{proof} The first statement is clear, so $\Bi^\infty$ is a von Neumann algebra. %We need only show that the projective-limit multiplication on $\Bi^\infty$ coincides with the Choi--Effros multiplication.
To prove the last statement we use the result \cite[Cor. 5.2]{Iz3} that the Choi--Effros product of $X,Y\in H^\infty(\Phi)$ is given by
$$
X\diamond Y=\lim_{r\to\infty}\Phi^r(XY),
$$
where $XY$ is the multiplication in $\Bi(\GH_\N)$ and the limit is in the strong operator topology. % Peterson shows that the limit belongs to $\Bi^\infty$ but does that imply that it is the norm limit? Do we need that? 
Now, the $m$th component of $X\diamond Y$ is
\begin{align*}
(X\diamond Y)_m&=\big(\lim_{r\to\infty}\Phi^r(XY)\big)_m=\lim_{r\to\infty}\jmath_{m+r,m}((XY)_{m+r})
\\&=\jmath_{\infty,m}\big(\lim_{r\to\infty}(XY)_{m+r}\big)
\\&=\big(\lim_{l\to\infty}\breve{\varsigma}^{(l)}(XY)\big)_m,
\end{align*}
%\begin{align*}
%X\diamond Y&=\lim_{r\to\infty}\Phi^r(XY)=\lim_{r\to\infty}(\jmath_{m+r,m}(XY)_{m+r})_{m\in\N_0}
%\\&=((\lim_{r\to\infty}\jmath_{\infty,m+r}(XY))_m)_{m\in\N_0}
%\\&=((\lim_{l\to\infty}\breve{\varsigma}^{(l)}(XY))_m)_{m\in\N_0},
%$$
%\lim_{m\to\infty}\Phi^m(XY)=\lim_{m\to\infty}\Phi^m(((XY)_l)_{l\in\N_0})
%=\lim_{m\to\infty}((XY)_{l-m})_{l\in\N_0}=\lim_{l\to\infty}(XY)_l,
%$$
%\end{align*}
where we used that $X_m=\jmath_{\infty,m}(X)=\breve{\varsigma}^{(m)}(X)$. This shows that $\diamond$ is the projective-limit multiplication \eqref{projlimmult} whenever we have convergence in norm. Since norm-convergence holds for $X=\breve{\varsigma}(f)$ and $Y=\breve{\varsigma}(g)$ with $f,g\in\Bi_\infty$, the claim holds. 
 %So $H^\infty(\Phi)=(\Bi^\infty)''$. 
%The fact that $H^\infty(\Phi)$ forms an algebra under $\diamond$ is then just the statement that $\lim_{l\to\infty}\breve{\varsigma}^{(l)}(XY)$ is constant under $\jmath_{\bullet,\bullet}$ (which is obvious) and exists in the strong topology. We thus know that the weak closure of $(\Bi^\infty,\diamond)$ is a von Neumann algebra. The commutant of $(\Bi^\infty,\diamond)$ is then also a von Neumann algebra. This does probably not imply that $\Bi^\infty$ is an algebra, for otherwise we would have the general statement that if $\Phi:\Ai\to\Ai$ is a completely positive map on a $C^*$-algebra such that the extension $\Phi:\Ai''\to\Ai''$ is normal the intersection of $H^\infty(\Phi)$ with any weakly dense subspace $\Bi\subset\Ai$ is an algebra. 
\end{proof}

\subsubsection{Martin boundaries}
While the Poisson boundary is a measure-theoretic object defined via a von Neumann algebra, the Martin boundary is specified in terms of a $C^*$-algebra \cite{NT1}. Its relation to Berezin quantization is the same "first-row versus all-of-$\G$`` story as with the Poisson boundary but we shall discuss only a special case in which $\G/\K$ agrees with the Martin boundary of $\hat{\G}$. The reason for this coincidence is that the defining representation of the chosen $\G$ is self-conjugate and irreducible, so that $\GH_\bullet$ contains all irreducible representations. 

Our approach here via inductive limits was partially inspired by \cite{VVer1}, where they construct a ``Martin boundary" of the dual of $\G$ for $\G=B_u(F)$ in the same way. Our notation $\Bi_\infty$ is chosen to make comparison with that paper easy. Let $F\in\GeL(n,\C)$ such that $\bar{F}F=\pm\bone$; this ensures that the defining representation of $\G=B_u(F)$ is irreducible. By construction, the Martin boundary of $\hat{\G}$ is equal to the inductive limit $\Bi_\infty$ of the $\G$-subproduct system.  

In \cite{VaVe1}, another realization of the Martin boundary was accomplished. First define $\Bi_u(F,F_q)$ to be the universal $C^*$-algebra generated by the entries of a unitary $2\times n$ matrix $Y$ satisfying
\begin{equation}\label{Yrels}
Y=F_qY^cF^{-1}.
\end{equation}
where $F_q:=\big(\begin{smallmatrix}0&|q|^{1/2}\\\pm|q|^{-1/2}&0\end{smallmatrix}\big)$, with $q$ defined by $|q+q^{-1}|=\Tr(F^*F)$ and $F\bar{F}=\pm q$ and we wrote $q=\mp|q|$. It is shown in \cite{VaVe1} that the $\Un(1)$-action $\rho$ on $\Bi_u(F,F_q)$ given by 
$$
\rho_\lambda(Y):=\begin{pmatrix}\lambda&0\\0&\bar{\lambda}\end{pmatrix}Y,\qquad \forall \lambda\in\C,\,|\lambda|=1
$$
allows recovering the Martin boundary as the fixed-point algebra $\Bi_u(F,F_q)^{\Un(1)}$. 

Since our inductive limit $\Bi_\infty$ coincides with the Martin boundary of the dual of $B_u(F)$, we know (using Theorem \ref{mainthmCMQGasindlim}) that $\Bi_u(F,F_q)^{\Un(1)}$ must coincide with $C(\G/\K)$. This can be seen directly. Let $z_1,\dots,z_n$ and $w_1,\dots,w_n$ be the elements of the first and second row of $Y$ respectively. Then \eqref{Yrels} reads
\begin{equation}\label{zasfromwes}
z_k=|q|^{1/2}\sum^n_{s=1}w_s^*F^{-1}_{s,k},\qquad \forall k=1,\dots,n.
\end{equation}
The action $\rho_z$ is given by
$$
\rho_\lambda(z_k)=\lambda z_k,\qquad \rho_\lambda(w_k)=\bar{\lambda}w_k,
$$
so the fixed-point algebra consists of elements of the form $z_\mathbf{j}w_\mathbf{k}$, and their adjoints, as well as $z_\mathbf{j}z_\mathbf{k}^*$ and $w_\mathbf{j}w_\mathbf{k}^*$. Now \eqref{zasfromwes} shows that $\Bi_u(F,F_q)^{\Un(1)}=C(\G/\K)$, as asserted. 

Note however that $\Bi_u(F,F_q)$ is not at all the same as the first-row algebra $C(\Sb_\G)$. 

Thus we have one example where the Martin boundary of a discrete quantum group identifies with the object $\G/\K$ defined in this paper. Also, let $q\in(0,1]$. Then for the $\G=\SU_q(2)$-subproduct system $\GH_\bullet$ we have an equivariant isomorphism
$$
\Oi^{(0)}_{\GH}\cong C(\Sb_q^2),
$$
where $\G/\K=\Sb^2_q$ is the Podle\'{s} sphere. For $q<1$, this shows again that our inductive limit $\Bi_\infty$ coincides with the $C^*$-algebra referred to as the Martin boundary of the dual of $\SU_q(2)$ in \cite{VVer1} (because the same boundary is known to be the Podle\'{s} sphere \cite{NT1}). For $q=1$, we have agreement with Biane's Martin boundary of the dual of $\SU(2)$ \cite{Biane1}.

\section{Concluding remarks}

We have seen that the inductive limit $\Bi_\infty$ associated with a subproduct system $\GH_\bullet$ is a sensible generalization of the $C^*$-algebra $C(M)$ of continuous functions on a quantizable Kähler manifold, in the case the Kähler structure is projectively induced (so $M$ is embeddable as a submanifold of projective space). We may therefore write
$$
C(\M):=\Bi_\infty,
$$
and say that $\M$ is the \textbf{noncommutative projective variety} associated to $\GH_\bullet$. We may also refer to elements of
$$
C(\Sb_\M):=\Oi_\GH.
$$
as functions on the total space $\Sb_\M$ of a noncommutative circle bundle over $\M$. Thus the notation $\M:=\G/\K$ and $\Sb_\M:=\Sb_\G$ would be consistent with that in \S\ref{CMQGappsec} when $\GH_\bullet$ is the $\G$-subproduct system. In \cite{An5} we refer to $\M$ as the ``dequantization manifold".

By defining $C(\M)$ to be \emph{equal} (and not just isomorphic) to the inductive limit $\Bi_\infty$, the noncommutative space $\M$ comes with more structure than just its topology. Namely, if $\M=M$ is commutative then the inductive system gives an embedding into projective space $\C\Pb^{n-1}$ and, if $M$ is non-singular, endows $M$ with a complex-analytic (in particular smooth) structure, a polarization $L$ (choice of ample line bundle), an inner product on $H^0(M;L)$ (the one we started with, making $H^0(M;L)$ into the Hilbert space $\GH$), a Hermitian metric on $L$ (the one defining the $*$-structure on $\Oi_\GH$; this is just the Fubini--Study metric associated with the inner product on $\GH$) and a volume form on $M$ (viz. the limit state, which need not be the same as Fubini--Study volume form). As we have seen, these structures have perfect generalizations to the noncommutative setting. Note that the quantum homogeneous spaces $\G/\K$ are ``balanced" in the sense that the limit state on $C(\G/\K)$ coincides with the state induced by the Haar state on $C(\G)$.

The covariant symbols $\varsigma^{(m)}(A)$ and the Toeplitz operators $\breve{\varsigma}^{(m)}(f)$ can be expressed in terms of the projections $P^{(m)}\in\Bi(\GH_m)\otimes C(\M)$ which define the modules $\Ei^{(m)}$. In this way one generalizes the Rawsnely coherent-state projections in \cite{RCG1}, and in particular the coherent states in \cite{Per1}. We will use this when we discuss the (fuzzy) geometry of $\M$ in another paper. 

There are also projections $P^{(-m)}$ and maps $\varsigma^{(-m)}$, $\breve{\varsigma}^{(-m)}$ etc. associated with the modules $\Ei^{(-m)}$. In this way one quantizes instead the anti-normal part of $C(\G/\K)$.

\section{References}

%\bibitem[AgMc1]{AgMc1} Agler J, McCarthy JE. Complete Nevanlinna--Pick kernels. J. Funct. Anal. Vol 175, pp.
11-124 (2000).

%\bibitem[AgMc2]{AgMc2} Agler J, McCarthy JE. Pick interpolation and Hilbert function spaces. Graduate Studies in Mathematics. Vol 44. Amer. Math. Soc. (2002).

\bibitem[An4]{An4} Andersson A. Detailed balance as a quantum-group symmetry of Kraus operators. arXiv:1506.00411
(2015).

\bibitem[An5]{An5} Andersson A. Dequantization via quantum channels. Lett. Math. Phys. Vol 106, Issue 10, pp. 1397-1414 (2016).

\bibitem[ArMa1]{ArMa1} Ara P, Mathieu M. Local multipliers of $C^*$-algebras. Springer (2003).

\bibitem[ArLo1]{ArLo1} Arezzo C, Loi A. Quantization of Kähler manifolds and the asymptotic expansion of Tian--Yau--Zelditch. J. Geom. Phys. Vol 47, Issue1, pp. 87-99 (2003).

\bibitem[ArZh1]{ArZh1} Artin M, Zhang JJ. Noncommutative projective schemes. Adv. in Math. Vol 109, Issue 2, pp.
228-287 (1994).

\bibitem[Arv6c]{Arv6c} Arveson W.  Subalgebras of $C^*$-algebras III: Multivariable operator theory. Acta Math. Vol 181, pp. 159-228 (1998).
\bibitem[Arv8]{Arv8} Arveson W. $p$-Summable commutators in dimension $d$. J. Operator Theory. Vol 54, pp. 101-117 (2005).
\bibitem[Arv9]{Arv9} Arveson W. Quotients of standard Hilbert modules. Trans. Amer. Math. Soc. Vol 359, pp. 6027-6055 (2007).
\bibitem[Arv10]{Arv10} Arveson W. Noncommutative Poisson boundaries. Unpublished. Available at http://math.berkeley.edu/~ arveson/Dvi/290F04/22Sept.pdf (2004).

\bibitem[Atha3]{Atha3} Athavale A. On the intertwining of joint isometries. J. Oper. Theory. Vol 23, pp. 339-350 (1990).

\bibitem[BDLMC]{BDLMC} Balachandran AP, Dolan BP, Lee JH, Martin X, O’Connor D. Fuzzy complex projective spaces and their star-products. J. Geom. Phys. Vol 43, pp. 184 (2002).

\bibitem[Ban3]{Ban3} Banica T. Théorie des représentations du groupe quantique compact libre $\On(n)$. C. R. Acad. Sci.
Paris Sér. I Math. Vol 322, pp. 241-244 (1996).

\bibitem[Ban4]{Ban4} Banica T. Le groupe quantique compact libre $\Un(n)$. Comm. Math. Phys. Vol 190, pp. 143–172 (1997).

\bibitem[BaGo1]{BaGo1} Banica T,  Goswami D. Quantum isometries and noncommutative spheres. Comm. Math. Phys. Vol 298, Issue 2, pp. 343-356 (2010). 

\bibitem[BMT1]{BMT1} Bédos E, Murphy GJ, Tuset L. Co-amenability of compact quantum groups. J. Geom. Phys. Vol 40, Issue 2, pp. 129-153 (2001).

\bibitem[BeSl1]{BeSl1} Berceanu S, Schlichenmaier M. Coherent state embeddings, polar divisors and Cauchy formulas. J. Geom. Phys. Vol, Issue 34, pp. 336-358 (2000).

\bibitem[Bere1]{Bere1} Berezin FA. Quantization. Math. USSR-Izv. Vol 38, pp. 1116-1175 (1974).

\bibitem[Bere2]{Bere2} Berezin FA. General concept of quantization. Comm. Math. Phys. Vol 40, pp. 153-174 (1995).

\bibitem[Bere3]{Bere3} Berezin FA. Covariant and contravariant symbols of operators. Math. USSR-Izv. Vol 6, Issue 5, p. 1117 (1972).

\bibitem[Bere4]{Bere4} Berezin FA. Some remarks about the associative envelope of a Lie algebra. Funct. Anal. Appl. Vol 1, pp. 91-102 (1967).

%\bibitem[BeHi1]{BeHi1} Bertram W, Hilgert J. Reproducing kernels on vector bundles. In: Lie Theory and Its Applications in Physics III. pp. 43-58 (1998).

\bibitem[Biane1]{Biane1} Biane P. Introduction to random walks on noncommutative spaces. In: Quantum Potential Theory. pp. 61-116. Springer (2008).

\bibitem[Bied1]{Bied1} Biedenharn LC. The quantum group $\SU_q(2)$ and a $q$-analogue of the boson operators. J. Phys. A: Math. Gen. Vol 22, pp. L873-L878  (1989).

\bibitem[Blac1]{Blac1} Blackadar B. Operator algebras: theory of $C^*$-algebras and von Neumann algebras. Springer (2006).

\bibitem[BlKi1]{BlKi1} Blackadar B, Kirchberg E. Generalized inductive limits of finite-dimensional $C^*$-algebras. Math. Ann. Vol 307, pp. 343-380 (1997).

\bibitem[BMS]{BMS} Bordemann M, Meinrenken E, Schlichenmaier M. Toeplitz quantization of Kähler manifolds and $\operatorname{gl}(n),n\to\infty$ limits. Comm. Math. Phys. Vol 165, Issue 2, pp. 281–296 (1994).

\bibitem[BHSS]{BHSS} Bordemann M, Hoppe J, Schaller P, Schlichenmaier M. $\textnormal{gl}(\infty)$ and geometric quantization. Comm. Math. Phys. Vol 138, Issue 2, pp. 209-244 (1991).

\bibitem[BLY1]{BLY1} Bourguignon J-P, Li P, Yau ST. Upper bound for the first eigenvalue of algebraic submanifolds. Comment. Math. Helv. Vol 69, pp. 199-207 (1994).

\bibitem[Brow1]{Brow1} Brown NP. Invariant means and finite representation theory of $C^*$-algebras. Mem. Amer. Math. Soc.   Vol 184, Issue 865 (2006). 
\bibitem[BrOz1]{BrOz1} Brown NP, Ozawa N. $C^*$-algebras and finite-dimensional approximations. Amer. Math. Soc. (2008). 

\bibitem[CGR1]{CGR1} Cahen M, Gutt S, Rawnsley J. Quantization of Kähler manifolds I: geometric interpretation of Berezin's quantization. J. Geom. Phys. Vol 7, Issue 1 (1990). 

\bibitem[CGR2]{CGR2} Cahen M, Gutt S, Rawnsley J. Quantization of Kähler manifolds II. Trans. Amer. Math. Vol 337, Issue 1 (1993). 

%\bibitem[CaKe1]{CaKe1} Cao HD, Keller J. On the Calabi problem: a finite-dimensional approach. J. Eur. Math. Soc. (JEMS) Vol 15, Issue 3, pp. 1033-1065 (2013).

\bibitem[DDL1]{DDL1} D'Andrea F, Dabrowski L, Landi G. The noncommutative geometry of the quantum projective plane. Reviews in Mathematical Physics. Vol 20, Issue 08, pp. 979-1006 (2008).

\bibitem[DRS1]{DRS1} Davidson KR, Ramsey C, Shalit OM. The isomorphism problem for some universal operator algebras. Adv. Math. Vol 228, Issue 1, pp. 167-218 (2011).
\bibitem[DRS2]{DRS2} Davidson KR, Ramsey C, Shalit OM. Operator algebras for analytic varieties. Trans. Amer. Math. Soc. Vol 367, Issue 2, pp. 1121-1150 (2015).

\bibitem[Don1]{Don1} Donaldson SK. Scalar curvature and projective embeddings, I J. Differential Geom. Vol 59, pp. 479-522 (2001).
\bibitem[Don2]{Don2} Donaldson SK. Scalar curvature and projective embeddings, II. Quart. J. Math. Vol 56, pp. 345-356 (2005).
\bibitem[Don3]{Don3} Donaldson SK. Some numerical results in complex differential geometry. Pure Appl. Math. Q. Vol 5, Issue 2, pp. 571-618 (2009).
\bibitem[Don9]{Don9} Donaldson SK. Scalar curvature and stability of toric varieties. J. Differential Geom. Vol 62, pp. 289-349 (2002).
\bibitem[Don12]{Don12} Donaldson SK. Lower bounds on the Calabi functional. J. Differential Geom. Vol 70, pp. 453-472 (2005).

\bibitem[Doug4]{Doug4} Douglas RG. Essentially reductive Hilbert modules. J. Operator Theory. Vol 55, Issue 1, pp. 117-133 (2006).

\bibitem[DoWa2]{DoWa2} Douglas RG, Wang K. Geometric Arveson--Douglas conjecture and holomorphic extension. arXiv:1511.00782v1 (2015).
%\bibitem[DoWa3]{DoWa3} Douglas RG, Wang Y. Geometric Arveson--Douglas conjecture -- Decomposition of varieties. Arxiv v1 (2017).
\bibitem[DoWa4]{DoWa4} Douglas RG, Wang K. Essential normality of cyclic submodule generated by any polynomial. J. Funct. Anal. Vol 261, pp. 3155-3180 (2011).

\bibitem[East1]{East1} Eastwood M. The Cartan product. Bull. Belg. Math. Soc. Simon Stevin. Vol. 11, Issue 5, pp. 641-651 (2005).

\bibitem[EnEs1]{EnEs1} Engliš M, Eschmeier J. Geometric Arveson--Douglas conjecture. Adv. Math. Vol 274, pp. 606-630 (2015).

\bibitem[Esch1]{Esch1} Eschmeier J. Essential normality of homogeneous submodules. Integr. Equ. Oper. Theory. Vol 69, Issue 2, pp. 171-182 (2011). 

\bibitem[Ev1]{Ev1} Evans DE. On $\Oi_n$. Publ. RIMS, Kyoto Univ. Vol 16, pp. 915-927 (1980).

\bibitem[Frank1]{Frank1} Frank M. Isomorphisms of Hilbert $C^*$-modules and*-isomorphisms of related operator $C^*$-algebras. Math. Scand. Vol 80, pp. 313-319 (1997).

\bibitem[Hawk1]{Hawk1} Hawkins E. Quantization of equivariant vector bundles. Comm. Math. Phys. Vol 202, Issue 3, pp. 517-546 (1999).

\bibitem[Hawk2]{Hawk2} Hawkins E. Geometric quantization of vector bundles and the correspondence with deformation quantization. Comm. Math. Phys. Vol 215, Issue 2, pp. 409-432 (2000).

\bibitem[Hawk3]{Hawk3} Hawkins E. An obstruction to quantization of the sphere. Comm. Math. Phys. Vol 283, Issue 3, pp. 675-699 (2008).

\bibitem[Huy]{Huy} Huybrechts D. Complex geometry. An introduction. Springer (2005). 

\bibitem[Iz1]{Iz1} Izumi M. Non-commutative Poisson boundaries and compact quantum group actions. Adv. Math. Vol 169, Issue 1, pp. 1-57 (2002).

\bibitem[Iz2]{Iz2} Izumi M. Non-commutative Markov operators arising from subfactors. Operator algebras and applications. Adv. Stud. Pure Math. Vol 38, pp. 201-217 (2004).

\bibitem[Iz3]{Iz3} Izumi M. $E_0$-semigroups: around and beyond Arveson’s work. J. Operator Theory. Vol 68, Issue 2, pp.
335-363  (2012).

\bibitem[Iz4]{Iz4} Izumi M. Non-commutative Poisson boundaries, Discrete geometric analysis. Contemp. Math. Vol 347, pp. 69-81 (2004).

\bibitem[INT1]{INT1} Izumi M, Neshveyev S, Tuset L. Poisson boundary of the dual of $\SU_q(n)$. Comm. Math. Phys. Vol 262, Issue 2, pp. 505-531 (2006).

\bibitem[KaSh1]{KaSh1} Kakariadis E, Shalit O. On operator algebras associated with monomial ideals in noncommuting variables.  arXiv:1501.06495 (2015).

\bibitem[KaSc1]{KaSc1} Karabegov AV, Schlichenmaier M. Identification of Berezin--Toeplitz deformation quantization. J. Reine Angew. Math. Issue 540, p.49-76 (2001).

\bibitem[Kenn1]{Kenn1} Kennedy M. Essential normality and the decomposability of homogeneous submodules. Trans. Amer. Math. Socl. Vol 367, Issue 1, pp. 293-311 (2015).

\bibitem[KeSh1]{KeSh1} Kennedy M, Shalit OM. Essential normality and the decomposability of algebraic varieties. New York J. Math. Vol 18, pp. 877-890 (2012).

\bibitem[KeSh2]{KeSh2} Kennedy M, Shalit OM. Essential normality, essential norms and hyperrigidity. New York J. Math. Vol 18, pp. 877-890 (2012).

\bibitem[KlS]{KlS} Klimyk AU, Schmüdgen K. Quantum groups and their representations. Vol. 552, Springer, Berlin (1997).

\bibitem[Lan1]{Lan1} Landsman NP.  Mathematical topics between classical and quantum mechanics. Springer (1998).

\bibitem[Lan2]{Lan2} Landsman NP.  Strict quantization of coadjoint orbits. J. Math. Phys. Vol 39, Issue 12, pp. 6372-6383 (1998).

\bibitem[LMS1]{LMS1} Lazaroiu CI, McNamee D, Sämann C. Generalized Berezin quantization, Bergman metrics and fuzzy Laplacians. J. High Energy Phys. Issue 9, pp. 1-59 (2008).

\bibitem[LuTe1]{LuTe1} Lübke M, Teleman A. The Kobayashi–Hitchin correspondence. World Scienlific Publishing
(1995).

\bibitem[MaVD]{MaVD} Maes A, Van Daele A. Notes on compact quantum groups. arXiv: math/9803122 (1998). 

\bibitem[McTr1]{McTr1} McCullough S, Trent TT. Invariant subspaces and Nevanlinna–Pick kernels. J. Funct. Anal. Vol 178, Issue 1, pp. 226-249 (2000).

\bibitem[NT1]{NT1} Neshveyev S, Tuset L. The Martin boundary of a discrete quantum group, J. Reine Angew. Math. Vol 568, pp. 23-70 (2004).

\bibitem[Per1]{Per1} Perelomov AM. Generalized coherent states and applications. Springer (1986).

\bibitem[Pims1]{Pims1} Pimsner MV. A class of $C^*$-algebras generalizing both Cuntz-Krieger algebras and crossed products by $\Z$. Fields Inst. Commun. Vol 12, pp. 189–212 (1997).

\bibitem[Pop1]{Pop1} Popescu G. Operator theory on noncommutative varieties. Indiana Univ. Math. J. Vol 56, Issue 2, pp. 389-442 (2006).
\bibitem[Pop2]{Pop2} Popescu G. Operator theory on noncommutative varieties II. Proc. Amer. Math. Soc. Vol 135, Issue 7, pp. 2151-2164 (2007).
%\bibitem[Pop7]{Pop7} Popescu G. Similarity and ergodic theory of positive linear maps. J. Reine Angew. Math. Vol 561, pp. 87-129 (2003).

\bibitem[Raj1]{Raj1} Rajeev SG. New classical limits of quantum theories. In: Infinite Dimensional Groups and Manifolds. Vol 5, 213 (2004).

\bibitem[RCG1]{RCG1} Rawnsley J, Cahen M, Gutt S. Quantization of Kähler manifolds I: geometric interpretation of Berezin's quantization. J. Geom. Phys. Vol 7, Issue 1, pp. 45-62 (1990). 

\bibitem[RCG2]{RCG2} Rawnsley J, Cahen M, Gutt S. Quantization of Kähler manifolds II. Trans. Amer. Math. Soc. Vol 337, Issue 1 (1993). 

\bibitem[Rie1]{Rie1} Rieffel M. Deformation quantization for actions of $\R^d$. Mem. Amer. Math. Soc. Vol 506 (1993).

\bibitem[Rie2]{Rie2} Rieffel M. Matrix algebras converge to the sphere for quantum Gromov--Hausdorff distance. Mem. Amer. Math. Soc. Vol 168, Issue 796, pp. 67-91 (2004).

\bibitem[RoSt1]{RoSt1} Rørdam M, Størmer E (Eds.). Classification of nuclear $C^*$-algebras. Entropy in operator algebras. Vol 7. Springer (2002).

\bibitem[Sain]{Sain} Sain J. Berezin quantization from ergodic actions of compact quantum groups, and quantum Gromov--Hausdorff distance. PhD thesis, University of California (2009).

\bibitem[Schl1]{Schl1} Schlichenmaier M. Berezin-Toeplitz quantization for compact Kähler manifolds. A review of results. Adv. Math. Phys. Article ID 927280, 38 pages (2010).

\bibitem[Schl2]{Schl2} Schlichenmaier M. Singular projective varieties and quantization. In: Quantization of Singular Symplectic Quotients. pp. 259-282 (2001).

\bibitem[Seg1]{Seg1} Segal G. Lectures on Lie groups. In: Lectures on Lie groups and Lie algebras. Cambridge University Press (1995).

\bibitem[Serre1]{Serre1} Serre JP. Géométrie algébrique et géométrie analytique. Ann. Inst. Fourier, Grenoble. Vol 6, pp. 1-42 (1956).
\bibitem[Serre2]{Serre2} Serre JP. Faisceaux algébriques cohérents. Ann. Math. Vol 61, Issue 2, pp. 197-278 (1955).

%\bibitem[Sha1]{Sha1} Shalit OM. Product systems, subproduct systems and dilation theory of completely positive semigroups. PhD thesis. arXiv:1002.4920v1 (2010).
%\bibitem[Sha2]{Sha2} Shalit OM. Operator theory and function theory in Drury--Arveson space and its quotients. arXiv:1308.1081v5 (2013).
\bibitem[Sha3]{Sha3} Shalit OM. Stable polynomial division and essential normality of graded Hilbert modules. J. London Math. Soc. Vol 83, Issue 2, pp. 273-289 (2011).

\bibitem[ShSo1]{ShSo1} Shalit OM, Solel B. Subproduct systems. Doc. Math. Vol 14, pp. 801-868 (2009).

\bibitem[Sten1]{Sten1}  Stenstrom B. Rings of quotients -- An introduction to methods of ring theory. Springer (1975). 

\bibitem[Tian1]{Tian1} Tian G. Kähler--Einstein metrics with positive scalar curvature. Invent. Math. Vol 137, pp. 1-37 (1995).

\bibitem[Tian2]{Tian2} Tian G. Canonical metrics in Kähler geometry. Birkhauser (2000).

\bibitem[Timm1]{Timm1} Timmermann T. An invitation to quantum groups and duality: from Hopf algebras to multiplicative unitaries and beyond. Eur. Math. Soc. (2008). 

\bibitem[Tom1]{Tom1} Tomatsu R. A characterization of right coideals of quotient type and its application to classification
of Poisson Boundaries. Comm. Math. Phys. Vol 275, Issue 1, pp 271-296 (2007).

\bibitem[UnUp1]{UnUp1} Unterberger A, Upmeier H. The Berezin transform and invariant differential operators. Comm. Math. Phys. Vol 164, pp. 563-597 (1994).

\bibitem[VaVe1]{VaVe1} Vaes S, Vennet N. Identification of the Poisson and Martin boundaries of orthogonal discrete quantum groups. J. Inst. Math. Jussieu. Issue 7, pp. 391-412 (2008).

\bibitem[VVer1]{VVer1} Vaes S, Vergnioux R. The boundary of universal discrete quantum groups, exactness and factoriality. Duke Math. J. Vol 140, Issue 1, pp. 35-84 (2006).

\bibitem[VaDW]{VaDW} Van Daele A, Wang SZ. Universal quantum groups. Int. J. Math. Vol 7, Issue 2, pp. 255-264 (1996).

\bibitem[Vas1]{Vas1} Vasselli E. Continuous fields of $C^*$-algebras arising from extensions of tensor $C^*$-categories.
J. Funct. Anal. Vol 199, pp. 122-152 (2003).

\bibitem[Vas3]{Vas3} Vasselli E.  The $C^*$-algebra of a vector bundle and fields of Cuntz algebras. J. Funct. Anal.
222, Issue 2, pp. 491-502 (2005).

\bibitem[Vis2]{Vis2} Viselter A. Cuntz--Pimsner algebras for subproduct systems. Int. J. Math. Vol 23, Issue 8, p. 1250081 (2012).

\bibitem[Wang1]{Wang1}  Wang SZ. Ergodic actions of universal quantum groups on operator algebras. Comm. Math. Phys. Vol 203, pp. 481-498 (1999).

\bibitem[Wang3]{Wang3} Wang SZ. Structure and isomorphism classification of compact quantum groups $A_u (q)$ and $B_u (q)$. arXiv:math/9807095v2 (2000).

\bibitem[Wa1]{Wa1} Wang X. Balance point and stability of vector bundles over a projective manifold. Math. Res. Lett. Vol 9, Issues 2-3, pp. 393-411 (2002).

%\bibitem[Wa2]{Wa2} Wang X. Canonical metrics on stable vector bundles, Comm. Anal. Geom. Vol 13, Issue 2, pp. 253-285 (2005).

\bibitem[Wor1]{Wor1} Woronowicz SL. Compact matrix pseudogroups. Comm. Math. Phys. Vol 111, pp. 613-665 (1987).

\bibitem[Wor3]{Wor3} Woronowicz SL. Tannaka--Krein duality for compact matrix pseudogroups. Twisted $\SU(n)$ groups, Inv. Math. Vol 93, pp. 35-76 (1988).

\bibitem[Wor4]{Wor4} Woronowicz SL. Compact quantum groups. Symétries quantiques, Les Houches 1995, pp. 845-884 (1998).

\bibitem[Yau1]{Yau1} Yau S-T. On the Ricci curvature of a compact Kähler manifold and the complex Monge--Amère equation, I. Commun. Pure Appl. Math. Vol 31, pp. 339-411 (1978).

\bibitem[Zeld1]{Zeld1} Zelditch S. Szegö kernels and a theorem of Tian. Int. Math. Res. Not. Issue 6, pp. 317-331 (1998).

\end{document}